\documentclass[11pt,reqno,tbtags]{amsart}
\usepackage{amsmath,amssymb,amsthm,accents,float,graphicx}
\usepackage{natbib}
\title{Extinction time for the weaker of two competing SIS epidemics}

{\author{Fabio Lopes}\thanks{FL was first a postdoctoral research assistant of ML funded by EPSRC Leadership Fellowship EP/J004022/2, and was later supported partly by Conicyt/Fondecyt Proyecto Postdoctorado N.3160163.}}
\address{Pontificia Universidad Cat\'olica de Valparaiso.}

{\author{Malwina Luczak}\thanks{ML was supported partly by EPSRC Leadership Fellowship EP/J004022/2 and partly by ARC Future Fellowship FT170100409.}}
\address{School of Mathematics and Statistics, University of Melbourne}


\renewcommand{\P}{\ensuremath{\mathbb P}}

\newcommand{\halmos}{\rule{1ex}{1.4ex}}
\newcommand{\proofbox}{\hspace*{\fill}\mbox{$\halmos$}}

\newcommand{\E}{\ensuremath{\mathbb E}}
\newcommand{\Z}{\mathbb{Z}}
\newcommand{\R}{\mathbb{R}}
\newcommand{\eps}{\varepsilon}
\newcommand{\tx}{\tilde{x}}

\theoremstyle{plain}

\newtheorem{theorem}{Theorem}
\newtheorem{lemma}{Lemma}
\newtheorem{remark}{Remark}

\begin{document}


\begin{abstract}
We consider a simple stochastic model for the spread of a disease caused by two virus strains in a closed homogeneously mixing population of size N. The spread of each strain in the absence of the other one is described by the stochastic logistic SIS epidemic process, and we assume that there is perfect cross-immunity between the two strains, that is, individuals infected by one are temporarily immune to re-infections and infections by the other. For the case where one strain has a strictly larger basic reproductive ratio than the other, and the stronger strain on its own is supercritical (that is, its basic reproductive ratio is larger than 1), we derive precise asymptotic results for the distribution of the time when the weaker strain disappears from the population, that is, its extinction time. We further extend our results to certain parameter values where the difference between the two reproductive ratios may tend to 0 as $N \to \infty$.

In proving our results, we illustrate a new approach to a fluid limit approximation for a sequence of Markov chains
in the vicinity of a stable fixed point of the limit. 


\end{abstract}

\maketitle

\section{Introduction}
\label{intro}

Mathematical models of epidemics provide important tools for understanding the spread of many diseases relevant to public health, and may help health authorities and organizations develop measures to prevent and manage epidemic outbreaks, as well as control the emergence of new infections.

Infectious disease control is constantly challenged by the diversity of pathogen populations and their continuous evolution in response to changing environments, technological advances (e.g.~air travel, antibiotics, see~Schrag and Wiener~1995), 
interactions with their hosts (e.g.~the structure of the network of contacts, see Leventhal et al.~2015), 
as well as interactions with other pathogens (e.g.~interference, see Gart and De Vries~1966). 

The vast majority of mathematical models of epidemics in the literature view infectious diseases as caused by a single and stable pathogen strain. Such models tend to be more tractable, and may even yield accurate descriptions of the short-term dynamics of certain diseases. However, they may be inappropriate for predicting the long-term evolutionary dynamics of pathogen populations, see Humplik et al.~(2014), 
or for analysing pathogen infections where host susceptibility may be altered due to infections by other pathogens. Examples include interference between co-morbid diseases (e.g.~yaws and chickenpox, see Gart and De Vries~1966), 
as well as scenarios where successive exposures to different strains of the pathogen may have important consequences for disease infectiousness and severity (e.g.~dengue fever, see Feng and Velasco-Hernandez~1997). 

Pathogen evolution and the capacity of pathogens to adapt to changes in their environment and hosts are both regarded as important factors for emergence of new diseases (Antia et al.~2003; Arinaminpathy and McLean~2009). 
For instance, it is known that pathogen strains which are sufficiently antigenically similar may induce a (partial) cross-protective immune response, so that hosts infected by one of the strains may acquire different degrees of temporary or permanent immunity to re-infections and infections by antigenically similar strains. Thus,
if a certain closed population of hosts is affected by a particular virus strain and a number of individuals infected by an antigenically similar strain are introduced, then the different pathogen strains may interact as if competing for susceptible individuals in the host population.

A long-standing principle in ecology known as the \textit{competitive exclusion principle} (Levin 1970) 
predicts that, when species sharing the same ecological niche compete for limited resources, then the one with even the slightest advantage will eventually outcompete the others and become dominant.
This form of competition is believed to be particularly important in the evolution of RNA viruses (Domingo et al.~1996; Moya et al.~2004).
For instance, 
it is shown by Bahl et al.~(2009) that viral gene flow from Eurasia had led to replacement of endemic avian influenza viruses in North America; moreover, the authors argue that the most likely mechanism for that was competition for susceptible hosts.

A number of deterministic models of coexistence and competitive exclusion in multiple strains pathogen populations have been studied,
describing the dynamics of important diseases such as gonorrhea~(Castillo-Chavez et al.~1999),
AIDS~(Anderson and May~1996),
and influenza~(Andreasen et al.~1997);
see also references within these works. However, deterministic models often fail to provide accurate descriptions of infectious disease dynamics at the beginning or at the end of an outbreak; this is partly because random fluctuations when the number of infectives is small can significantly affect the outcome, in particular, the chances of the disease persisting in the population, as well as the duration of the outbreak.


In this work, we consider a simple model for the spread of a disease with stochastic susceptible-infective-susceptible (SIS) dynamics caused by two different virus strains with a perfect cross-protective immune response, so individuals infected by one strain are temporarily immune to re-infections and infections by the other strain.
We focus on the case where one of the virus strains has some advantage over its competitor (a higher basic reproductive  ratio), and competitive exclusion occurs. Starting with positive but otherwise arbitrary proportions of infected individuals of each virus strain in a large host population, we track the long-term evolution of this process, so as to obtain the distribution of the time until competitive exclusion occurs, that is~the extinction time of the weaker virus strain.



The simplest stochastic model for a disease with SIS dynamics is the stochastic SIS logistic epidemic model.
In that model, each individual within the population is either susceptible or infective. We assume a population of size $N$, and let $\lambda>0$ denote the infection rate. Each infective individual encounters uniformly at random another member of the population at rate $\lambda$; if the encountered individual is susceptible, then he\slash she becomes infective. Also, each infective individual recovers at rate $\mu > 0$ and, once recovered, becomes susceptible again.

Let $Y_N(t)$ denote the number of infective individuals in the population at time $t$; then $(Y_N (t))_{t\geq 0}$ is a continuous-time Markov chain on $\{0,1,\ldots, N\}$ with transition rates from state $Y$ given by
\[\begin{array}{llll}
Y &\rightarrow& Y+1   \quad \quad \hbox{   at rate  }    &\lambda  Y (1- Y\slash N);\\
Y  &\rightarrow & Y-1    \quad \quad \hbox{ at rate  }   & \mu Y.
\end{array}
\]
 The extinction time $\tau_N$ is defined as $\tau_N=\inf \{ t\geq0: Y_N(t)=0\}$, and, since the state space is finite, $\tau_N$ is a.s. finite.
The following theorem summarises asymptotic results for the distribution of $\tau_N$ in the case where the initial epidemic infects a positive proportion of the population, see Andersson and Djehiche (1998), 
as well as Brightwell, House, and Luczak (2018). 

We recall that a random variable $W$ has a standard Gumbel distribution if $\P(W\leq w)= e^{-e^{-w}}$, for all $w\in \mathbb{R}$.


\begin{theorem} Let $\lambda, \mu, \alpha>0$, and suppose that $X_N(0)\slash N \rightarrow \alpha$ as $N\rightarrow \infty$.
\label{teomalwina}
\begin{itemize}
\item [(i)] (Supercritical case; Andersson and Djehiche (1998).) If $\lambda> \mu$, then $\tau_N\slash \E(\tau_N)\rightarrow Z$ in distribution, as $N\rightarrow \infty$, where $Z$ is an exponential random variable with parameter 1. Furthermore,
\[ \E(\tau_N) \sim \sqrt{\frac{2 \pi}{N}} \frac{\lambda}{(\lambda-\mu)^2} e^{N v},\]
as $N\rightarrow \infty$, where $v=\log (\lambda/\mu) -1 + \frac{\mu}{\lambda}$.
\item [(ii)] (Subcritical case; Brightwell, House and Luczak (2018).) If $\lambda< \mu$, then as $N \to \infty$
\[ (\mu-\lambda)\tau_N - \left\{\log\alpha + \log N + \log (1-\lambda/\mu)  - \log \left( 1 +\lambda \alpha/(\mu - \lambda) \right) \right\} \rightarrow W\]
in distribution, where $W$ is a standard Gumbel random variable.
\end{itemize}
\end{theorem}

When $\lambda = \mu$, then, for most starting states, the time to extinction is of the order $N^{1/2}$, see Nasell (2011). 
Brightwell, House and Luczak~(2018) 
also consider more general initial conditions, as well as determine the extinction time when $\lambda = \lambda (N)$, $\mu = \mu (N)$ satisfy $\mu - \lambda \to 0$ and $(\mu - \lambda)N^{1/2} \to \infty$ (the barely subcritical case).

\smallskip

The stochastic SIS logistic competition model
describes the spread of a disease in a homogeneously mixing population via two different virus strains, say types 1 and 2, which are sufficiently antigenically similar to induce a cross-protective immune response.
An individual infected with strain $i$ ($i=1,2$) stays infected for an exponentially distributed time with rate $\mu_i > 0$, and,  during the infectious period, they independently make an infectious contact to a random individual according to a Poisson process with rate $\lambda_i > 0$; if the individual is currently susceptible, then they become infected with strain $i$ as a result.
The dynamics can thus be described as a two-dimensional
continuous-time Markov chain $(X_N(t))_{t\geq 0}=(X_{N,1}(t), X_{N,2}(t))_{t \geq 0}$ , where $X_{N,1}(t)$ and $X_{N,2}(t)$ denote the numbers of individuals infected with strains of type $1$ and $2$ respectively, at time $t$. 
The state space is $S_N=\{(X_1,X_2)^T: X_1,X_2 \in \Z^+, 0 \leq X_1+X_2 \leq N \}$, and the
transition rates from state $(X_1,X_2)$ can be written as follows:
\[
\begin{array}{lllll}
(X_1,X_2)& \rightarrow& (X_1+1,X_2)   &\hbox{   at rate  }&     \lambda_1  X_1 (1- X_1\slash N - X_2 \slash N);\\
(X_1,X_2) & \rightarrow & (X_1,X_2+1)  &  \hbox{ at rate  }&   \lambda_2  X_2 (1- X_1\slash N - X_2 \slash N);\\
(X_1,X_2)  &\rightarrow & (X_1-1,X_2)   & \hbox{ at rate  }&      \mu_1 X_1;\\
(X_1,X_2)  &\rightarrow & (X_1,X_2-1)   &\hbox{ at rate  } &     \mu_2 X_2.
\end{array}
\]
We note that, in the absence of one of the strains, 
the other strain evolves according to the basic stochastic SIS logistic epidemic model described above.

\smallskip

We assume that $\lambda_1/\mu_1> \lambda_2/\mu_2$, and $\lambda_1/\mu_1 > 1$, as well as that  $X_{N,1}(0)\slash N\rightarrow \alpha$ and  $X_{N,2}(0) \slash N\rightarrow \beta$ as $N\rightarrow \infty$ ($0<\alpha, \beta$, $\alpha + \beta \le 1$).
The assumption $\lambda_1/\mu_1 > \lambda_2/\mu_2$  means that strain 1 has a higher basic reproductive ratio, and is thus more infectious than strain 2.
Since $\lambda_1/\mu_1 > 1$,
Theorem \ref{teomalwina} implies that the stronger subtype, in the absence of its competitor, would stay endemic in the population for a time that grows exponentially in the size $N$ of the population.

\smallskip

This model was proposed by Parsons and Quince (2007a,b) 
as an extension to the Moran model for a haploid population studied in Moran (1958). They assume that both alleles (strains) are supercritical, which in our setting translates to assuming that $\lambda_2/\mu_2 > 1$. Parsons and Quince (2007a) consider the case where one of alleles (strains) is weaker (the case considered in the present paper), while Parsons and Quince (2007b) consider the case where both types of allele (strain) have equal fitness, translating to the case $\lambda_1/\mu_1 = \lambda_2/\mu_2$, not studied here.
Parsons and Quince (2007a,b) study the
allele fixation probability for this model, equivalent to the probability of one virus strain displacing the other in the competing epidemic setting. 
Also, Humplik et al. (2014)~
study the effects of virulence on the probability of strain 2 invading strain 1, mainly for small populations, though they appear to be unaware of the earlier works of Parsons and Quince (2007a,b).

For a closely related model with $K \ge 2$ types, Parsons, Quince and Plotkin (2008) 
obtain analytic approximations for the expected fixation time (i.e.~the time until  competitive exclusion occurs), which turns out to be linear in the population size when all the alleles (strains) have the same (supercritical) basic reproductive ratio. These authors further argue that a similar result should hold for the model considered in our paper, and Kogan et al. (2014) have shown this is indeed the case for $K=2$ strains of equal strength.

Theorem \ref{teomy} below concerns the case where there is a dominant, supercritical, strain and each of the two strains initially affects a positive fraction of the population. Under these conditions, competitive exclusion of the weaker strain by the stronger occurs with high probability (i.e. with probability tending to $1$ as the population size $N \to \infty$). Our result shows that, with high probability,
the extinction time for the weaker type scales logarithmically while the time to extinction for the dominant strain scales exponentially with the population size.

The corresponding deterministic SIS logistic competition model is among the simplest epidemic models for infections caused by multiple pathogen strains. 
It is represented by the pair
\begin{eqnarray}
\frac{dx_1(t)}{dt} & = & \lambda_1 x_1(t) (1-x_1 (t) -x_2(t)) - \mu_1 x_1 (t) \nonumber \\
\frac{dx_2(t)}{dt} & = & \lambda_2 x_2 (t) (1-x_1(t) - x_2(t)) - \mu_2 x_2 (t) \label{eq.det-comp}
\end{eqnarray}
of differential equations, and is thus a particular instance of the deterministic Lotka-Volterra system -- see Lotka (1925), Volterra (1931), Zeeman (1995) and Chapter 8 of Renshaw (2011) -- which has found applications, for instance, in biology, ecology, and economics.
%
Various stochastic Lotka-Volterra systems have also been studied.
For instance, Kirupaharan and Allen (2004) study a stochastic Lotka-Volterra system with multiple species and demography (i.\ e.\ births and deaths), with a focus on the probability distributions of the numbers of individuals of each species conditioned on non-extinction. They provide numerical examples of competitive exclusion as well as coexistence, and compare the behaviour of the stochastic model to its deterministic version.
Also, Cattiaux and M\'el\'eard (2010) consider a stochastic Lotka-Volterra process on $\mathbb{R}^2_+$, modeling interactions in a  two-type density dependent population as a generalisation of the one-dimensional logistic Feller diffusion. They study the long-term behaviour of the process, proving existence and uniqueness of the quasi-stationary distribution in different regions of the parameter space. For the parameter region in which the two types compete, they show that there is a timescale on which only one type survives,
though they do not consider the distribution of the time until competitive exclusion occurs.





\smallskip

Let $R_{0,1} = \lambda_1/\mu_1$ and $R_{0,2} = \lambda_2/\mu_2$ denote the basic reproductive ratios of the two strains.
$\kappa_N=\inf \{ t\geq0: X_{N,2}(t)=0\}$, the time when the weaker competing species goes extinct.
Let also $\tau_N=\inf \{ t\geq0: X_{N,1}(t) =0\}$, the time the stronger species becomes extinct.
We now state our main result, concerning the distribution of $\kappa_N, \tau_N$. 

\begin{theorem}
\label{teomy}
%
Suppose that $R_{0,1}> R_{0,2}$ and that $R_{0,1} > 1$. Suppose further that $X_{N,1}(0)\slash N\rightarrow \alpha$ and  $X_{N,2}(0) \slash N\rightarrow \beta$ as $N\rightarrow \infty$, where $\alpha,\beta>0$ and $\alpha+\beta\leq1$. Then, as $N \to \infty$,
{\small
\begin{eqnarray*}
\mu_2 \Big (1-\frac{R_{0,2}}{R_{0,1}} \Big )\kappa_N - 
 \Big [\log \Big ( N \beta \Big (1 - \frac{R_{0,2}}{R_{0,1}}\Big ) \Big ) + \frac{R_{0,2}\mu_2}{R_{0,1}\mu_1} \log \Big ( \frac{1-R_{0,1}^{-1}}{\alpha} \Big ) \Big ] \to W,
\end{eqnarray*}
}
where 
$W$ has a standard Gumbel distribution.

Furthermore, as $N \to \infty$,
\[ \E(\tau_N) \sim \sqrt{\frac{2 \pi}{N}} \frac{\lambda_1}{(\lambda_1-\mu_1)^2} e^{N v_1},\]
as $N\rightarrow \infty$, where $v_1=\log (\lambda_1/\mu_1) -1 + \frac{\mu_1}{\lambda_1}$,
and $\tau_N/\E(\tau_N)\rightarrow Z$ in distribution, where $Z$ is an exponential random variable with parameter 1.
\end{theorem}

Theorem~\ref{teomy} thus shows that the extinction time $\kappa_N$ of the weaker strain can be written as
$$\kappa_N  = \frac{\log N + Z}{\mu_2 \Big (1-\frac{R_{0,2}}{R_{0,1}} \Big )},$$
where $Z$ is a random variable with a bounded mean and variance, while the extinction time $\tau_N$ of the stronger strain asymptotically has the same distribution as if the weaker strain was absent to begin with.

The long-term behaviour of Markov population processes is of considerable importance in applications. In epidemic models, long-term phenomena include extinction of certain pathogen strains, or replacement of a dominant pathogen strain in the host population by another more adapted pathogen strain introduced into the host population e.g.\ due to mutation or migration. Mathematically, these phenomena are related to the behaviour of the scaled process near fixed points of its approximating differential equation, including absorbing boundaries for one or more coordinates.
Recently, Barbour, Hamza, Kaspi and Klebaner (2015) have shown that, under appropriate conditions, a density dependent Markov population process that starts near an absorbing boundary and manages to escape from it, still can be well-approximated by the deterministic solution as described by the standard theory but with a random time shift, and that the time to escape from such a boundary is random and of order $O(\log N)$, see Theorem 1.1 in Barbour et al.~(2015).
Also, similar to the phenomenon we investigate in the present work, they describe in a very general setting the behaviour of a class of population processes near a fixed point at which one or more coordinates of the process have value 0, i.e.\ they are extinct, and derive the limit distribution for the extinction times for such processes as a standard Gumbel random variable, after scaling and centering, see Theorem 1.2 in Barbour et al.~(2015).
In both results, the randomness when the process is escaping or reaching an absorbing boundary is captured by a branching process approximation. However, at their level of generality, the formulae they obtain contain non-explicit constants, and their bounds on the rate of convergence are too weak to investigate near-critical phenomena. Also, rigorous justification of such a general approximation, based on an abstract coupling of Thorisson (see Theorem 7.3 in Thorisson 2000) is quite involved.


 In the present work, we develop a related but more direct and optimal approximation to prove an explicit formula for the extinction time of the weaker virus strain in the stochastic logistic SIS competition model. Like the approach of Barbour et al.~(2015), our approach is based on decomposing the drifts of the process into linear and non-linear parts, and using a variation of constants formula. However, we additionally take full advantage of the fact that the non-linear parts are small in the neighbourhood of a fixed point, and provide more refined bounds on the deviations of the martingale transform appearing in the equations.
 Similar ideas were also used in a different context by Barbour and Luczak (2012) 
 and, in discrete time, by Brightwell and Luczak (2012).

Unlike the approach of Barbour et al.~(2015), the precision of our approximation facilitates study of near-critical phenomena. In Section~\ref{sec:critical}, we
allow the basic reproductive ratios $R_{0,1}$ and $R_{0,2}$ 
to be functions of the population size $N$. We
show that Theorem~\ref{teomy} below can be extended to certain near-critical regimes where $\lambda_1/\mu_1 - \lambda_2/\mu_2 \to 0$, while $\lambda_1/\mu_1$ may or may not tend to $1$ as $N \to \infty$. We do not cover the entire spectrum of near-critical behaviours: the example considered here is meant as a proof of concept, and a full investigation will be carried out systematically in future work. One challenge of such an investigation will be to understand the behaviour of the approximating deterministic process in various near-critical regimes. Also in future work we intend to study the critical case when the strengths of the two strains are even closer to identical,
and to extend our results to competition of more than $2$ strains. A further project is to rigorously study the probability that the strongest strain wins when starting with only a small number of infected individuals relative to the number of infectives with weaker strains, in particular in near-critical scenarios, where there is likely to be a delicate interplay between initial conditions and the asymptotic differences between the strengths of the different strains.

In Section~\ref{auxiliar}, we present some preliminaries concerning the stability of fixed points
of the deterministic logistic SIS competition model.
Furthermore, we give an overview of the strategy used to prove Theorem \ref{teomy}. The idea is that the stochastic logistic SIS competition process follows closely the corresponding deterministic process for a long time, until the latter one is close to its attractive fixed point at $((\lambda_1-\mu_1)/\lambda_1,0)^T$. From there on, the time to extinction for the second species is short, and well approximated by a linear birth-and-death chain, with the randomness captured by the Gumbel distribution.
We break up the analysis of the process into phases, similarly to the approach of Brightwell, House and Luczak~(2018) used to prove a general version of Theorem \ref{teomalwina} $(ii)$. We analyze each of these phases in the subsequent sections. In Section~\ref{sec:proof}, we combine the results obtained in the preceding sections to prove Theorem \ref{teomy}.

In Section~\ref{sec:critical}, we state and prove Theorem~\ref{thm.extinction-nc}, which extends our work to a near-critical case where 
$R_{0,1}= R_{0,1}(N)$ and $R_{0,2} = R_{0,2} (N)$ 
are such that $R_{0,1}-R_{0,2} \to 0$, while $R_{0,1}-1$ may or may not tend to $0$ but satisfies $(R_{0,1}-R_{0,2})(R_{0,1}-1)^{-1} \to 0$ as $N \to \infty$. In the particular case we consider, $\mu_1 = \mu_2 = 1$, and $\lambda_1 = \lambda_1 (N)$, $\lambda_2 = \lambda_2 (N)$ are such that $(\lambda_1 - \lambda_2) (\lambda_1-1)^{-1} \to 0$. We further assume that
\begin{equation}
\label{cond-crit-win}
N (\lambda_1-\lambda_2)^3 (\lambda_1-1)^{-1}/\log \log \Big ( N (\lambda_1-\lambda_2)^2 \Big ) \to \infty.
\end{equation}
Condition~\ref{cond-crit-win} is an artefact of our proof technique, and does not define the transition to criticality.
We do believe that, similarly to a single stochastic SIS logistic epidemic, the condition $N(\lambda_1 - \lambda_2)^2 \to \infty$ is necessary and sufficient for the formula to hold. It seems feasible to refine our proof technique by splitting the differential equation approximation phase into subphases, possibly to a great enough extent so as to relax our assumption on separation from criticality to the best possible ; however, in the interest of greater clarity, we do not explore such improvements in the present paper. A more detailed discussion of this as well as of what happens when the condition $N(\lambda_1 - \lambda_2)^2 \to \infty$ is not satisfied is included in Subsection~\ref{sub.discussion}. 

Throughout our proofs, we treat $X_N(t)$ and $x(t)$ as column vectors.

\section{Preliminaries}
\label{auxiliar}

In this section, we discuss the deterministic Lotka-Volterra system.
For suitable choices of parameter values, this model becomes the deterministic logistic SIS competition model, and approximates the stochastic logistic SIS competition model over certain timescales.

We further outline the proof of our main result, Theorem~\ref{teomy}.

\subsection{A deterministic version of the competition model}

\label{sec.det-comp}

The deterministic competitive Lotka-Volterra system represents a community of $k$ mutually competing species described by equations
\begin{equation}\label{lotta}
\frac{dx_i(t)}{dt} = x_i(t) \left(b_i - \sum_{j=1}^k a_{ij} x_{j}(t)\right), \quad \hbox{  } i=1,\ldots,k,
\end{equation}
where $x_i(t)$ denotes the population size of the $i$-th species at time $t$. It is assumed that
$b_i > 0$ for all $i$, and $a_{ij}> 0$ for all $i,j$.
For each $i=1, \ldots, k$, species $i$ would by itself, in the absence of all the other species, exhibit logistic growth, that is, its behaviour would be described by the equation
\[ \frac{d{x}_{i}(t)}{dt} = x_{i}(t) \left(b_i - a_{ii}x_{i}(t)\right), \hspace{0.3cm} b_i,a_{ii}>0.
 \]
This equation has two fixed points: $0$ and $b_i\slash a_{ii}$, the latter being the \textit{carrying capacity} of species $i$.\\
The following result of Zeeman (1995) gives simple algebraic criteria on the parameters $b_i$ and $a_{ij}$ of (\ref{lotta}) which guarantee that, for all strictly positive initial conditions of (\ref{lotta}), all but one of the species is driven to extinction, while the one remaining species in the community stabilizes at its own carrying capacity.

We recall that a fixed point $x^*$ of a system of ordinary differential equations is  {\it globally attractive} on a
set $U$ if and only if its basin of attraction is equal to $U$. In other words, $x^*$ is globally attractive if every solution to the system with initial condition in $U$ converges to $x^*$ as $t\rightarrow \infty$. 

\bigskip

\begin{theorem} (Zeeman 1995)
\label{thm.zeeman}
Suppose that system (\ref{lotta}) satisfies the inequalities
\begin{eqnarray} \label{lotta2}
\frac{b_j}{a_{jj}} & <& \frac{b_i}{a_{ij}} \quad \forall i<j \nonumber \\ 
\frac{b_j}{a_{jj}} & > & \frac{b_i}{a_{ij}} \quad \forall i>j. \label{lotta2a}
 \end{eqnarray}
Then fixed point $\left( \frac{b_1}{a_{11}},0, \ldots, 0 \right)^T$ is globally attractive on the interior of $\mathbb{R}_{+}^k$.
\end{theorem}

Clearly, if conditions~(\ref{lotta2a}) are satisfied, and $x_i (0) = 0$ for some $i > 1$, then the solution $x(t)$ still converges to $\left( \frac{b_1}{a_{11}},0, \ldots, 0 \right)^T$ as $t \to \infty$, as this case amounts to eliminating species $i$ from the equations.



In the case $\lambda_1/\mu_1 > \lambda_2/\mu_2 > 1$, the stochastic competition model can be naturally associated with a particular two-dimensional instance of (\ref{lotta}) with $b_i=\lambda_i-\mu_i$ for $i=1,2$, and $a_{ij}=\lambda_i$ for $i,j=1,2$, which gives the system~(\ref{eq.det-comp})
with initial conditions $x(0)=(x_1(0),x_2(0))^T\in \mathbb{R}^2_{+}$. In epidemic modelling,  $x_1(t)$ and $x_2(t)$ represent the proportions of individuals who at time $t$ are infected by strains 1 and 2 respectively, in a closed homogeneously mixing population.
By Theorem \ref{thm.zeeman}, all solutions with $x_1(0) > 0$ converge to $\left(\frac{\lambda_1-\mu_1}{\lambda_1},0 \right)^T$ as $t\rightarrow \infty$.

We claim this still holds even when $\lambda_2/\mu_2 \le 1$. First of all, observe that $x_2(t)^{\mu_1}/x_1(t)^{\mu_2}$ is decreasing in $t$, and hence one can see that $d x_1(t)/dt \ge 0$ if $x_1 (t) \le \eps$ for $\eps > 0$ small enough. (We can choose $\eps$ such that $\eps + (\eps/x_1(0))^{\mu_2/\mu_1}  \le 1-\mu_1/\lambda_1$.) Now consider the Lyapunov function
$$
\phi(x_1,x_2) = \frac{1}{2} \left( x_1 + x_2 - 1 + \frac{\mu_1}{\lambda_1} \right)^2 + x_2 \left(\frac{\mu_2}{\lambda_2} - \frac{\mu_1}{\lambda_1}\right).
$$
This function is non-negative, and is zero only at the fixed point $(x_1,x_2) = \left(1- \frac{\mu_1}{\lambda_1}, 0 \right)$. The derivative is given by
\begin{eqnarray*}
\lefteqn{\frac{d}{dt} \phi(x_1(t),x_2(t))} \\
&=& \left(x_1 + x_2 - 1 +\frac{\mu_1}{\lambda_1} \right) \left ( \frac{d x_1}{dt} + \frac{dx_2}{dt} \right )
+ \left(\frac{\mu_2}{\lambda_2} - \frac{\mu_1}{\lambda_1}\right) \frac{d x_2}{dt} \\
&=& - \lambda_1 x_1 \left(x_1 + x_2 - 1 +\frac{\mu_1}{\lambda_1} \right)^2 +  \left( x_1 + x_2 - 1 +\frac{\mu_2}{\lambda_2} \right) \frac{d x_2}{dt}\\
&=& - \lambda_1 x_1 \left(x_1 + x_2 - 1 +\frac{\mu_1}{\lambda_1} \right)^2 - \lambda_2 x_2 \left(x_1 + x_2 - 1 +\frac{\mu_2}{\lambda_2} \right)^2,
\end{eqnarray*}
and is non-positive everywhere; furthermore, for any $\eps > 0$, it is zero in $\{(x_1,x_2)^T: x_1 \ge \eps, x_2 \ge 0, x_1 + x_2 \le 1 \}$ only at the fixed point
$\left(1- \frac{\mu_1}{\lambda_1}, 0 \right)^T$. Since $dx_1(t)/dt \ge 0$ if $x_1(t) \le \eps$, the set $\{(x_1,x_2)^T: x_1 \ge \eps, x_2 \ge 0, x_1 + x_2 \le 1 \}$ is invariant for the deterministic logistic SIS competition model. As $\eps$ can be taken arbitrarily small, the claim follows.

\smallskip

It is easy to check that each solution $x(t) = (x_1(t), x_2(t))$ to~(\ref{eq.det-comp}) must satisfy
\begin{equation}
\label{relation}
\frac{{(x_1(t))}^{\lambda_2}}{{(x_2(t))}^{\lambda_1}} =   \frac{{(x_1(0))}^{\lambda_2}}{{(x_2(0))}^{\lambda_1}} e^{(\mu_2\lambda_1-\mu_1\lambda_2 )t}, \quad \quad \mbox{ for all }t \ge 0.
\end{equation}

Given a solution $(x(t))_{t\geq0}$ to (\ref{eq.det-comp}), this relation can be used to calculate the time $t_{a\rightarrow b}$ spent by $(x(t))_{t\geq0}$ to travel from a point $a=(a_1,a_2)^T$ to another point $b=(b_1,b_2)^T$:
\begin{equation}
\label{time}
t_{a\rightarrow b} = \frac{\lambda_2}{\mu_2 \lambda_1-\mu_1 \lambda_2} \log \left( b_1 \slash a_1  \right)  - \frac{\lambda_1}{\mu_2 \lambda_1-\mu_1 \lambda_2} \log \left( b_2 \slash a_2  \right).
\end{equation}


The Jacobian of $(\ref{eq.det-comp})$ at $(\frac{\lambda_1-\mu_1}{\lambda_1},0)^T$ is given by
\begin{eqnarray}
\label{eq-jacobian}
A = \begin{pmatrix}
-(\lambda_1-\mu_1)  & -(\lambda_1-\mu_1)\\
0 & -(\mu_2 - \lambda_2 \mu_1/\lambda_1),
\end{pmatrix}
\end{eqnarray}
and thus has eigenvalues $-(\lambda_1 - \mu_1)$ and $-(\mu_2 - \lambda_2\mu_1/\lambda_1)$,
which are real and strictly negative under the assumptions that $\lambda_1/\mu_1 > 1$ and $\lambda_1/\mu_1 > \lambda_2/\mu_2$. By standard theory, the speed of convergence is determined by $-\min \{\lambda_1-\mu_1, \mu_2 - \lambda_2 \mu_1/\lambda_1\}$. Indeed,
by Chapter VII \S 29, Theorem VII in Walter (1998), for any
$0 < \sigma < \min \{\lambda_1-\mu_1, \mu_2 - \lambda_2 \mu_1/\lambda_1\}$,
there exist $\eta > 0$, $C>0$ such that, if
$\|(x_1(0),x_2(0))^T-\Big (\frac{\lambda_1-\mu_1}{\lambda_1},0 \Big )^T\|_2<\eta$, then
$$\|(x_1(t),x_2(t))^T-(\frac{\lambda_1-\mu_1}{\lambda_1},0)^T\|_2\leq C e^{ -\sigma t} \hbox{ for all }  t\geq0.$$
In Subsection~\ref{subs:conv-det} below, we will give a stronger bound, as well as a lower bound on the speed of convergence.


\subsection{Convergence to fixed point}

\label{subs:conv-det}

Let $\eta_1 = \lambda_1 - \mu_1$ and let $\eta_2 = \mu_2 - \lambda_2\mu_1/\lambda_1$. Note that $\eta_1, \eta_2 > 0$ and $-\eta_1, - \eta_2$ are the eigenvalues of $A$.
Let
\begin{equation}
\label{eq-def-a}
a = 1-\frac{\eta_2}{\eta_1},
\end{equation}
and assume that $a \not = 0$, that is $\eta_2 \not = \eta_1$,
When $a = 0$, then the matrix~(\ref{eq-jacobian}) has repeated eigenvalues. We will consider this case at the end of this subsection.

We introduce new co-ordinates
$\tx_1 (t) = x_1(t) - \frac{\lambda_1-\mu_1}{\lambda_1} + \frac{1}{a} x_2(t)$ and $\tx_2 (t) = x_2(t)$, and let $\tx (t) = (\tx_1(t), \tx_2 (t))^T$. In the new co-ordinates, the differential equation~$(\ref{eq.det-comp})$ is expressed as
{\small
\begin{eqnarray}
\frac{d \tx_1 (t)}{dt} & = & -\eta_1 \tx_1 (t) - \lambda_1 \tx_1 (t)^2 - \frac{\eta_2(\lambda_1-\lambda_2)}{\eta_1} \Big (\frac{\tx_2 (t)}{a} \Big )^2 + \Big ( \lambda_1-\lambda_2+ \frac{\lambda_1\eta_2}{\eta_1} \Big ) \tx_1 (t) \frac{\tx_2 (t)}{a} \nonumber \\
\frac{d \tx_2 (t)}{dt} & = & - \eta_2 \tx_2 (t) - \lambda_2 \tx_2 (t) \tx_1 (t) + \frac{\lambda_2}{a} \frac{\eta_2}{\eta_1} \tx_2 (t)^2. \label{eq-diff-eq-eigen}
\end{eqnarray}
}

Note the diagonal form of the linear terms in the equation, reflecting the fact that $(1,1/a)$ and $(0,1)$ are the left eigenvectors of the matrix $A$, with eigenvalues $-\eta_1$ and $-\eta_2$ respectively.

\begin{lemma}
\label{lem-sol-decay}
Suppose that $a \not = 0$.
Let $L = \min \{\eta_1, \eta_2\}$. Also, let
$L_1 = (\lambda_1 + |\lambda_1-\lambda_2|) \frac{\eta_1 + \eta_2}{\eta_1}$.
Suppose $\tx(0)$ is such that
$$y(0)=\max \{ |\tx_1(0)|, \tx_2(0)/|a| \} \le L/2L_1.$$
Then, for all $t \ge 0$, $|\tx_1(t)| \le 2 y(0) e^{-tL}$, and $\tx_2(t) \le 2 |a| y(0) e^{-t L}$.
\end{lemma}
\begin{proof}
We can write, as is standard,
$$\tx(t) = \tx(0) + \int_0^t F(\tx(s))ds,$$
where $F: \R^2 \to \R^2$ is given by
\begin{eqnarray}
\label{eq.drift}
F(x) = \begin{pmatrix}
F_1(x) \\
F_2 (x)
\end{pmatrix}
=
\begin{pmatrix}
-\eta_1 x_1 - \lambda_1 x_1^2 - \frac{\eta_2(\lambda_1-\lambda_2)}{\eta_1} \Big (\frac{x_2}{a} \Big )^2 + \Big ( \lambda_1-\lambda_2+ \frac{\lambda_1\eta_2}{\eta_1} \Big ) x_1 \frac{x_2}{a}\\
 - \eta_2 x_2 - \lambda_2 x_2 x_1  + \frac{\lambda_2}{a} \frac{\eta_2}{\eta_1} x_2^2
\end{pmatrix}.
\end{eqnarray}

We then decompose
$$F(x) = \tilde{A} \begin{pmatrix} x_1 \\ x_2 \end{pmatrix} + \tilde{F} (x),$$
where
\begin{eqnarray}
\label{eq.matrix}
\tilde{A} = \begin{pmatrix}
-\eta_1  &  0\\
0 & -\eta_2,
\end{pmatrix}
\end{eqnarray}
and
\begin{equation}
\label{eq-error}
\tilde{F}(x) = \begin{pmatrix} - \lambda_1 x_1^2 - \frac{\eta_2(\lambda_1-\lambda_2)}{\eta_1} \Big (\frac{x_2}{a} \Big )^2 + \Big ( \lambda_1-\lambda_2+ \frac{\lambda_1\eta_2}{\eta_1} \Big ) x_1 \frac{x_2}{a} \\
- \lambda_2 x_2 x_1  + \frac{\lambda_2}{a} \frac{\eta_2}{\eta_1} x_2^2
 \end{pmatrix}.
\end{equation}
It is then not hard to check that the solution $\tx(t)$ satisfies
\begin{eqnarray*}
\tx (t) = e^{t\tilde{A}} \tx (0) + \int_0^t e^{(t-s)\tilde{A}} \tilde{F} (\tx (s)) ds,
\end{eqnarray*}
or, equivalently,
\begin{eqnarray*}
&& \begin{pmatrix} \tx_1(t)\\ \tx_2(t) \end{pmatrix} =
\begin{pmatrix} e^{-t\eta_1} \tx_1(0)  \\ e^{-t \eta_2}x_2(0) \end{pmatrix}\\
&& {} +   \int_0^t \begin{pmatrix} e^{-(t-s)\eta_1} [- \lambda_1 \tx_1(s)^2 - \frac{\eta_2(\lambda_1-\lambda_2)}{\eta_1} \Big (\frac{\tx_2(s)}{a} \Big )^2 + \Big ( \lambda_1-\lambda_2+ \frac{\lambda_1\eta_2}{\eta_1} \Big ) \tx_1(s) \frac{\tx_2(s)}{a}] \\
e^{-(t-s)\eta_2}[- \lambda_2 \tx_2(s) \tx_1(s) + \frac{\lambda_2}{a} \frac{\eta_2}{\eta_1} \tx_2(s)^2]\end{pmatrix} ds.
\end{eqnarray*}
Let $y_1(t) = |\tx_1(t)| e^{Lt}$, let $y_2(t) = \frac{\tx_2(t)}{|a|} e^{Lt}$ and let $y(t) = \max \{y_1(t), y_2(t)\}$. Then, from the above,
\begin{eqnarray*}
y_1(t) & \le & y_1 (0) + \int_0^t e^{Ls} \Big [\lambda_1 \tx_1(s)^2 + \frac{\eta_2|\lambda_1-\lambda_2|}{\eta_1}  \Big (\frac{\tx_2(s)}{a} \Big )^2 \\
&&{} + \Big | \lambda_1-\lambda_2+ \frac{\lambda_1\eta_2}{\eta_1} \Big | |\tx_1(s)| \frac{\tx_2(s)}{|a|} \Big ]ds\\
& \le & y_1 (0) + \Big (\lambda_1 + \frac{\eta_2|\lambda_1-\lambda_2|}{\eta_1} +   | \lambda_1-\lambda_2 |+ \frac{\lambda_1\eta_2}{\eta_1} \Big ) \int_0^t y(s)^2 e^{-Ls} ds\\
& \le & y_1 (0) + \frac{\eta_1 + \eta_2}{\eta_1} \Big (\lambda_1 + |\lambda_1-\lambda_2|  \Big ) \int_0^t y(s)^2 e^{-Ls} ds,
\end{eqnarray*}
and also
\begin{eqnarray*}
y_2(t) \le y_2 (0) +  \lambda_2 \frac{\eta_1 + \eta_2}{\eta_1} \int_0^t y(s)^2 e^{-Ls}ds,
\end{eqnarray*}
so that
\begin{eqnarray*}
y(t) \le y(0) +  L_1 \int_0^t y(s)^2 e^{-Ls} ds.
\end{eqnarray*}
Now, the equation
\begin{eqnarray*}
z(t) = z(0) +  L_1 \int_0^t z(s)^2 e^{-Ls} ds
\end{eqnarray*}
is solved by
$$z(t) = \frac{L z(0)}{L + L_1 z(0) (e^{-Lt}-1)},$$
for all $t$, as long as $z(0) < L/L_1$,
and, if $z(0) \le L/2L_1$, then $z(t) \le 2z(0)$ for all $t$. Now a standard argument, considering the difference $z(t) - y(t)$, shows that $y(t) \le z(t) \le 2 z(0)$ provided that $y(0) \le z(0) \le L/2L_1$.

So $|\tx_1 (t)| \le 2 \max \{\tx_1(0),\tx_2 (0)/|a|\} e^{-Lt}$ and $\tx_2 (t) \le 2|a| \max \{\tx_1(0),\tx_2 (0)/|a|\} e^{-Lt}$ for all $t$, as required.
\end{proof}

\begin{lemma}
\label{lem-sol-decay-1}
Suppose that $a \not = 0$.
Let $L,L_1$ be as in Lemma~\ref{lem-sol-decay}. Suppose $\tilde{x}(0)$ is such that
$$y(0) = \max \{ |\tilde{x}_1(0)|, \tx_2(0)/|a| \} \le L/8L_1.$$
Then, for all $t \ge 0$, $x_2(t) \le 2 x_2(0) e^{-t \eta_2}$, and
$x_2(t) \ge \frac12 x_2(0) e^{-t \eta_2}$.
\end{lemma}

\begin{proof}
%
Defining $y(t)$ as in the proof of Lemma~\ref{lem-sol-decay},
\begin{eqnarray*}
x_2(t) & \le & x_2(0) e^{-t \eta_2}
 +   \lambda_2  \int_0^t e^{-(t-s) \eta_2} | \tx_1(s) | x_2(s) ds
 +  \frac{\lambda_2}{|a|}\frac{\eta_2}{\eta_1}  \int_0^t e^{-(t-s) \eta_2}\tx_2 (s)^2 ds\\
& \le &  x_2 (0) e^{-t \eta_2}
 +  2\lambda_2 \frac{\eta_1+\eta_2}{\eta_1} e^{-t \eta_2} y(0) \int_0^t x_2 (s) e^{s \eta_2} e^{-sL} ds.
\end{eqnarray*}
Letting $\tilde{y}_2(t) = x_2(t) e^{t \eta_2}$, and using the fact that $\lambda_2 \le \lambda_1 + |\lambda_1 - \lambda_2|$,
\begin{eqnarray*}
\tilde{y}_2(t) & \le & \tilde{y}_2(0)
 +  2 L_1  y(0) \int_0^t \tilde{y}_2 (s) e^{-s L} ds,
\end{eqnarray*}
so, by Gronwall's lemma,
$$\tilde{y}_2 (t) \le \tilde{y}_2 (0) \exp (2 L_1 y(0) \int_0^t e^{-s L} ds )\le  \tilde{y}_2 (0) \exp (2 L_1y(0)/L),$$
so $y_2(t) \le 2 y_2 (0)$, for all $t$, as required, since $y(0) \le L/8L_1$.


Furthermore, for all $t$,
\begin{eqnarray*}
x_2(t) & \ge & x_2(0) e^{-t \eta_2}
 -  \lambda_2  \int_0^t e^{-(t-s) \eta_2} |\tx_1(s)|   x_2(s)  ds
 -  \frac{\lambda_2}{|a|}\frac{\eta_2}{\eta_1}  \int_0^t e^{-(t-s) \eta_2}\tx_2 (s)^2 ds\\
& \ge & x_2(0) e^{-t \eta_2}
 -   2L_1 e^{-t\eta_2} y(0) \int_0^t x_2(s) e^{s\eta_2} e^{-Ls} ds\\
& \ge &  x_2(0) e^{-t \eta_2}
-  4L_1 y(0)x_2(0) e^{-t \eta_2}  \int_0^t e^{-s L} ds\\
& \ge &  x_2(0) e^{-t \eta_2}
- \frac{4L_1 y(0)}{L} x_2(0)  e^{-t \eta_2} \\
& \ge & \frac12 x_2(0) e^{-t \eta_2}.
\end{eqnarray*}
\end{proof}


\smallskip

%
%


In the case $\eta_1 = \eta_2$, we work with the original variables $x_1(t),x_2(t)$, and write
\begin{eqnarray*}
\lefteqn{\begin{pmatrix} x_1(t) - \frac{\lambda_1 -\mu_1}{\lambda_1} \\ x_2(t) \end{pmatrix}  =  e^{t A}
\begin{pmatrix} x_1(0) - \frac{\lambda_1 -\mu_1}{\lambda_1} \\ x_2(0) \end{pmatrix}
+ \int_0^t e^{A (t-s)} \tilde{F} (x(s)) ds} \\
& = & e^{t A}
\begin{pmatrix} x_1(0) - \frac{\lambda_1 -\mu_1}{\lambda_1} \\ x_2(0) \end{pmatrix}\\
&&{} +   \int_0^t e^{A (t-s)} \begin{pmatrix} - \lambda_1 \Big (x_1(s) - \frac{\lambda_1-\mu_1}{\lambda_1} \Big )^2  - \lambda_1 (x_1(s) - \frac{\lambda_1-\mu_1}{\lambda_1} ) x_2(s) \\ - \lambda_2 (x_1(s) - \frac{\lambda_1-\mu_1}{\lambda_1}) x_2(s) - \lambda_2 (x_2(s))^2\end{pmatrix} ds,
\end{eqnarray*}
where
\begin{eqnarray*}
e^{tA}
& = & \begin{pmatrix}  e^{-t(\lambda_1 -\mu_1)}  & -(\lambda_1-\mu_1)t e^{-t(\lambda_1 -\mu_1)} \\ 0 & e^{-t (\lambda_1 - \mu_1)} \end{pmatrix}.
\end{eqnarray*}

Let $y(t) = e^{t(\lambda_1-\mu_1)/2} \max \{|x_1(t)-(\lambda_1-\mu_1)/\lambda_1|,x_2(t)\}$.
\begin{lemma}
\label{lem-sol-decay-2}
Suppose that $a = 0$.
Assume that $y(0) \le (\lambda_1-\mu_1)/32 (\lambda_1+\lambda_2)$.
Then, for all $t \ge 0$, $|x_1(t) - (\lambda_1-\mu_1)/\lambda_1|\le 4 y(0) e^{-t(\lambda_1-\mu_1)/2}$ and
$x_2(t) \le 4 y(0) e^{-t(\lambda_1-\mu_1)/2}$.
\end{lemma}
\begin{proof}
Since $\lambda_2[(\lambda_1-\mu_1)t +\lambda_1/\lambda_2]e^{-t(\lambda_1-\mu_1)/2} \le 2(\lambda_1+\lambda_2)$ and $[(\lambda_1-\mu_1)t +1]e^{-t(\lambda_1-\mu_1)/2} \le 2$, we have
\begin{eqnarray*}
\lefteqn{|x_1(t) - (\lambda_1-\mu_1)/\lambda_1|} \\
&\le& 2 y(0) e^{-t(\lambda_1-\mu_1)/2} + 4(\lambda_1+\lambda_2) e^{-t(\lambda_1-\mu_1)/2} \int_0^t y(s)^2 e^{-s(\lambda_1-1)/2} ds,
\end{eqnarray*}
and
\begin{eqnarray*}
x_2 (t) \le y(0) e^{-t(\lambda_1-\mu_1)/2} + 2 \lambda_2 e^{-t(\lambda_1-\mu_1)/2}\int_0^t  y(s)^2 e^{-s(\lambda_1-\mu_1)/2} ds.
\end{eqnarray*}
It follows using the same argument as in the proof of Lemma~\ref{lem-sol-decay} that
$$y(t) \le 2 y(0) + 4(\lambda_1+\lambda_2) \int_0^t y(s)^2 e^{-s(\lambda_1-\mu_1)/2}ds,$$
so
$$y(t) \le \frac{(\lambda_1-\mu_1) y(0)}{(\lambda_1-\mu_1)/2 + 8(\lambda_1+\lambda_2)y(0) (e^{-t(\lambda_1-\mu_1)/2}-1)}.$$
Thus if $y(0) \le (\lambda_1-\mu_1)/32 (\lambda_1+\lambda_2)$, $y(t) \le 4 y(0)$, so
$|x_1(t) - (\lambda_1-\mu_1)/\lambda_1|\le 4 \max \{x_1(0), x_2(0) \} e^{-t(\lambda_1-\mu_1)/2}$ and $x_2(t) \le 4 \max \{x_1(0), x_2(0) \} e^{-t(\lambda_1-\mu_1)/2}$.
\end{proof}

\begin{lemma}
\label{lem-sol-decay-3}
Suppose that $a = 0$.
Assume that 
$y(0) \le (\lambda_1-\mu_1)/32(\lambda_1+\lambda_2)$.
Then, for all $t \ge 0$,
$x_2(t) \le 2 x_2(0) e^{-t(\lambda_1-\mu_1)}$ and $x_2(t) \ge \frac12 x_2(0) e^{-t(\lambda_1-\mu_1)}$.
\end{lemma}
\begin{proof}
Letting $\tilde{y}_2 (t) = x_2(t) e^{(\lambda_1-\mu_1)t}$, we have, using Lemma~\ref{lem-sol-decay-2},
$$\tilde{y}_2 (t) \le \tilde{y}_2 (0) + 4\lambda_2  y(0) \int_0^t \tilde{y}_2(s) e^{-s(\lambda_1-\mu_1)/2}ds,$$
so
$$\tilde{y}_2 (t) \le \tilde{y}_2 (0) \exp (8\lambda_2 y(0)/(\lambda_1-\mu_1)),$$
so $x_2(t) \le 2 x_2(0) e^{-t(\lambda_1-\mu_1)}$,
since $y(0) \le (\lambda_1-\mu_1)/32\lambda_2$. Finally,
\begin{eqnarray*}
\tilde{y}_2 (t) & \ge & \tilde{y}_2 (0) - 4 \lambda_2 y(0) \int_0^t \tilde{y}_2(s) e^{-s(\lambda_1-\mu_1)/2}ds \\
& \ge & \tilde{y}_2 (0) - 8 \lambda_2 y(0) x_2 (0) \int_0^t e^{-s(\lambda_1-\mu_1)/2}ds\\
& \ge & x_2 (0) - 16\lambda_2 y(0) x_2(0)/(\lambda_1-\mu_1),
\end{eqnarray*}
so
$y_2 (t) \ge \frac12 x_2(0) e^{-t(\lambda_1-\mu_1)}$.
\end{proof}


\subsection{Proof strategy}

As we mentioned in the introduction, we break up the analysis of the competition process into phases as follows.

\textbf{Initial phase (`burn-in' period)}:
By standard theory, see for instance Kurtz (1970) or Darling and Norris (2008), over a fixed length (i.e. independent of $N$) interval $[0,t_0]$, $X_N(t)/N$, is well approximated by the solution
$x(t)$ of~$(\ref{eq.det-comp})$ starting from the same (or nearby) initial condition. We will choose $t_{0}$ such that $x_1(t_0)$ is close to $\frac{\lambda_1-\mu_1}{\lambda_1}$ and
$x_2(t_0)$ is very small. (This is possible by Theorem~\ref{thm.zeeman} and the discussion following it, since the fixed point $((\lambda_1-\mu_1)/\lambda_1, 0)^T$ is stable.)

\textbf{Intermediate phase}:
After time $t_0$, we linearise~$(\ref{eq.det-comp})$ and its stochastic analogue around $((\lambda_1-\mu_1)/\lambda_1, 0)^T$, and use this to
show that $x_N(t) = X_N(t)/N$ follows the solution $x(t)$ to~$(\ref{eq.det-comp})$ for quite a long time after $t_0$. Our approach here is a variation on standard martingale techniques adapted to exploit the proximity of a stable fixed point.

We choose the time $t_{N,1}$ as the time when $x_2(t)$ drops down to $N^{-1/4}$ (so $X_{N,2}(t)$ will be around $N^{3/4}$).
%


\textbf{Final phase}: This phase starts with $X_{N,1}(t)$  near $\frac{\lambda_1-\mu_1}{\lambda_1}N$ and $X_{N,2}(t)$  near $N^{3\slash 4}$. From then onwards, `logistic effects' can be ignored, and the path of $X_{N,2}(t)$ can be sandwiched between the paths of two subcritical linear birth and death processes also starting near $N^{3\slash 4}$. Since the time to extinction of a linear birth and death process is well known, we obtain the distribution of the remaining time until the extinction of $X_{N,2}(t)$.

Theorem \ref{teomy} follows by adding up the times spent in each phase.





\section{Initial phase}
\label{initial}

\begin{lemma}
\label{lem-phase1}
Let $x_N(t) = X_N(t)/N$. Let $t_0 > 0$, let $0 < \delta \le (\log 4) t_0 (\lambda_1  + 1)$, and
%
assume that $\|x_N(0) - x(0)\|_1 \le \delta$. Then
$$\P (\sup_{t \le t_0} \|x_N(t) - x(t)\|_1 \le 2 \delta e^{(5 \lambda_1 +1)t_0}) \le 4e^{-\delta^2N/4t_0 (\lambda_1  + 1)}.$$
\end{lemma}

\begin{proof}
The general method 
is described in, for instance, Darling and Norris (2008). In the next section, we will develop a variant adapted to the case where the solution $x(t)$ is in the neighbourhood of a stable fixed point.

As is standard, we write
$$x(t) = x(0) + \int_0^t F(x(s))ds,$$
where $F: \R^2 \to \R^2$ is given by
$$F(x) = \begin{pmatrix}
F_1(x) \\
F_2 (x)
\end{pmatrix}
=
\begin{pmatrix}
\lambda_1 x_1 (1-x_1 - x_2) - x_1\\
\lambda_2 x_2 (1-x_1 - x_2) - x_2
\end{pmatrix}.
$$
Also, $x_N(t) = X_N(t)/N$ satisfies
$$x_N(t) = x_N(0) + \int_0^t F(x_N(s))ds+M_N(t),$$
where $(M_N(t))$ is a zero-mean martingale.

We can take $5 \lambda_1 +1$ for
a Lipschitz constant of $F$ with respect to $\| \cdot \|_1$ in the subset of $\R^2$ given by $\{x= (x_1,x_2)^T: 0 \le x_1, x_2 \le 1\}$. Then for $t \le t_0$,
\begin{eqnarray*}
\lefteqn{\|x_N(t) - x(t)\|_1} \\
& \le & \|x_N(0) - x(0)\|_1 + \int_0^{t_0} \|F(x_N(s)) - F(x(s))\|_1 ds + \sup_{t \le t_0} \|M_N(t)\|_1 \\
& \le & \|x_N(0) - x(0)\|_1 + (5 \lambda_1 +1) \int_0^{t_0} \sup_{u \le s} \|x_N(u)- x(u)\|_1 ds + \sup_{t \le t_0} \|M_N(t)\|_1\\
& \le & ( \|x_N(0) - x(0) \|_1 + \sup_{t \le t_0} \|M_N(t) \|_1 ) e^{(5 \lambda_1 +1) t_0},
\end{eqnarray*}
so
\begin{eqnarray}
\label{eq.lln}
\sup_{t \le t_0} \|x_N(t) - x(t)\|_1 \le ( \|x_N(0) - x(0) \|_1 + \sup_{t \le t_0} \|M_N(t) \|_1 ) e^{(5 \lambda_1 +1 ) t_0}.
\end{eqnarray}

For $\theta \in {\mathbb R}^2$, let
\begin{eqnarray*}
\lefteqn{Z_N(t,\theta)} \\
& = & \exp \Big (\theta^T (x_N(t) -x_N(0)) - \int_0^t ds \sum_y q_N (x_N(s),x_N(s)+y) (e^{\theta^T y} -1) \Big )\\
& = & \exp \Big ( \theta^T M_N(t) - \int_0^t ds \sum_y q_N (x_N(s),x_N(s)+ y) (e^{\theta^T y} -1 - \theta^T y) \Big ),
\end{eqnarray*}
where $q_N(x,x+y)$ denotes the rate of $y$ jumps of $x_N(t)$ when in state $x$. Then $(Z_N(t,\theta))$ is a mean 1 martingale. Note that, for our model, the jumps $y$ are of the form $(\pm 1/N,0)^T$ and $(0, \pm 1/N)^T$, and $F(x) = \sum_y q_N (x,x+y) y$.

Using the identity $e^z-1 - z = z^2 \int_{r=0}^1 e^{rz}(1-r)dr$, we write $Z_N(t,\theta)$ as
\begin{eqnarray*}
\lefteqn{Z_N(t,\theta)}\\
& = &  \exp \Big  (  \theta^T M_N(t) - \int_0^t \sum_y q_N(x_N(s),x_N(s) + y) (\theta^T y)^2 (\int_0^1 e^{r\theta^T y}(1-r)dr )  ds    \Big ).
\end{eqnarray*}
As the jumps $y$ are of the form $(\pm 1/N,0)^T$ and $(0, \pm 1/N)^T$,
$$\int_0^1 e^{r\theta^T y}(1-r)dr \le \frac12 e^{\gamma},$$
for $\|\theta \| \le \gamma N$.
It follows that, for all $t$,
$$Z_N (t,\theta)  \ge  \exp \Big  (  \theta^T M_N(t) - \frac12 e^{\gamma } \int_0^t \sum_y q_N(x_N(s),x_N(s)+y) (\theta^T y )^2  ds\Big ).$$

In particular, let $\theta_1, \theta_2 \in \R$, and let $\theta^i = \theta_i e_i$ (where $e_i$ is the unit vector with 1 in the $i$-th co-ordinate).
Then, for all $t$, for $i=1,2$,
$$Z_N (t,\theta^{i}) \ge \exp \Big  (  \theta_i M_{N,i}(t) - \frac12 e^{\gamma }\theta_i^2 t(\lambda_1  +1)\frac{1}{N}\Big ).$$
For $\delta > 0$, let $T^i_+(\delta) = \inf \{t \ge 0: M_{N,i}(t) > \delta \}$ and let $T^i_-(\delta) = \inf \{t \ge 0: M_{N,i}(t) < -\delta \}$.
By optional stopping and Markov inequality,
$$\P (T^i_+(\delta) \le  t_0 ) \le  \exp (-\theta_i \delta +\frac12 e^{\gamma } \theta_i^2 t_0 (\lambda_1  + 1)\frac{1}{N}).$$
Choosing $\gamma = \log 2$, $\theta_i = N \delta/2t_0 (\lambda_1  +1)$
$0 \le \theta_i \le \gamma N$, as long as $\delta \le (\log 4) t_0 (\lambda_1  + 1)$.
We then obtain, for $i=1,2$,
$$\P (T^i_+(\delta) \le  t_0 ) \le e^{-\delta^2N/4t_0 (\lambda_1  + 1)}.$$
Arguing similarly about negative $\delta$, we see that, for $i=1,2$,
$$\P ( \sup_{t \le t_0} |M_{N,i}(t)| > \delta ) \le 2 e^{-\delta^2N/4t_0 (\lambda_1  + 1)}.$$
Then the lemma follows from~(\ref{eq.lln}).
%

\end{proof}



\section{Intermediate phase: long-term differential equation approximation}

\label{sec:lt-approx}

As in the previous section, we use $x_N(t)$ to denote $X_N(t)/N$. The aim of this section is to show that $x_N(t)$ stays concentrated around the solution $x(t)$ of the deterministic system~(\ref{eq.det-comp}) for a long period of time, provided $x_N(0)$ and $x(0)$ are close to each other, and $x(0)$ is close to the fixed point $((\lambda_1-\mu_1)/\lambda_1,0)^T$.

We will treat in detail only the case where the eigenvalues of matrix $A$ are distinct, so that $a \not = 0$.
By analogy with the notation in Section~\ref{subs:conv-det}, we let $\tilde{x}_{N,1}(t) = x_{N,1}(t) - \frac{\lambda_1-\mu_1}{\lambda_1} + \frac{1}{a}x_{N,2}(t)$, $\tilde{x}_{N,2}(t) = x_{N,2}(t)$, and we let $\tilde{x}_N(t)$ be the column vector with components $\tilde{x}_{N,1}(t)$ and $\tilde{x}_{N,2}(t)$.

As in Section~\ref{subs:conv-det}, let $L = \min \{\eta_1, \eta_2\}$ and let $L_1 = \Big (\lambda_1 + |\lambda_1-\lambda_2| \Big ) \frac{\eta_1 + \eta_2}{\eta_1}$.

Let $b = \frac{|a|+1}{|a|}$, let $a_1 = (2 \eta_1)^{-1}b^2 (\lambda_1 + \mu_1 + \lambda_2 + \mu_2)$
and let $a_2 = (2\eta_2)^{-1} (\lambda_2 + \mu_2)$. Let also $\tilde{L} = \max \{\eta_1,\eta_2\}$.

\begin{lemma}
\label{lem-lt-approx}
Let $\omega$ satisfy $0 < \omega < 4 (\log 2)^2 N a_i/b^2$ for $i=1,2$.
For $t > 0$, let
$$f_N(t) = \max \{ |\tx_{N,1}(t) - \tx_1 (t)|, |\tx_{N,2}(t) - \tx_2(t)| |a|^{-1} \},$$
and suppose that
$$f_N(0) \le e^{\tilde{L}} \Big (\frac{\omega (a_1+a_2)}{N} \Big )^{1/2}.$$
Suppose also that $y(0) = \max \{\tilde{x}_1(0), |a|^{-1}\tilde{x}_2(0) \} \le L/8L_1$.
Then
$$\P \Big (\sup_{t \le \lceil e^{\omega/8} \rceil } f_N(t) > 8 e^{\tilde{L}} \Big (\frac{\omega (a_1+a_2)}{N} \Big )^{1/2} \Big ) \le 8e^{-\omega/8}.$$
\end{lemma}

The proof of Lemma~\ref{lem-lt-approx} will follow shortly. By standard theory,
$$\tilde{x}_N(t) = \tilde{x}_N(0) + \int_0^t F(\tilde{x}_N(s))ds + M_N(t),$$
where $(M_N(t))$ is a martingale, and $F(x)$
is the drift of $(\tilde{x}_N(t))$ when in state $x$, given as in~(\ref{eq.drift}). Analogously to the deterministic process $x(t)$,
\begin{eqnarray*}
\lefteqn{\tilde{x}_N(t)  =  e^{\tilde{A} t} \tilde{x}_{N,0}
+ \int_0^t e^{\tilde{A} (t-s)} \tilde{F} (X_N(s)) ds
+ \int_0^t e^{\tilde{A} (t-s)} dM_N(s)} \\
& = &
\begin{pmatrix} e^{-t \eta_1} \tilde{x}_{N,1}(0) \\ e^{-t \eta_2}\tilde{x}_{N,2}(0) \end{pmatrix}\\
& + &  \int_0^t  \begin{pmatrix} e^{-(t-s)\eta_1} \Big [- \lambda_1 \tilde{x}_{N,1}(s)^2 - \frac{\eta_2(\lambda_1-\lambda_2)}{\eta_1} \Big (\frac{\tilde{x}_{N,2}(s)}{a} \Big )^2 + \Big ( \lambda_1-\lambda_2 + \frac{\lambda_1 \eta_2}{\eta_1}  \Big )  \frac{\tilde{x}_{N,1}(s)\tilde{x}_{N,2}(s)}{a} \Big ] \\
e^{-(t-s)\eta_2} \Big [- \lambda_2 \tilde{x}_{N,2}(s) \tilde{x}_{N,1}(s) + \frac{\lambda_2 \eta_2}{\eta_1 a}  \tilde{x}_{N,2}(s)^2 \Big ]\end{pmatrix} ds\\
& + & + \int_0^t e^{\tilde{A} (t-s)} dM_N(s),
\end{eqnarray*}
where $\tilde{A}$ is as in~(\ref{eq.matrix}) and $\tilde{F}$ is as in~(\ref{eq-error}). (This formula is proven in the same way as the `variation of constants' formula in Lemma 4.1 in Barbour and Luczak (2012).)

The following analysis of the martingale transform $\int_0^t e^{(t-s)\tilde{A}} dM_N(s)$ is generic, and applicable in the context of any finite dimensional Markov chain.



\begin{lemma}
\label{lem.mart-dev}
Let $(X(t))$ be a Markov chain with state space $S \subseteq {\mathbb R}^k$, where $k$ is a positive integer. For $x,y \in {\mathbb R}^k$ such that $x, x+y \in S$, let $\tilde{q}(x,x+y)$ denote the rate of jump $y$ from $x$, and assume that $\|y\| \le B$ for each possible jump $y$. Suppose further that, for each $x \in S$, the drift $F(x):=\sum_y \tilde{q}(x,x+y) y$ at $x$ can be written in the form
$$F(x) = \tilde{A}x + \tilde{F}(x),$$
where $\tilde{A}$ is a $k \times k$ matrix with non-positive eigenvalues. Let $(M(t))$ be the corresponding Dynkin martingale, that is,
$$X(t) = X(0) + \int_0^t (\tilde{A}X(s) + \tilde{F} (X(s)))ds + M(t).$$
Given a vector ${\bf e} \in {\mathbb R}^k$ with $\|{\bf e}\| =1$, let $v_{{\bf e}} (x,u) = \sum_y \tilde{q}(x,x+y) ({\bf e}^T e^{\tilde{A}u}y )^2$, and, for $K > 0$,
let $S_{{\bf e}} (K)= \inf \{t \ge 0: \int_0^t v_{{\bf e}} (X(s),t-s) ds > K \}$. Let $S_i (K) = S_{{\bf e}_i} (K)$, where ${\bf e}_i$ is a unit vector with $1$ in the $i$-th co-ordinate.

Suppose ${\bf e}$ is a unit eigenvector of $\tilde{A}$ with eigenvalue $-\eta$, where $\eta \ge 0$. Then, given $K,\sigma > 0$, and $0 < \omega < 4 (\log 2)^2 K/B^2$
$$\P( \sup_{t \le \sigma \lceil e^{\omega/8}\rceil  \land S_{\bf e} (K)} |\int_0^t {\bf e}^T e^{\tilde{A} (t-s)}dM(s)| > e^{\sigma \eta}\sqrt{\omega K} ) \le 4 e^{-\omega/8}.$$

Given numbers $K_1, \ldots, K_k > 0$, let $S(K_1, \ldots, K_k) = \land_{i=1}^k S_i (K_i)$. Let~$\omega$ satisfy $0 < \omega < 4 (\log 2)^2 K_i/B^2$ for each $i$. Then, for an arbitrary unit vector ${\bf e}$,
\begin{eqnarray*}
\P \Big(\sup_{t \le \sigma \lceil e^{\omega/8} \rceil  \land S (K_1, \ldots, K_k)} |\int_0^t {\bf e}^T e^{\tilde{A} (t-s)}dM(s)|
> e^{\sqrt{k}\sigma \min \{\|\tilde{A}\|_1, \|\tilde{A} \|_{\infty}}\sqrt{\omega} \sqrt{\sum_i K_i}\Big) \\
\le 4k e^{-\omega/8}.
\end{eqnarray*}
\end{lemma}
\begin{proof}
Fix a time $\tau > 0$, and consider $M^{\tau}$ given by
$$M^{\tau}(t) = \int_0^{t\land \tau} e^{\tilde{A} (\tau-s)} dM(s) = \int_0^{t \land \tau} e^{\tilde{A}(\tau - s)} (dX(s) - \sum_y y \tilde{q}(X(s),X(s)+y) ds).$$
Then $(M^{\tau}(t))$ is a zero mean martingale. Also, for each $t \ge 0$,
$\int_0^t e^{\tilde{A} (t-s)} dM(s) = M^t (t)$.
We now define 
\begin{eqnarray*}
\lefteqn{Y^{\tau}(t)} \\
& = & \int_0^{t \land \tau} e^{\tilde{A}(\tau - s)} Fs(X(s))ds + M^{\tau}(t)
 =  \int_0^{t \land \tau}  \sum_y e^{\tilde{A} (\tau - s)}  y \tilde{q}(X(s),X(s)+y) ds\\
& + & \int_0^{t \land \tau} e^{\tilde{A}(\tau - s)} (dX(s) - \sum_y y \tilde{q}(X(s),X(s)+y) ds)
 =  \int_0^{t \land \tau} e^{\tilde{A}(\tau - s)} dX(s).
\end{eqnarray*}

For $\theta \in \R^k$, let $R^{\tau}(t,\theta)$ be defined by
$$R^{\tau}(t,\theta) = e^{\theta^T Y^{\tau}(t)} - \int_0^{t \land \tau} \sum_y \tilde{q}(X(r),X(r)+y) \Big (e^{\theta^T[Y^{\tau}(r) + e^{\tilde{A}(\tau-r)}y]} -e^{\theta^T Y^{\tau}(r)} \Big ) dr.$$
Then $(R^{\tau}(t,\theta))$ is a martingale.
Also, for $\theta \in \R^k$, $(Z^{\tau} (t,\theta))$ given by
$$Z^{\tau} (t,\theta)= e^{\theta^T Y^{\tau}(t)} \exp \Big (- \int_0^{t \land \tau} \sum_y \tilde{q}(X(s),X(s)+y) (e^{\theta^T e^{\tilde{A}(\tau-s)}y} -1)   ds \Big )$$
is a mean 1 martingale, since, for all $t$, using integration by parts,
\begin{eqnarray*}
\lefteqn{Z^{\tau}(t,\theta) =  e^{\theta^T Y^{\tau}(0)}} \\
&&{} + \int_0^{t} \exp \Big (- \int_0^{r \land \tau} \sum_y \tilde{q}(X(s),X(s)+y) (e^{\theta^T e^{\tilde{A}(\tau-s)}y} -1)   ds \Big ) d e^{\theta^T Y^{\tau}(r)}\\
&&{} -  \int_0^{t \land \tau} e^{\theta^T Y^{\tau}(r)}\sum_y \tilde{q}(X(r),X(r)+y) (e^{\theta^T e^{\tilde{A}(\tau-r)}y} -1) \\
&&{} \times \exp \Big (- \int_0^{r \land \tau} \sum_y \tilde{q}(X(s),X(s)+y) (e^{\theta^T e^{\tilde{A}(\tau-s)}y} -1)   ds \Big ) dr\\
& = & 1  +  \int_0^{t \land \tau} \exp \Big (- \int_0^{r \land \tau} \sum_y \tilde{q}(X(s),X(s)+y) (e^{\theta^T e^{\tilde{A}(\tau-s)}y} -1)   ds \Big ) \\
&&{} \times  \Big ( d e^{\theta^T Y^{\tau}(r)} -  \sum_y \tilde{q}(X(r),X(r)+y) \Big (e^{\theta^T[Y^{\tau}(r) + e^{\tilde{A}(\tau-r)}y]} -e^{\theta^T Y^{\tau}(r)} \Big )dr \Big )\\
& = & 1+ \int_0^{t \land \tau} \exp \Big (- \int_0^{r \land \tau} \sum_y \tilde{q}(X(s),X(s)+y) (e^{\theta^T e^{\tilde{A}(\tau-s)}y} -1)   ds \Big ) dR^{\tau}(r, \theta).
\end{eqnarray*}

Note that
\begin{eqnarray*}
\lefteqn{Z^{\tau}(t,\theta)} \\
& = & \exp \Big (\theta^T Y^{\tau}(t)- \int_0^{t \land \tau} \sum_y \tilde{q} (X(s),X(s)+y) \theta^T e^{\tilde{A} (\tau-s)}y ds \Big )\\
& \times & \exp \Big ( -\int_0^{t \land \tau} \sum_y \tilde{q}(X(s),X(s)+y) (e^{\theta^T e^{\tilde{A}(\tau-s)}y} -1- \theta^T e^{\tilde{A}(\tau-s)}y )   ds \Big )\\
& = &  \exp \Big  (  \theta^T M^{\tau}(t) - \int_0^{t \land \tau} \sum_y \tilde{q}(X(s),X(s)+y) (e^{\theta^T e^{\tilde{A}(\tau-s)}y} -1- \theta^T e^{\tilde{A}(\tau-s)}y ) ds  \Big ).
\end{eqnarray*}
Using the identity $e^z-1 - z = z^2 \int_{r=0}^1 e^{rz}(1-r)dr$, we can rewrite that as
\begin{eqnarray*}
\lefteqn{Z^{\tau}(t,\theta)  =   \exp \Big  (  \theta^T M^{\tau}(t)} \\
&-&\int_0^{t \land \tau} \sum_y \tilde{q}(X(s),X(s)+y) (\theta^T e^{\tilde{A}(\tau-s)}y)^2 (\int_0^1 e^{r\theta^T e^{\tilde{A}(\tau-s)}y}(1-r)dr) ds \Big ).
\end{eqnarray*}
%
We have assumed that each jump $y$ satisfies $\|y\| \le B$, so, since $\tilde{A}$ has negative eigenvalues, $\|e^{\tilde{A}(\tau-s)}y\|$ is also always bounded by $B$. Hence, for $\|\theta \| \le \Gamma$,
$$\int_0^1 e^{r\theta^T e^{\tilde{A}(\tau-s)}y}(1-r)dr \le
\frac12 e^{B \Gamma}.$$

It follows that, for all $t$, 
$$Z^{\tau} (t,\theta) \ge \exp \Big  (  \theta^T M^{\tau}(t) - \frac12 e^{B\Gamma} \int_0^{t \land \tau} \sum_y \tilde{q}(X(s),X(s)+y) (\theta^T e^{\tilde{A}(\tau-s)}y )^2  ds\Big ).$$

Writing $\theta = \|\theta \| {\bf e}$, for a unit vector ${\bf e} \in {\mathbb R}^k$,
$$Z^{\tau}(t,\theta) \ge \exp \Big  (  \|\theta \| {\bf e}^T M^{\tau}(t) - \frac12 e^{B\Gamma } \|\theta\|^2\int_0^{t \land \tau} v_{{\bf e}} (X(s),\tau-s)ds \Big ).$$
Recall that for $K >0$,
$S_{\bf e} (K)= \inf \{t \ge 0: \int_0^t v_{\bf e} (X(s),t-s) ds > K \}$. Given any path of the process $(X(t))$, $\int_0^t v_{\bf e} (X(s),t-s) ds$ is continuous in $t$, and so $S_{\bf e} (K)$ is a stopping time. Also, for $\tau > 0$, $K > 0$, let
$$S^{\tau}_{\bf e} (K)= \inf \{t \ge 0: \int_0^{t \land \tau} v_{\bf e}(X(s),\tau-s) ds > K \}.$$
Then necessarily either $S^{\tau}_{\bf e} (K) \le \tau$ or $S^{\tau}_{\bf e} (K) = \infty$; and $S^{\tau}_{\bf e} (K) \le \tau$ if and only if $S_{\bf e}(K) \le \tau$. Therefore, also, if $S_{\bf e} (K) > \tau$, then $S^t_{\bf e} (K) > t$ for all $t \le \tau$.

For $t,\tau> 0$, unit vector ${\bf e}$, let $M_{\bf e}^{\tau}(t) = e^T M^{\tau}(t)$. Also given
$\delta > 0$, let $T^{\tau,+}_{\bf e}(\delta) = \inf \{t \ge 0: M_{\bf e}^{\tau}(t) > \delta \}$, and let
$T^{\tau,-}_i (\delta) = \inf \{t \ge 0: M_{\bf e}^{\tau}(t) < -\delta \}$. Then $T^{\tau,\pm}_{\bf e}(\delta) \le \tau$ or $T^{\tau,\pm}_{\bf e}(\delta) =\infty$.

Given $K > 0$, on the event $\{T^{\tau,+}_{\bf e}(\delta) \le  \tau \land S^{\tau}_{\bf e} (K) \}$,
$$Z^{\tau} (T^{\tau,+}_{\bf e}(\delta), \|\theta \|{\bf e}) \ge \exp (\|\theta \| \delta -\frac12 e^{B\Gamma} \|\theta \|^2 K).$$
By optional stopping and Markov inequality,
$$\P (T^{\tau,+}_{\bf e}(\delta) \le  \tau \land S^{\tau}_{\bf e} (K)  ) \le  \exp (-\|\theta \| \delta +\frac12 e^{B\Gamma} \|\theta \|^2 K).$$
Choosing $\|\theta \| = \delta e^{-B\Gamma}/K$, and assuming $\delta \le \Gamma K e^{B\Gamma}$ so that
$\|\theta \| \le \Gamma$, we obtain
$$\P (T^{\tau,+}_{\bf e}(\delta) \le  \tau \land S^{\tau}_{\bf e} (K)  ) \le e^{-\delta^2/2K e^{B \Gamma}},$$
and, similarly,
$$\P (T^{\tau,-}_{\bf e}(\delta) \le  \tau \land S^{\tau}_{\bf e} (K) ) \le e^{-\delta^2/2Ke^{B\Gamma}}.$$
Letting $T^{\tau}_{\bf e}(\delta) = T^{\tau,+}_{\bf e}(\delta) \land T^{\tau,-}_{\bf e}(\delta)$, it follows that
$$\P (T^{\tau}_{\bf e}(\delta) \le  \tau \land S^{\tau}_{\bf e} (K) ) \le 2e^{-\delta^2/2Ke^{B\Gamma}}.$$
Choosing $\Gamma = B^{-1} \log 2$ and $\delta = \sqrt{\omega K}$, with $0 < \omega < 4 (\log 2)^2 K/B^2$, we have
$\|\theta \| \le \Gamma$, so
$$\P ( \sup_{t \le \tau \land S^{\tau}_{\bf e} (K)} |M_{\bf e}^{\tau}(t)| > \sqrt{\omega K}) \le 2 e^{-\omega/4}.$$
We will now apply this to all times $\tau_1 = \sigma, \tau_2 = 2 \sigma, \ldots, \tau_{\lceil e^{\omega/8} \rceil}= \lceil e^{\omega/8} \rceil \sigma$, up to time $\sigma \lceil e^{\omega/8} \rceil$. Since $e^{\omega/8} \ge 1$,
\begin{eqnarray*}
\P (\exists j: T^{\tau_j}_{\bf e}(\sqrt{\omega K}) \le \tau_j \land S^{\tau_j}_{\bf e} (K) ) \le 4 e^{-\omega/8}.
\end{eqnarray*}

Now, for $t \le \tau$,
$M^t(t)=e^{\tilde{A}(t - \tau)} M^{\tau}(t)$,
and so, $M^t_{\bf e} (t) = {\bf e}^T e^{\tilde{A}(t - \tau)} M^{\tau}(t)$. In particular, if ${\bf e}$ is an eigenvector of $\tilde{A}$ with eigenvalue $-\eta$, then $M^t_{\bf e} (t) = e^{-\eta(t - \tau)} M_{\bf e}^{\tau}(t)$. In general,
\begin{eqnarray*}
M^t_{\bf e} (t) & = & {\bf e}^T e^{\tilde{A}(t - \tau)} M^{\tau}(t) \le \| {\bf e}^T e^{\tilde{A}(t - \tau)} \| \|M^{\tau}(t)\| \\
& \le & \|e^{\tilde{A}(t - \tau)} \| \sqrt{\sum_i (M^{\tau}_i (t))^2} \le e^{\nu (t-\tau)} \sqrt{\sum_i (M^{\tau}_i (t))^2}\\
& \le & e^{\|\tilde{A}\| (t-\tau)}  \sqrt{\sum_i (M^{\tau}_i (t))^2} \le e^{\min \{\|\tilde{A}\|_1, \|\tilde{A}\|_{\infty}\} \sqrt{k}(t-\tau)}  \sqrt{\sum_i (M^{\tau}_i (t))^2},
\end{eqnarray*}
where $M^{\tau}_i (t) = M^{\tau}_{{\bf e}_I} (t)$ with ${\bf e}_i$ a unit vector with $1$ in the $i$-th co-ordinate, $\nu$ is the largest eigenvalue of $(\tilde{A} + \tilde{A}^T)/2$, $\|\tilde{A}\|_{\infty} = \max_i \sum_j |\tilde{A}_{ij}|$ and $\|\tilde{A}\|_1 = \max_j \sum_i |\tilde{A}_{ij}|$.

Choosing $\sigma < \tau$, if ${\bf e}$ is an eigenvector of ${\tilde A}$ with eigenvalue $-\eta$, then
$$\sup_{\tau-\sigma \le t \le \tau} |M_{\bf e}^{t}(t)|   \le e^{\sigma \eta} \sup_{\tau-\sigma \le t \le \tau} |M_{\bf e}^{\tau}(t)| .$$
For general ${\bf e}$,
$$\sup_{\tau-\sigma \le t \le \tau} |M_{\bf e}^{t}(t)|\le e^{\min \{\|\tilde{A}\|_1, \|\tilde{A}\|_{\infty}\} \sqrt{k}\sigma }  \sqrt{\sum_i (M^{\tau}_i (t))^2}.$$
%

From the above, since if $t < S_{\bf e} (K)$, then $S^t_{\bf e} (K)= \infty$, for ${\bf e}$ an eigenvector of $\tilde{A}$ with eigenvalue $\eta$,
$$\P( \sup_{t \le \sigma \lceil e^{\omega/8}\rceil  \land S_{\bf e} (K)} |M_{\bf e}^t(t)| > e^{\sigma \eta}\sqrt{\omega K} ) \le 4 e^{-\omega/8}.$$
Recalling that, given numbers $K_1, \ldots, K_k > 0$, $S (K_1, \ldots, K_k) = \land_{i=1}^k S_i (K_i)$, we have, for a general vector ${\bf e}$,
$$\P( \sup_{t \le \sigma \lceil e^{\omega/8}\rceil  \land S(K_1, \ldots, K_k)} |M_{\bf e}^t(t)| > e^{\min \{\|\tilde{A}\|_1, \|\tilde{A}\|_{\infty}\} \sqrt{k}\sigma } \sqrt{\omega} \sqrt{\sum_i K_i} ) \le 4k e^{-\omega/8}.$$
\end{proof}

\smallskip



{\bf Proof of Lemma~\ref{lem-lt-approx}}.
The possible jumps $y$ of $(\tx_N(t))$ are of the form $\pm (1/N,0)^T$ and $\pm (1/aN, 1/N)^T$, so $\|e^{\tilde{A}(\tau-s)}y\| \le B:=b/N$ for $s \le \tau$, where $b = (|a|+1)/|a|$.

We bound $\int_0^t v_{i} (\tx_N(s),t-s)ds = \int_0^t \sum_y \tilde{q}_N(\tx_N(s),\tx_N(s)+y) (e^{\tilde{A}(t-s)}y)_i )^2ds$:
\begin{eqnarray*}
\lefteqn{\int_0^t v_{1} (\tx_N(s),t-s)ds} \\
& \le &   \frac{\lambda_1}{N} \int_0^t x_{N,1}(s) (1-x_{N,1}(s) - x_{N,2}(s))e^{-2\eta_1(t-s)} ds \\
& + &  \frac{\mu_1}{N} \int_0^t x_{N,1}(s) e^{-2\eta_1(t-s)}ds + \frac{\mu_2}{N a^2}  \int_0^t x_{N,2}(s) e^{-2 \eta_1 (t-s)} ds\\
& + & \frac{\lambda_2}{Na^2} \int_0^t x_{N,2}(s) (1-x_{N,1}(s) - x_{N,2}(s))e^{-2\eta_1(t-s)}ds \\
& \le & N^{-1} (\lambda_1 + \mu_1) \int_0^t  e^{-2\eta_1(t-s)} ds
 +  N^{-1} \frac{(\lambda_2 + \mu_2)}{a^2} \int_0^t e^{-2\eta_1(t-s)} ds \\
& \le & N^{-1} (2 \eta_1)^{-1} \Big [(\lambda_1 +\mu_1) + a^{-2}(\lambda_2 + \mu_2) \Big ]  \le \frac{a_1}{N};
%
\end{eqnarray*}
\begin{eqnarray*}
\lefteqn{\int_0^t v_{2} (\tx_N(s),t-s)ds} \\
& \le &
 \frac{\lambda_2}{N} \int_0^t x_{N,2}(s)
 (1-x_{N,1}(s) - x_{N,2}(s))
 e^{-2\eta_2(t-s)} ds
 +   \frac{\mu_2}{N} \int_0^t x_{N,2}(s) e^{-2\eta_2(t-s)} ds\\
& \le &  N^{-1} (\lambda_2 + \mu_2) \int_0^t  e^{-2\eta_2(t-s)} ds
 \le  N^{-1} \frac{\lambda_2 + \mu_2}{2 \eta_2} = \frac{a_2}{N}.
\end{eqnarray*}
%
Now, 
\begin{eqnarray*}
|\tx_{N,2}(t) - \tx_2(t)| & \le & e^{-t \eta_2} |\tx_{N,2}(0) -\tx_2(0)| + \Big |\int_0^t e^{-(t-s)\eta_2} dM_{N,2}(s) ds \Big |\\
& + & \frac{\lambda_2}{|a|} \frac{\eta_2}{\eta_1}\int_0^t e^{-(t-s) \eta_2}  |(\tx_{N,2}(s))^2-(\tx_2(s))^2| ds\\
& + & \lambda_2 \int_0^t e^{-(t-s) \eta_2} \Big | \tx_{N,1}(s) \tx_{N,2}(s) - \tx_1(s) \tx_2(s) \Big | ds,\\
\end{eqnarray*}
and
\begin{eqnarray*}
\lefteqn{|\tx_{N,1}(t) - \tx_1(t)|} \\
& \le & e^{-t \eta_1} |\tx_{N,1}(0) -\tx_1(0)| +  \Big |\int_0^t e^{-(t-s) \eta_1} dM_{N,1}(s) ds \Big |\\
& + & \lambda_1 \int_0^t e^{-(t-s) \eta_1} \Big | \tx_{N,1}(s)^2- \tx_1(s)^2 \Big | ds\\
& + & \frac{\eta_2|\lambda_1-\lambda_2|}{\eta_1 a^2}\int_0^t e^{-(t-s) \eta_1}  \Big | \tx_{N,2}(s)^2 - \tx_2(s)^2 \Big | ds\\
& + & \frac{1}{|a|}\Big (|\lambda_1-\lambda_2| + \frac{\lambda_1 \eta_2}{\eta_1} \Big )\int_0^t e^{-(t-s) \eta_1}
\Big | \tx_{N,1}(s)\tx_{N,2}(s) -  \tx_1(s) \tx_2(s) \Big | ds.\\
\end{eqnarray*}

Then, using Lemma~\ref{lem-sol-decay}, 
\begin{eqnarray*}
\lefteqn{\frac{1}{|a|}|\tx_{N,2}(t) - \tx_2(t)|} \\
& \le & e^{-t \eta_2} f_N(0) + \Big |\int_0^t e^{- (t-s)\eta_2} dM_{N,2}(s) ds \Big |\\
& + &  \frac{\lambda_2}{ |a|} \frac{\eta_2}{\eta_1}\int_0^t e^{-(t-s) \eta_2}  f_N(s) (2 \tx_2(s) + |a| f_N(s)) ds\\
& + & \frac{1}{|a|} \int_0^t e^{-(t-s) \eta_2} \lambda_2
f_N(s) \Big (\tx_2(s) + |a| f_N(s) +  |a| |\tx_1(s)| \Big )ds\\
& \le & e^{-t \eta_2} f_N(0) + \lambda_2 \Big (  \frac{\eta_2}{\eta_1} + 1 \Big )\int_0^t (f_N(s))^2 e^{-(t-s) \eta_2} ds \\
& + & \lambda_2 \Big (  \frac{2 \eta_2}{\eta_1} + 1 \Big )\int_0^t f_N(s) e^{-(t-s) \eta_2} \frac{1}{|a|} \tx_2(s)  ds
 +  \lambda_2 \int_0^t f_N(s) e^{-(t-s) \eta_2} |\tx_1(s)| ds\\
& + & \Big |\int_0^t e^{-(t-s) \eta_2} dM_{N,2}(s) ds \Big |\\
& \le & e^{-t \eta_2} f_N(0) + L_1 \int_0^t (f_N(s))^2 e^{-(t-s) \eta_2} ds
 +
4L_1 y(0) e^{-tL}\int_0^t f_N(s)  ds\\
& + & \Big |\int_0^t e^{-(t-s) \eta_2} dM_{N,2}(s) ds \Big |,
\end{eqnarray*}
and
\begin{eqnarray*}
\lefteqn{|\tx_{N,1}(t) - \tx_1(t)|} \\
& \le &  e^{-t \eta_1}f_N(0)
+ \Big |\int_0^t e^{-(t-s)\eta_1} dM_{N,1}(s) ds \Big |\\ 
& + & \lambda_1 \int_0^t e^{-(t-s) \eta_1} f_N(s) (2 |\tx_1(s)| + f_N(s))ds\\
& + & \frac{|\lambda_1-\lambda_2| \eta_2}{\eta_1}\int_0^t e^{-(t-s) \eta_1} f_N(s) \Big ( 2 \frac{\tx_2(s)}{|a|} +f_N(s) \Big ) ds\\
& + & \Big ( |\lambda_1-\lambda_2| + \frac{\lambda_1 \eta_2}{\eta_1} \Big ) \int_0^t e^{-(t-s) \eta_1}
f_N(s) \Big (\frac{\tx_2(s)}{|a|} + f_N(s) + |\tx_1(s)| \Big )ds\\
%
& \le &  e^{-t \eta_1} f_N(0)
 +  L_1\int_0^t e^{-(t-s) \eta_1}(f_N(s))^2 ds\\
& + &  4L_1 e^{-tL} y(0) \int_0^t f_N(s) ds
 +  \Big |\int_0^t e^{-(t-s)\eta_1} dM_{N,1}(s) ds \Big |.
%
\end{eqnarray*}

%

Let $T_1$ be the infimum of times $t$ such that
$$\Big |\int_0^t e^{-(t-s)\eta_i} dM_{N,i}(s) ds \Big |> \Big (\frac{\omega a_i}{N} \Big )^{1/2} e^{\eta_i},$$
for $i=1$ or $i=2$.
%
%
%
On the event $t < T_1$,
\begin{eqnarray*}
\frac{|\tx_{N,2}(t) - \tx_2(t)|}{|a|}
& \le & e^{-t \eta_2} f_N(0) + L_1 \int_0^t  e^{-(t-s) \eta_2} (f_N(s))^2 ds \\
& + & 4L_1 y(0) e^{-t L} \int_0^t f_N(s) ds + \Big (\frac{\omega a_2}{N} \Big )^{1/2} e^{\eta_2},
\end{eqnarray*}
and
\begin{eqnarray*}
|\tx_{N,1}(t) - \tx_1(t)|
& \le & e^{-t\eta_1} f_N(0)
 +  L_1 \int_0^t e^{-(t-s) \eta_1}(f_N(s))^2 ds\\
& + & 4L_1 y(0) e^{-t L} \int_0^t f_N(s) ds
 +  \Big (\frac{\omega a_1}{N} \Big )^{1/2} e^{\eta_1}.
\end{eqnarray*}
Hence, for $t < T_1$, 
\begin{eqnarray*}
f_N(t) &\le& e^{-tL} f_N(0) + L_1 \int_0^t e^{-(t-s)L}(f_N(s))^2 ds \\
&&{} + 4L_1 y(0) e^{-t L} \int_0^t f_N(s) ds + \Big (\frac{2\omega (a_1+a_2)}{N} \Big )^{1/2} e^{\tilde{L}}.
\end{eqnarray*}
Let $T_2 = \inf \{t: f_N(t) > 10e^{\tilde{L}} \Big (\frac{\omega (a_1+a_2}{N} \Big )^{1/2}\}$. Then on the event $t < T_1 \land T_2$,
\begin{eqnarray*}
f_N(t) &\le& e^{-tL} f_N(0) + 100 e^{2 \tilde{L}} \frac{L_1}{L} \frac{\omega (a_1+a_2)}{N}  + 4L_1 y(0) e^{-t L} \int_0^t f_N(s) ds \\
&&{}+ \Big (\frac{2\omega (a_1+a_2)}{N} \Big )^{1/2} e^{\tilde{L}},
\end{eqnarray*}
so, for $N$ large enough,
\begin{eqnarray*}
f_N(t) \le e^{-tL} f_N(0)  + 4L_1 y(0) e^{-t L} \int_0^t f_N(s) ds + 2\Big (\frac{2\omega (a_1+a_2)}{N} \Big )^{1/2} e^{\tilde{L}}.
\end{eqnarray*}
Letting $g_N(t) = f_N(t) e^{tL}$, we see that, for large $N$, on the event $t < T_1 \land T_2$,
$$g_N(t) \le g_N(0)  + 4L_1 y(0) \int_0^t e^{-sL} g_N(s) ds + 2 e^{tL} \Big (\frac{2\omega (a_1+a_2)}{N} \Big )^{1/2} e^{\tilde{L}}.$$
By Gr\"onwall's lemma, for large $N$, for $t < T_1 \land T_2$,
\begin{eqnarray*}
g_N(t)
& \le & \Big (g_N(0) + 2 e^{tL + \tilde{L}} \Big (\frac{2\omega (a_1+a_2)}{N} \Big )^{1/2} \Big ) e^{\frac{4L_1 y(0)}{L}},
\end{eqnarray*}
so, if $y(0) \le L/8L_1$ and $f_N(0) \le e^{\tilde{L}} (\omega (a_1+a_2)/N)^{1/2}$, then
$$f_N(t) \le 2 \Big (f_N(0) + 2 e^{\tilde{L}} \Big (\frac{2\omega (a_1+a_2)}{N} \Big )^{1/2} \Big ) \le 8 e^{\tilde{L}} \Big (\frac{\omega (a_1+a_2)}{N} \Big )^{1/2}.$$
Fix $0 < t_0 \le e^{\omega/8}$.
Let $T_3 = \inf \{t: f_N(t) > 8e^{\tilde{L}} \Big (\frac{\omega (a_1+a_2)}{N} \Big )^{1/2}\}$. 
We now apply Lemma~\ref{lem.mart-dev} to $(\tilde{x}_N(t))$, with matrix $\tilde{A}$ as in~\eqref{eq.matrix} and $\tilde{F}$ as in~\eqref{eq-error}, $B=b/N$, $\sigma = 1$. We take $\eta$ to be equal to $\eta_i$ and $K$ to be equal to $K_i = a_i N^{-1}$, and ${\bf e}$ to be equal to ${\bf e}_i$, for  $i=1,2$, and note that we have shown that $\P (\int_0^t v_i (\tx_N(s),t-s) ds \le K_i \quad i=1,2, \quad \forall t \le  \lceil e^{\omega/8}\rceil \}) =1$. Lemma~\ref{lem.mart-dev} then implies that
$
\P (T_1 \le  \lceil e^{\omega/8}\rceil ) \le 8 e^{-\omega/8}$.

Also, since jumps are of size $O(1/N)$, $\P (T_2 \le T_3) = 0$ for large $N$.
Furthermore, we showed above that $\P (T_3 < T_1 \land T_2) = 0$. Then we can only have $T_3 < T_1$ if $T_3 \ge T_2$, and hence $\P (T_3 < T_1) = 0$.
It follows that
$$\P (T_3 \le \lceil e^{\omega/8} \rceil ) \le \P (T_1 \le \lceil e^{\omega/8} \rceil ) + \P (T_3 < T_1 \land T_2 ) + \P (T_2 \le T_3) \le 
8 e^{-\omega/8},$$
which completes the proof of Lemma~\ref{lem-lt-approx}.

\begin{remark}
\label{remark-2}
When $a=0$, the matrix $A$ has a repeated eigenvalue, and is not diagonalisable. We instead work with the original variables $x_{N,1}(t)$ and $x_{N,2}(t)$. Letting $f_N(t) = \max \{ |x_{N,1}(t) - x_1(t)|, |x_{N,2}(t)-x_2(t)|\}$, we can show, using Lemmas~\ref{lem-sol-decay-2} and~\ref{lem-sol-decay-3}, and the fact that $\lambda_2e^{-(\lambda_1-\mu_1)t/2} [(\lambda_1-\mu_1)t +\lambda_1/\lambda_2] \le 2(\lambda_1 + \lambda_2)$ for all $t$, that
\begin{eqnarray*}
f_N(t) & \le & 2 e^{-(\lambda_1-\mu_1)t/2} f_N(0) + 4 (\lambda_1 + \lambda_2) \int_0^t e^{-(\lambda_1-\mu_1)(t-s)/2} f_N(s)^2 ds \\
& + & 32 (\lambda_1 + \lambda_2) y(0)e^{-(\lambda_1-\mu_1)t/2} \int_0^t f_N(s) ds + M,
\end{eqnarray*}
where $$M = \max \Big \{\Big |\int_0^t (e^{A(t-s)} dM_N(s))_1 ds \Big |, \Big |\int_0^t (e^{A(t-s)} dM_N(s))_2 ds \Big |  \Big \},$$
\begin{eqnarray*}
e^{uA}
& = & \begin{pmatrix}  e^{-u(\lambda_1 -\mu_1)}  & -(\lambda_1-\mu_1)u e^{-u(\lambda_1 -\mu_1)} \\ 0 & e^{-u (\lambda_1 - \mu_1)} \end{pmatrix},
\end{eqnarray*}
$y(0) = \max \{|x_1 (0) - (\lambda_1-\mu_1)/\lambda_1|, x_2(0)\}$.
The remainder of the analysis can then be carried out in a way analogous to the case with distinct eigenvalues. We apply Lemma~\ref{lem.mart-dev}, taking ${\bf e} = {\bf e}_i$ for $i=1,2$, and noting that ${\bf e}_2$ is an eigenvector of $A$. We further take $B=3/N$, $K_1 = K_2 = N^{-1} (2\eta)^{-1} (\lambda_1 + \mu_1 + \lambda_2 + \mu_2)$, where $\eta = \eta_1 = \eta_2 = \lambda_1 - \mu_1$.
Then we can bound the size of the martingale deviation $M$ by $e^{3(\lambda_1-\mu_1)/2}N^{-1/2} \sqrt{\omega (\lambda_1 + \mu_1+\lambda_2 + \mu_2)/(\lambda_1-\mu_1)}$, and hence we can show that
$$f_N(t) \le 8 e^{3(\lambda_1-\mu_1)/2}N^{-1/2} \sqrt{\omega (\lambda_1 + \mu_1+\lambda_2 + \mu_2)/(\lambda_1-\mu_1)},$$
provided $f_N(0) \le e^{3(\lambda_1-\mu_1)/2}N^{-1/2} \sqrt{\omega (\lambda_1 + \mu_1+\lambda_2 + \mu_2)/(\lambda_1-\mu_1)}$ and $y(0) \le (\lambda_1-\mu_1)/128(\lambda_1 + \lambda_2)$.
\end{remark}

\section{Final Phase}

We will prove the following lemma.
\begin{lemma}
\label{lem-final-phase}
For $w \in {\mathbb R}$, let
$t_N(y,w) = (\mu_2-\lambda_2\mu_1/\lambda_1)^{-1} (\log y + \log (1-\lambda_2\mu_1/\lambda_1\mu_2)+w)$.
Let $0 < \varepsilon < 1/4$.
Suppose that $|x_{N,1}(0) - \frac{\lambda_1-\mu_1}{\lambda_1}| \le N^{-\eps}$ and $|x_{N,2}(0) - N^{-1/4}| \le N^{-1/3}$.
Then, as $N \to \infty$,
$$\P (\kappa_N \le t_N(N^{3/4},w)) \to e^{-e^{-w}}.$$
\end{lemma}

Let $(Y(t))_{t\geq0}$ be a subcritical linear birth and death chain, with birth and death rates $\lambda$ and $\mu$, where $\mu>\lambda$.
Let $T^Y=\inf\{ t\geq 0: Y(t)=0\}$.
It is a well known fact, see for example Renshaw (2011), that, for $t\geq 0$,
\begin{eqnarray*}
\P\left(T^Y\le t \right) &=& \P\left( Y(t) = 0 \right)
= \left( 1 -\frac{(\mu-\lambda)e^{-(\mu-\lambda)t}}{\mu-\lambda e^{-(\mu-\lambda)t}}  \right)^{Y(0)}.
\end{eqnarray*}
Assume that $Y(0) = y$, and let
$$t(y,w) = \frac{\log y + \log (\mu - \lambda) - \log \mu + w}{\mu - \lambda},$$
for $y$ and $w$ such that this is positive. Then $e^{-(\mu - \lambda) t(y,w)} = \mu e^{-w}/(\mu - \lambda) y$, so
\begin{eqnarray*}
\P\left(T^Y\le t(y,w) \right) &=& \left ( 1- \frac{e^{-w}/y}{1 - \lambda  e^{-w}/(\mu - \lambda) y} \right )^y.
\end{eqnarray*}

Now, consider a sequence $(Y_N(t))$ of linear birth and death chains with birth rate $\lambda = \lambda (N)$ and death rate $\mu = \mu (N)$, where $\mu (N) > \lambda (N)$. Assume further that $Y_N(0) = y(N)$, where $y (\mu - \lambda) \to \infty$.
Then
\begin{eqnarray*}
\P\left(T^{Y_N}\le t(y(N),w) \right) &=& \left ( 1- \frac{e^{-w}/y(N)}{1 - \lambda  e^{-w}/(\mu - \lambda) y(N)} \right )^{y(N)}\\
& = & \left ( 1- \frac{e^{-w}}{y(N) - \lambda  e^{-w}/(\mu - \lambda)}\right )^{y(N)} \to e^{-e^{-w}},
\end{eqnarray*}
as $N \to \infty$. In other words,
the following result holds for the asymptotic distribution of the extinction times of a sequence of subcritical linear birth and death chains.

\begin{lemma}
\label{extinction}
Let $(Y_N(t))$ be a sequence of subcritical linear birth and death chains with birth and death rates $\lambda (N)$ and $\mu (N)$, respectively,
where $\mu (N) >\lambda (N)$. Suppose further that $Y_N(0)=y(N)$, where $y(\mu - \lambda)  \rightarrow \infty$.  Let $T^{Y_N}=\inf\{ t\geq 0: Y_N(t)=0\}$. Then, as $N \to \infty$,
\[ (\mu (N)-\lambda(N)) T^{Y_N} - \left(\log y(N) + \log (\mu (N)-\lambda (N)) - \log \mu (N) \right) \rightarrow G,\]
in distribution, where $G$ has a standard Gumbel distribution.
\end{lemma}

{\bf Proof of Lemma~\ref{lem-final-phase}.}\
Let $x_1(0) = x_{N,1}(0)$ and $x_2(0) = x_{N,2} (0)$, so $|x_1(0) - (\lambda_1-\mu_1)/\lambda_1| \le N^{-\eps}$ and $x_2 (0) \le 2N^{-1/4}$.

By Lemma~\ref{lem-sol-decay-1}, for large enough $N$, for all $t \ge 0$, $x_2 (t) \le 4 N^{-1/4}$.
Also, by Lemma~\ref{lem-sol-decay}, if $N$ is large enough,
for all $t \ge 0$,
$|x_1(t) -(\lambda_1-\mu_1)/\lambda_1| \le 4 N^{-\eps}$.

Let $Z_N$ be a linear birth and death process defined as follows. The death rate is $\mu_2$, the birth rate is
$$\lambda_2 \Big (1 - \frac{\lambda_1 - \mu_1}{\lambda_1} + 5N^{-\eps} \Big ) = \frac{\lambda_2\mu_1}{\lambda_1} + 5 \lambda_2 N^{-\eps},$$
and $Z_N(0) = N^{3/4} + N^{2/3}$. By Lemma~\ref{extinction}, as $N \to \infty$, in distribution,
\begin{eqnarray*}
\lefteqn{ \Big (\mu_2-\frac{\lambda_2\mu_1}{\lambda_1} - 5\lambda_2 N^{-\eps} \Big ) T^{Z_N}}\\
&&{} - \left(\log (N^{3/4} + N^{2/3}) + \log \Big (\mu_2-\frac{\lambda_2\mu_1}{\lambda_1} - 5\lambda_2 N^{-\eps}  \Big ) - \log \mu_2 \right)
\rightarrow G,
\end{eqnarray*}
where $G$ has a standard Gumbel distribution, and so, for $w \in {\mathbb R}$,
$$\P  \Big (T^{Z_N} \le \frac{\log (N^{3/4} + N^{2/3}) + \log \Big (\mu_2-\frac{\lambda_2\mu_1}{\lambda_1} - 5\lambda_2 N^{-\eps} \Big ) - \log \mu_2 +w}{\mu_2-\frac{\lambda_2\mu_1}{\lambda_1} - 5\lambda_2 N^{-\eps}}\Big ) \to e^{-e^{-w}}.$$
This means that
$$\P  \Big (T^{Z_N} \le \frac{\log N^{3/4} + \log \Big (\mu_2-\frac{\lambda_2\mu_1}{\lambda_1}\Big )- \log \mu_2 +w+o(1)}{\mu_2-\frac{\lambda_2\mu_1}{\lambda_1}}\Big ) \to e^{-e^{-w}},$$
and so also
$$\Big (\mu_2-\frac{\lambda_2\mu_1}{\lambda_1}  \Big ) T^{Z_N} - \left(\log N^{3/4} + \log \Big (\mu_2-\frac{\lambda_2\mu_1}{\lambda_1} \Big ) - \log \mu_2  \right) \rightarrow G.$$

Let $W_N$ be a linear birth and death process defined as follows. The death rate is $\mu_2$, the birth rate is
$$\lambda_2 \Big (1 - \frac{\lambda_1 - \mu_1}{\lambda_1} - 6 N^{-\eps} \Big ) = \frac{\lambda_2\mu_1}{\lambda_1} - 6\lambda_2N^{-\eps},$$
and $W_N(0) = N^{3/4} - N^{2/3}$. By Lemma~\ref{extinction}, in distribution,
\begin{eqnarray*}
\lefteqn{\Big (\mu_2-\frac{\lambda_2\mu_1}{\lambda_1} + 6\lambda_2 N^{-\eps} \Big ) T^{W_N}
 -  \log (N^{3/4} (1- N^{-1/12}))} \\
&&{} - \log \Big (\mu_2-\frac{\lambda_2\mu_1}{\lambda_1}  +  6\lambda_2 N^{-\eps}  \Big ) + \log \mu_2
 \rightarrow G,
\end{eqnarray*}
where $G$ has a standard Gumbel distribution. As above, it follows also that
$$ \Big (\mu_2-\frac{\lambda_2\mu_1}{\lambda_1}  \Big ) T^{W_N} - \left(\log N^{3/4} + \log \Big (\mu_2-\frac{\lambda_2\mu_1}{\lambda_1} \Big ) - \log \mu_2 \right) \rightarrow G.$$


Let $f_N(t)$ be as in Lemma~\ref{lem-lt-approx}, and let $E_t$ be the event that $f_N(s) \le N^{-1/3}$ for all $s < t$. For $N$ large enough, on the event $E_t$, for all $s < t$,
$$-5N^{-\eps} \le x_{N,1}(s) - \frac{\lambda_1 -1}{\lambda_1} \le 5N^{-\eps},$$
and, furthermore,
$$0 \le x_{N,2}(s) \le 4N^{-1/4} + N^{-1/3} \le N^{-\eps}.$$
Therefore, on the event $E_t$, we can couple $Z_N$, $W_N$ and $X_{N,2}$ in such a way that, for $s  \le t$,
$$W_N(s) \le X_{N,2}(s) \le Z_N(s).$$
It follows that, on the event $E_t$, $T^{Z_N} \le t$ implies $\kappa_N \le t$, and $\kappa_N \le t$ implies $T^{W_N} \le t$. Also, by Lemma~\ref{lem-lt-approx} (with any $\omega = \omega (N)$ such that $\omega (N)/N^{1/3} \to 0$, $\P (\overline{E_t}) \to 0$ as long as $t \le e^{\omega/8}$. So, choosing $\omega (N) = N^{1/4}$, for $t \ge 0$,
$$\P (\{T^{Z_N} \le t \} \cap E_t) \le \P (\{\kappa_N \le t \} \cap E_t) \le \P (\{T^{W_N} \le t \} \cap E_t).$$
Hence, for any fixed $w$,
$$\P (\kappa_N \le t_N(N^{3/4},w) ) \le \P (T^{W_N} \le t_N(N^{3/4},w)) + \P (\overline{E_{t_N(N^{3/4},w)}}) \to e^{-e^{-w}},$$
and
$$\P (\kappa_N \le t_N(N^{3/4},w) ) \ge
\P (T^{Z_N} \le t_N(N^{3/4},w)) - \P (\overline{E_{t_N(N^{3/4},w)}}) \to e^{-e^{-w}},$$
which completes the proof of Lemma~\ref{lem-final-phase}.
\proofbox

\section{Proof of Theorem~\ref{teomy}}

\label{sec:proof}




By assumption, $x_N(0) = (\alpha_N, \beta_N)^T$, where $\alpha_N \to \alpha$ and $\beta_N \to \beta$ as $N \to \infty$. We let $x(0) = x_N(0)$ as the initial condition for~$(\ref{eq.det-comp})$.
By Theorem~\ref{thm.zeeman} and the discussion following it, if $\lambda_1/\mu_1 > \lambda_2/\mu_2$ and $\lambda_1/\mu_1 > 1$, then the fixed point $x^* = \left( \frac{\lambda_1-\mu_1}{\lambda_1},0 \right)^T$ of~$(\ref{eq.det-comp})$ is asymptotically stable, so that there exists $t_0 > 0$ such that, with $y(t)= e^{Lt}\max \{\|\tilde{x}_1 (t)\|, \tilde{x}_2(t) |a|^{-1}$, $L = \min \{\eta_1, \eta_2\}$, as defined in Section~\ref{subs:conv-det},
$|y(t_0)| \le L/8L_1$. 
It is also not hard to see that we can choose a finite $t_0$ that works for every value of $N$, for $N$-dependent initial conditions as above.

Let $t_N = \inf \{t \ge t_0: x_2(t) \le N^{-1/4}\}$. Lemma~\ref{lem-sol-decay-1} implies that, as $N \to \infty$,
$$t_N = \frac{1}{4\eta_2} \log N + O(1).$$
It then also follows from Lemmas~\ref{lem-sol-decay} and~\ref{lem-sol-decay-1} that there exists $0 < \varepsilon < 1/4$ such that, if $N$ is large enough, then
$$\Big |x_1 (t_N) - \frac{\lambda_1 -\mu_1}{\lambda_1} \Big | \le \frac12 N^{-\varepsilon}.$$

By Lemma~\ref{lem-phase1} with $\delta = N^{-5/12}$ and Lemma~\ref{lem-lt-approx} with $\omega = N^{1/4}$, if $N$ is large enough, then
\begin{eqnarray}
\sup_{t \le t_N} |x_{N,1} (t) - x_1(t)| & \le & \frac12 N^{-1/3} \le \frac12 N^{-\varepsilon} \nonumber \\
\sup_{t \le t_N} |x_{N,2} (t) - x_2(t)| & \le & \frac12 N^{-1/3}, \label{final-phase-ic}
\end{eqnarray}
with probability at least $1-e^{-N^{1/12}}$.

On the event that~\eqref{final-phase-ic} holds, we can use Lemma~\ref{lem-final-phase} with $x_{N,1}(t_N)$ and $x_{N,2}(t_N)$ as initial values, since these values satisfy its hypotheses in this case.

\smallskip

By~\eqref{time}, the length $t_0$ of the first phase can be written as
$$\frac{\lambda_2}{\mu_2\lambda_1 - \mu_1 \lambda_2} \log (x_1(t_0)/x_1(0)) - \frac{\lambda_1}{\mu_2\lambda_1 - \mu_1\lambda_2} \log (x_2(t_0)/x_2(0)),$$
and the length $t_{N,1}-t_0$ of the second phase can be written as
$$\frac{\lambda_2}{\mu_2\lambda_1 - \mu_1 \lambda_2} \log (x_1(t_{N,1})/x_1(t_0)) - \frac{\lambda_1}{\mu_2\lambda_1 - \mu_1\lambda_2} \log (x_2(t_{N,1})/x_2(t_0)).$$
On the event ${\mathcal E}(t_N)$ that~\eqref{final-phase-ic} holds, by Lemma~\ref{lem-final-phase}, the length of the third phase is
$$\frac{\lambda_1}{\mu_2\lambda_1 - \mu_1\lambda_2} \left(\log N^{3/4} + \log \Big (1-\frac{\mu_1\lambda_2}{\lambda_1\mu_2} \Big )  \right) + \frac{\lambda_1}{\mu_2\lambda_1 - \mu_1\lambda_2} G_N,$$
where $G_N$ converges in distribution to a standard Gumbel variable $G$.

It follows that, on the event ${\mathcal E}(t_N)$, the total time $\kappa_N$ until the extinction of $X_{N,2}$ is
\begin{eqnarray*}
\frac{\lambda_2}{\mu_2\lambda_1 - \mu_1 \lambda_2} \log (x_1(t_{N,1})/\alpha_N) - \frac{\lambda_1}{\mu_2 \lambda_1 - \mu_1 \lambda_2} \log (N^{-1/4}/\beta_N )\\
+\frac{\lambda_1}{\mu_2\lambda_1 - \mu_1 \lambda_2} (\log N^{3/4}
+ \log \Big (1-\frac{\mu_1\lambda_2}{\mu_2\lambda_1} \Big )  ) + \frac{\lambda_1}{\mu_2\lambda_1 - \mu_1\lambda_2} G_N.
\end{eqnarray*}
Since $\P ({\mathcal E}(t_{N,1})) \to 1$, $\alpha_N \to \alpha$, $\beta_N \to \beta$ as $N \to \infty$, and, for large $N$, $|x_1(t_{N,1}) - (\lambda_1-\mu_1)/\lambda_1 | \le N^{-\eps}$, for an $\eps > 0$, we conclude
\begin{eqnarray}
\label{extinction-formula}
\frac{\mu_2\lambda_1 - \mu_1 \lambda_2}{\lambda_1}\kappa_N - \Big ( \log N \beta \Big (1 -\frac{\mu_1\lambda_2}{\mu_2\lambda_1} \Big )   + \frac{\lambda_2}{\lambda_1} \log \Big (  \frac{1 - \mu_1/\lambda_1}{\alpha} \Big )   \Big ) \to G,
\end{eqnarray}
in distribution, where $G$ is a standard Gumbel, 
so our proof of the first part of Theorem~\ref{teomy} is complete.

The second part of Theorem~\ref{teomy} follows using Theorem~\ref{teomalwina}, since, after the extinction of the weaker species, the stronger species evolves as a single supercritical logistic epidemic, and we have shown that $\kappa_N$ is with high probability negligible in comparison with $\tau_N$.

\section{Near-critical phenomena}

\label{sec:critical}

In this section, we will show that Theorem~\ref{teomy} can be extended to near-criticality.
As a proof of concept, we consider the following special case where $\mu_1 = \mu_2 = 1$, $\lambda_1 > \lambda_2 > 1$, and
$\lambda_1 - \lambda_2 = \lambda_1(N) - \lambda_2(N) \to 0$ as $N \to \infty$ (while $\lambda_1 = \lambda_1(N)$ may or may not tend to $1$).
This may for example model a real-world scenario where a slightly more infectious strain emerges during an outbreak, for instance, via a mutation, and we want to know the time taken for it to supplant the weaker one in the population.

We assume that $(\lambda_1-\lambda_2)(\lambda_1-1)^{-1} \to 0$. We further assume that
%
%
$$N(\lambda_1-\lambda_2)^3 (\lambda_1-1)^{-1} (\log \log (N(\lambda_1-\lambda_2)^2))^{-1} \to \infty.$$
%
We believe that the last condition is not best possible, and that it is only necessary to have $N(\lambda_1-\lambda_2)^2 \to \infty$ for the formula in Theorem~\ref{thm.extinction-nc} to hold.

Note that under our assumptions $(\lambda_2 -1)(\lambda_1-1)^{-1} \to 1$.
Also, the quantity $a$ defined in~\eqref{eq-def-a} satisfies $0 < 1 \le a$ for $N$ large enough, and converges to 1 as $N \to \infty$.
As before, we assume that $X_{N,1}(0) = \alpha_N$, $X_{N,2}(0) = \beta_N$, where $\alpha_N \to \alpha > 0$ and $\beta_N \to \beta > 0$.
We will prove the following result, which is an extension of Theorem~\ref{teomy} to this case.

\begin{theorem}
\label{thm.extinction-nc}
Under the above assumptions,
$$\kappa_N = \frac{\lambda_1}{\lambda_1-\lambda_2} \Big (  \log \Big (N\frac{(\lambda_1-1)(\lambda_1-\lambda_2)}{\lambda_1^2}
\frac{\beta}{\alpha} \Big )  + G_N \Big ),$$
where $G_N$ is a random variable converging in distribution to a Gumbel random variable $G$.
\end{theorem}

\smallskip

To prove Theorem~\ref{thm.extinction-nc}, we first prove long-term estimates on the total number of infectives of either type and on the ratio between the number of infectives of the two types.

Let $\omega = \omega (N)> 0$ satisfy 
$N(\lambda_2-1)^2/\omega \to \infty$ as $N \to \infty$. (Note that our assumptions imply that $N(\lambda_2-1)^2 \to \infty$.) We will show that, with high probability (i.e with probability tending to 1 as $N \to \infty$),
 $x_{N,1}(t) + x_{N,2}(t) \ge (\lambda_2-1)/2\lambda_2$ for all $t \le (\lambda_2-1)^{-1} e^{\omega/8}$.

\begin{lemma}
\label{lem.stoch-dom}
Let $Y_N(t)$ denote the number of infectives in a stochastic logistic SIS epidemic with infection rate $\lambda_2$ and recovery rate $1$.
Suppose that $X_{N,1}(0)+X_{N,2}(0) \ge Y_N(0)$.
Then, for all $t$, $X_{N,1}(t) + X_{N,2}(t)$ stochastically dominates $Y_N(t)$.

Further, let $Z_N(t)$ denote the number of infectives in a stochastic SIS logistic epidemic with infection rate $\lambda_1$ and recovery rate $1$.
Suppose that $X_{N,1}(0)+X_{N,2}(0) \le Z_N(0)$.
Then, for all $t$, $X_{N,1}(t) + X_{N,2}(t)$ is stochastically dominated by $Z_N(t)$.
\end{lemma}
\begin{proof}
The process $X_{N,1}(t) + X_{N,2}(t)$ jumps by $+1$ at rate at least $\lambda_2 (X_{N,1}+X_{N,2}) (1 - X_{N,1}-X_{N,2})$ and jumps by $-1$ at rate $X_{N,1}+X_{N,2}$. We can then couple $X_{N,1}+X_{N,2}$ and $Y_N(t)$ so that they always jump down together as much as possible, and jump up together as much as possible, and otherwise each jumps on its own with any excess rate in either direction. With this coupling, $X_{N,1}(t) + X_{N,2}(t) \ge Y_N(t)$.
The second part can be proved analogously.
\end{proof}

\smallskip

\begin{lemma}
\label{lem.supercrit-log}
Let $Y_N(t)$ denote the number of infective individuals in a stochastic SIS logistic epidemic with infection rate $\lambda = \lambda (N)$ and recovery rate $\mu = \mu (N)$. Let $y(t)$ denote the number of infectives in the corresponding deterministic SIS logistic epidemic.
We assume that $\lambda = \lambda (N)$ and $\mu = \mu (N)$, where $\lambda, \mu$ are bounded, $\lambda > \mu >0$, and $(\lambda (N)- \mu (N))^2 N \to \infty$ as $N \to \infty$.
Let $\omega = \omega (N)> 0$ satisfy 
$N(\lambda-\mu)^2/\omega \to \infty$ as $N \to \infty$.

We assume that $y(0):=y_N \le 2 (\lambda-\mu)/\lambda$, and $y_N/(\lambda - \mu)$ is bounded away from $0$ as $N \to \infty$.
Further, let $Y_N(0)$ be such that $|N^{-1} Y_N(0) - y(0)| \le \frac12 \sqrt{\frac{2\omega (\lambda + \mu)}{N\lambda}}$.
If $y_N \ge (\lambda-\mu)/\lambda$, then, for $N$ sufficiently large,
\begin{eqnarray*}
\P \Big ( \sup_{t \le (\lambda -\mu)^{-1} e^{\omega/8}} |N^{-1} Y_N(t) - y(t)| > 4 e^4 \sqrt{\frac{2\omega (\lambda + \mu)}{N\lambda}}  \Big ) \le 4 e^{-\omega/8}.
\end{eqnarray*}
If $y_N < (\lambda-\mu)/\lambda$, then, for $N$ sufficiently large,
\begin{eqnarray*}
\P \Big ( \sup_{t \le (\lambda -\mu)^{-1} e^{\omega/8}} |N^{-1} Y_N(t) - y(t)| > 4 e^{4(\lambda-\mu)/\lambda y(0)} \sqrt{\frac{2\omega (\lambda + \mu)}{N\lambda}}  \Big ) \le 4 e^{-\omega/8}.
\end{eqnarray*}

Suppose further that $(\lambda-\mu)/2\lambda \le y(0)$.
Then, in particular, since $N (\lambda - \mu)^2/\omega \to \infty$, 
$y_N(t) = N^{-1} Y_N(t)$ remains well concentrated around $y(t)$ until time at least $(\lambda - \mu)^{-1} e^{\omega/8}$.
Consequently, if $N$ is sufficiently large, $(\lambda - \mu)/4\lambda \le y_N(t) \le 4(\lambda - \mu)/\lambda$ for all $t \le (\lambda - \mu)^{-1}  e^{\omega/8}$ with probability at least $1-4 e^{-\omega/8}$.
\end{lemma}
\begin{proof}
Note that $(\lambda - \mu)/\lambda$ is an attractive fixed point for
\begin{equation}
\label{ode.logistic}
\frac{dy(t)}{dt} = \lambda y(t) (1-y(t)) - \mu y(t),
\end{equation}
so $y(t) \to (\lambda - \mu)/\lambda$. 

Let $\tilde{y}(t) = y(t) - (\lambda-\mu)/\lambda$ and let $\tilde{y}_N(t) = N^{-1} Y_N(t) - (\lambda-\mu)/\lambda$. Then
\begin{eqnarray*}
\tilde{y}(t) = \tilde{y}(0) - (\lambda - \mu) \int_0^t \tilde{y}(s) ds - \lambda \int_0^t \tilde{y}(s)^2 ds,
\end{eqnarray*}
and
\begin{eqnarray*}
\tilde{y}_N(t) = \tilde{y}_N(0) - (\lambda - \mu) \int_0^t \tilde{y}_N(s) ds - \lambda \int_0^t \tilde{y}_N(s)^2 ds + m_N(t),
\end{eqnarray*}
where $(m_N(t))$ is a zero-mean martingale.

%

It follows that
\begin{eqnarray*}
\tilde{y}(t) = e^{-(\lambda - \mu) t} \tilde{y}(0)  - \lambda \int_0^t e^{-(\lambda - \mu)(t-s)} \tilde{y}(s)^2 ds,
\end{eqnarray*}
and
\begin{eqnarray*}
\tilde{y}_N(t) = e^{-(\lambda - \mu) t} \tilde{y}_N(0)  - \lambda \int_0^t e^{-(\lambda - \mu)(t-s)} \tilde{y}_N(s)^2 ds + \int_0^t e^{-(\lambda - \mu)(t-s)} dm_N(s).
\end{eqnarray*}
%
%

Letting $f_N(t) = |\tilde{y}(t) - \tilde{y}_N(t)|= |y(t) - y_N(t)|$, we thus have
\begin{eqnarray*}
f_N(t) \le f_N(0) + \lambda \int_0^t e^{-(\lambda - \mu)(t-s)} f_N(s) |\tilde{y}_N(s) + \tilde{y}(s)|ds + \Big |\int_0^t e^{-(\lambda - \mu)(t-s)} dm_N(s)\Big |,
\end{eqnarray*}
and so
\begin{eqnarray*}
f_N(t) & \le & f_N(0) + \lambda \int_0^t e^{-(\lambda - \mu)(t-s)} (f_N(s))^2 ds + 2 \lambda \int_0^t e^{-(\lambda - \mu)(t-s)} f_N(s)|\tilde{y}(s)|ds \\
& + & \Big |\int_0^t e^{-(\lambda - \mu)(t-s)} dm_N(s) \Big |.
\end{eqnarray*}
To estimate the deviations of the martingale transform $\int_0^t e^{-(\lambda - \mu)(t-s)} dm_N(s)$, we let $T_1 = \inf \{t: y_N(t) > 2 y(t)\}$;
on the event $t < T_1$,
\begin{eqnarray*}
\int_0^t v(\tilde{y}_N(s), t-s) & := & \frac{\lambda}{N} \int_0^t y_N(s) (1-y_N(s)) e^{-2(\lambda-\mu)(t-s)}ds \\
&& + \frac{\mu}{N} \int_0^t y_N(s)e^{-2(\lambda-\mu)(t-s)}ds\\
& \le & \frac{2(\lambda + \mu)}{N} \int_0^t  e^{-2(\lambda-\mu)(t-s)} y(s) ds.
\end{eqnarray*}
As is well known,
\begin{equation}
\label{eq.log-sol}
y(t) = \frac{y(0)(\lambda-\mu)}{\lambda y(0) + \lambda e^{-(\lambda - \mu)t} \Big (\frac{\lambda - \mu}{\lambda}-y(0)\Big )},
\end{equation}
and so
\begin{equation}
\label{eq.log-sol2}
\tilde{y}(t) = \frac{\tilde {y}(0) e^{-(\lambda -\mu)t}}
{y(0) \frac{\lambda}{\lambda-\mu} \big(1-e^{-(\lambda-\mu)t}\big) + e^{-(\lambda-\mu)t} }.
\end{equation}
We restrict attention for the moment to the case where $y(0) \ge (\lambda-\mu)/\lambda$, when
we have $\tilde{y}(t) \le \tilde{y}(0) e^{-(\lambda-\mu)t}$, and
$(\lambda-\mu)/\lambda \le y(t) \le y(0) \le 2 (\lambda - \mu)/\lambda$.
(If $y(0) \le (\lambda-\mu)/\lambda$, then we have $\tilde{y}(t) \le \frac{\lambda-\mu}{\lambda y(0)} \tilde{y}(0) e^{-(\lambda -\mu)t}$, and
$y(0) \le y(t) \le (\lambda -\mu)/\lambda$.)
It follows that, on the event $t < T_1$,
\begin{eqnarray*}
\int_0^t v(\tilde{y}_N(s), t-s) \le \frac{2 (\lambda + \mu)}{N\lambda }.
\end{eqnarray*}

Given $\omega = \omega (N)> 0$, let $T_2$ be the infimum of times $t$ such that
$$\Big |\int_0^t e^{-(\lambda - \mu)(t-s)} dm_N(s)\Big |> 3 \sqrt{\frac{2\omega (\lambda + \mu)}{N \lambda}}.$$
 By Lemma~\ref{lem.mart-dev} applied to $(y_N(t))$, with $B=1/N$, $\tilde{A} = -(\lambda - \mu)$, $\sigma = 1/(\lambda - \mu)$, $\eta = \lambda - \mu$, $K= \frac{2 (\lambda + \mu)}{N\lambda }$, we see that,
if $\omega \le 8 (\log 2)^2 N  (\lambda + \mu)/\lambda$ (which will hold for $N$ large enough if $N(\lambda-\mu)^2/\omega \to \infty$) and $(\lambda - \mu)^{-1} \lceil e^{\omega/8} \rceil \ge t_0 = t_0(N)$, then $$\P (T_2 \le T_1 \land t_0 ) \le 4 e^{-\omega/8}.$$

Also, by the above, and using the assumption that $f_N(0) \le \frac12 \sqrt{\frac{2\omega (\lambda + \mu)}{N\lambda }}$, on the event $t_0 < T_1 \land T_2$, for all $t \le t_0$,
\begin{eqnarray*}
f_N(t) &\le& \lambda \int_0^t e^{-(\lambda - \mu)(t-s)} (f_N(s))^2 ds + 2 \lambda \int_0^t e^{-(\lambda - \mu)(t-s)} f_N(s)|\tilde{y}(s)|ds \\
&& + \frac72 \sqrt{\frac{2\omega (\lambda + \mu)}{N\lambda }}.
\end{eqnarray*}
Let $T_3$ be the infimum of times $t$ such that $f_N(t) > 5 e^{4} \sqrt{\frac{2\omega (\lambda + \mu)}{N\lambda}}$. Then, if $N$ is large enough, on the event $t_0 < T_1 \land T_2 \land T_3$, for $t \le t_0$,
\begin{eqnarray*}
f_N(t) \le \frac{\lambda}{\lambda - \mu}  \frac{50 e^{8}\omega  (\lambda + \mu)}{N\lambda}  + 2 \lambda \int_0^t e^{-(\lambda - \mu)(t-s)} f_N(s)|\tilde{y}(s)|ds + \frac72 \sqrt{\frac{2\omega (\lambda + \mu)}{N\lambda}}.
\end{eqnarray*}
Since $N (\lambda - \mu)^2/\omega \to \infty$, then, for $N$ large enough,
$$\frac{\lambda}{\lambda - \mu}  \frac{50 e^{8}\omega (\lambda + \mu)}{N} \le \frac12 \sqrt{\frac{2\omega (\lambda + \mu)}{N\lambda}},$$
and so, for $N$ large enough, on the event $t_0 < T_1 \land T_2 \land T_3$, for $t \le t_0$,
\begin{eqnarray*}
f_N(t) e^{(\lambda - \mu) t} \le 2 \lambda \int_0^t e^{(\lambda - \mu)s} f_N(s)|\tilde{y}(s)|ds + 4 e^{(\lambda - \mu) t} \sqrt{\frac{2\omega (\lambda + \mu)}{N \lambda}}.
\end{eqnarray*}
From~\eqref{eq.log-sol2},
$\tilde{y}(t) \le e^{-(\lambda - \mu)t} \tilde{y}(0) \le 2e^{-(\lambda - \mu)t}  (\lambda - \mu)/\lambda$, so, by Gr\"onwall's inequality, on the event $t_0 < T_1 \land T_2 \land T_3$, for $t \le t_0$,
\begin{eqnarray*}
f_N(t)  \le 4  \sqrt{\frac{2\omega (\lambda + \mu)}{N \lambda }} e^{4}.
\end{eqnarray*}
Letting $T_4$ be the infimum of $t$ with $f_N (t) > 4  \sqrt{\frac{2\omega (\lambda + \mu)}{N\lambda }} e^{4}$, from the above
\begin{eqnarray*}
\P (T_4 \le t_0) & \le & \P (T_1 \land T_2 \land T_3 \le T_4 \land t_0),
\end{eqnarray*}
and so, since $\P (T_3 \le T_4) = 0$ and $\P (T_1 \le T_4) = 0$ for $N$ large enough,
\begin{eqnarray*}
\P (T_4 \le t_0)
& \le & \P (T_1 \le T_4) + \P (T_3 \le T_4) + \P (T_2 \le T_1 \land t_0) \le 4 e^{-\omega/8},
\end{eqnarray*}
and so, as claimed, for $N$ large enough,
\begin{eqnarray*}
\P \Big (\sup_{t \le (\lambda - \mu)^{-1} \lceil e^{\omega/8} \rceil } |y_N (t) - y(t)| > 4 e^4 \sqrt{\frac{2\omega (\lambda + \mu)}{N\lambda }} \Big ) \le 4 e^{-\omega/8},
\end{eqnarray*}
and the remaining conclusions also follow in the case when $y(0) \ge (\lambda - \mu)/\lambda$.

The case $y(0) \le (\lambda- \mu)/\lambda$ is similar, using the inequality from~\eqref{eq.log-sol2} that
$$|\tilde{y}(t)| \le \frac{|\tilde{y}(0)|(\lambda-\mu)}{y(0) \lambda} e^{-(\lambda- \mu)t}.$$
\end{proof}

\begin{lemma}
\label{lem.supercrit-log-1}
Let $Y_N(t)$ denote the number of infective individuals in a stochastic SIS logistic epidemic with infection rate $\lambda = \lambda (N)$ and recovery rate $\mu = \mu (N)$. Let $y(t)$ denote the number of infectives in the corresponding deterministic SIS logistic epidemic.
We assume that $\lambda = \lambda (N)$ and $\mu = \mu (N)$, where $\lambda, \mu$ are bounded, $\lambda > \mu >0$, and $(\lambda (N)- \mu (N))^2 N \to \infty$ as $N \to \infty$.

We further assume that $y_N(0) = N^{-1}Y_N(0) > 2 (\lambda-\mu)/\lambda$, and we let $y(0) = y_N(0)$.
Let $\tau$ be such that $y(\tau) = 2 (\lambda-\mu)/\lambda$.
Then, for $N$ sufficiently large,
\begin{eqnarray*}
\P \Big ( \sup_{t \le \tau} |y_N(t) - y(t)| > 2 (N(\lambda-\mu)^2)^{1/16}\sqrt{\frac{y_N(0)(\lambda + \mu)}{N(\lambda-\mu)}}  \Big ) \le 2 e^{-(N(\lambda-\mu)^2)^{1/8}/8}.
\end{eqnarray*}
\end{lemma}
\begin{proof}
We have
\begin{eqnarray*}
y_N(t)- y(t) & = & (y_N(0) - y(0))  - \lambda \int_0^t  (y_N(s)- y(s)) \Big (y_N(s) + y(s) - \frac{\lambda-\mu}{\lambda} \Big ) ds \\
&& {} + m_N(s).
\end{eqnarray*}

For $t \le \tau$, $y(t) \ge 2(\lambda - \mu)/\lambda$ and so $y_N(s) + y(s)-(\lambda-\mu)/\lambda \ge 0$. By Lemma 3.2 in Brightwell, House and Luczak (2018), 
$$\sup_{t \le \tau} |y_N(t)- y(t)| \le 2|y_N(0)- y(0)| + 2 \sup_{t \le \tau} |m_N(s)|.$$
From~\eqref{eq.log-sol},
$$e^{(\lambda-\mu) \tau}-1 = 1 - \frac{2(\lambda - \mu)}{\lambda y(0)}.$$

Arguing as in the proof of Lemma 3.1 in Brightwell, House and Luczak (2018),
using standard martingale techniques and the fact that
$$\int_0^t y(s) ds = \frac{1}{\lambda}  \log (\lambda y(0) (e^{(\lambda_\mu)t} -1 ) + (\lambda-\mu)) - \frac{1}{\lambda} \log (\lambda - \mu),$$
we see that, if $\phi = \phi (N) \le N^{1/2}$, then
\begin{eqnarray*}
\P \Big (\sup_{t \le \tau} |y_N(t) - y(t)| >  2 \sqrt{y_N(0)\phi (\lambda+ \mu)/N (\lambda-\mu)} \Big ) \le 2 e^{-\phi/8},
\end{eqnarray*}
and so, taking $\phi = (N (\lambda-\mu)^2)^{1/8}$,
\begin{eqnarray*}
\P \Big (\sup_{t \le \tau} |y_N(t) - y(t)| >  2(N(\lambda-\mu)^2)^{1/16} \sqrt{\frac{y_N(0)(\lambda+ \mu)}{N(\lambda-\mu)}} \Big ) \le 2 e^{-(N(\lambda-\mu)^2)^{1/8}/8}.
\end{eqnarray*}
\end{proof}

Next we consider the ratio $Q(X_N(t)) = X_{N,1}(t)/X_{N,2}(t)$. We will show that the value of $Q(X_N(t))$ does not change very much over a time period of length `nearly' $(\lambda_1 - \lambda_2)^{-1}$.

\begin{lemma}
\label{lem.ratio}
Let $\psi = \psi (N) \to \infty$ in such a way that $\psi (\lambda_1-\lambda_2) \to 0$ as $N \to \infty$. Let $t_0  = (\lambda_1-\lambda_2)^{-1} \psi (N)^{-1}$. Then, for $N$ large enough,
$$\P \Big (\sup_{t \le t_0} \Big | \frac{X_{N,1}(t)}{X_{N,2}(t)} - \frac{X_{N,1}(0)}{X_{N,2}(0)} \Big | > 2 \psi (N)^{-1/4}  \Big ) \le 2 e^{-\psi (N)^{1/2}\beta^2/128\alpha (\alpha+\beta)} + 4 e^{-\sqrt{N} (\lambda_2-1)}.$$
\end{lemma}
\begin{proof}
Given a vector $X = (X_1, X_2)^T$ with non-negative integer components, such that $X_2 > 0$ and $X_1 + X_2 \le N$, the drift $g_N(X)$ in $Q(X)$ is
\begin{eqnarray*}
\frac{1}{X_2} \lambda_1 X_1 \Big (1-\frac{X_1}{N}-\frac{X_2}{N} \Big ) - \frac{X_1}{X_2} \\+ X_1 \Big (\frac{1}{X_2 +1} - \frac{1}{X_2} \Big ) \lambda_2 X_2 \Big (1-\frac{X_1}{N}-\frac{X_2}{N} \Big ) + X_1 \Big (\frac{1}{X_2 -1} - \frac{1}{X_2} \Big ) X_2\\
=  \Big (1-\frac{X_1}{N}-\frac{X_2}{N} \Big ) \Big ( \frac{\lambda_1 X_1}{X_2} - \frac{\lambda_2 X_1}{X_2+1} \Big ) - \frac{X_1}{X_2} + \frac{X_1}{X_2 -1}.
\end{eqnarray*}
Clearly, we see that $g_N(X_N(t)) \ge 0$ for all $t$ at most the weaker species extinction time $\kappa_N$. Also,
\begin{eqnarray*}
g_N(X) & = & \frac{X_1}{X_2} (\lambda_1 - \lambda_2) + \lambda_2 X_1 \Big ( \frac{1}{X_2} - \frac{1}{X_2+1} \Big ) - X_1 \Big ( \frac{1}{X_2} - \frac{1}{X_2-1} \Big )\\
&  +  & \Big (\frac{X_1}{N}+\frac{X_2}{N} \Big ) X_1 \Big (  \frac{\lambda_2}{X_2+1} - \frac{\lambda_1}{X_2} \Big )\\
& \le & \frac{X_1}{X_2} (\lambda_1 - \lambda_2) + \lambda_2 X_1 \Big ( \frac{1}{X_2} - \frac{1}{X_2+1} \Big ) - X_1 \Big ( \frac{1}{X_2} - \frac{1}{X_2-1} \Big )\\
& \le & \frac{X_1}{X_2} \Big (\lambda_1 - \lambda_2 + \frac{\lambda_2}{X_2+1} + \frac{1}{X_2-1} \Big )
\end{eqnarray*}

Let $T_1$ be the infimum of times $t$ such that $X_{N,2}(t)-1 < N (\lambda_1 - \lambda_2)$. Then, since $N(\lambda_1-\lambda_2)^2 \to \infty$, $0 \le g_N (X_N(t)) \le 3 Q(X_N(t)) (\lambda_1 - \lambda_2)$ for $t < T_1$, if $N$ is large enough.

We write $Q(X_N(t)) = Q(X_N(0)) + \int_0^t g_N (X_N(s)) ds + M_N(t)$, where $M_N(t)$ is a martingale.
Let $R_N(X)$ be given by
\begin{eqnarray*}
\frac{\lambda_1 X_1}{X_2^2}  \Big (1-\frac{X_1}{N}-\frac{X_2}{N} \Big ) + \frac{X_1}{X_2^2} + X_1^2 \Big (\frac{1}{X_2 +1} - \frac{1}{X_2} \Big )^2 \lambda_2 X_2 \Big (1-\frac{X_1}{N}-\frac{X_2}{N} \Big )\\
 + X_1^2 \Big (\frac{1}{X_2 -1} - \frac{1}{X_2} \Big )^2 X_2.
\end{eqnarray*}
Let $T_2$ be the infimum of times $t$ such that $Q(X_N(t)) > 2 Q (X_N(0))$. For $t < T_1 \land T_2$, if $N$ is large enough,
\begin{eqnarray*}
R_N(X_N(t)) & \le & 6 Q (X_N(0)) (\lambda_1-\lambda_2) + 4 (Q(X_N(0)))^2(\lambda_1-\lambda_2) \\
&& {} + 4 (Q(X_N(0)))^2 (\lambda_1-\lambda_2) \\
& \le & 8 Q(X_N(0)) (1+ Q(X_N(0))) (\lambda_1-\lambda_2).
\end{eqnarray*}
Let $T(\delta) = \inf \{t \ge 0: |M_N(t)| > \delta \}$, and denote $Q(X_N(0)) = q_0$. Then a standard exponential martingale argument, using the bound on the quantity $R_N(X_N(t))$ above, shows that, given $t_0 > 0$, $N$ large enough and $0 \le \delta \le 16 \log 2 t_0  q_0 (1 + q_0)$,
$$\P (T(\delta) \le t_0 \land T_1 \land T_2) \le 2 e^{-\delta^2/32t_0q_0(1+q_0)(\lambda_1-\lambda_2)}.$$
Also, by Gr\"onwall's inequality, on the event $t_0 < T_1 \land T(\delta)$,
\begin{eqnarray*}
\sup_{t \le t_0} Q_N(X_N(t)) & \le & (Q_N(X_N(0)) + \sup_{t \le t_0} |M_N(t)|) e^{3 (\lambda_1 - \lambda_2) t_0}
 \le  (q_0 + \delta) e^{3 (\lambda_1 - \lambda_2) t_0}.
\end{eqnarray*}
Furthermore, on the event $t_0 < T_1 \land T(\delta)$,
$\inf_{t \le t_0} Q_N(X_N(t)) \ge q_0 - \delta$.
In other words, on the event $t_0 < T_1 \land T(\delta)$,
$$\sup_{t \le t_0} |Q_N(X_N(t)) - q_0| \le \delta e^{3 (\lambda_1 - \lambda_2) t_0} + q_0 (e^{3 (\lambda_1 - \lambda_2) t_0} -1).$$
Let $\psi = \psi (N) \to \infty$ as $N \to \infty$,
in such a way that $\psi (\lambda_1-\lambda_2) \to 0$,
and let $t_0 (N) = (\lambda_1-\lambda_2)^{-1} \psi (N)^{-1}$. We also let $\delta = \psi_N^{-1/4}$, so  $\delta \le 16 t_0 \log 2 q_0 (1 + q_0)$ for $N$ sufficiently large. It follows that, for $N$ sufficiently large, 
on the event $t_0 < T_1 \land T(\psi (N)^{-1/4})$,
$$\sup_{t \le t_0} |Q_N(X_N(t)) - q_0| \le 2 \psi (N)^{-1/4}.$$
Let $T_3 = \inf \{t \ge 0: |Q_N(X_N(t) - q_0| > 2 \psi (N)^{-1/4}\}$. Clearly, $\P (T_2 < T_3) = 0$.
Then we have shown that, with $t_0 = (\lambda_1-\lambda_2)^{-1} \psi (N)^{-1}$ as above,
\begin{eqnarray*}
\P (T_3 \le t_0) & \le & \P (T_1 \land T(\psi (N)^{-1/4}) \le t_0 \land T_3) \\
& \le & \P (T(\psi (N)^{-1/4}) \le t_0 \land T_1 \land T_3) + \P (T_1 \le t_0 \land T_3)\\
& \le & 2 e^{-\psi (N)^{1/2}/32q_0(1+q_0)} + \P (T_1 \le t_0 \land T_3).
\end{eqnarray*}
Let $T_4 = \inf \{t \ge 0: X_{N,1}(t) + X_{N,2}(t) < N (\lambda_1-1)/4  \}$. Then, if $N$ is sufficiently large, $\P (T_1 \le t_0 \land T_3) \le \P (T_4 \le t_0)$. We will use Lemma~\ref{lem.stoch-dom}, and Lemma~\ref{lem.supercrit-log}, with $\lambda = \lambda_2$, $\mu = 1$, $\omega = 8\sqrt{N(\lambda_2-1)^2}$. Note $(\lambda_2-1)^{-1} e^{\sqrt{N(\lambda_2-1)^2}} \ge (\lambda_1-\lambda_2)^{-1} \ge t_0$ for large $N$, since
$$e^{\sqrt{N(\lambda_2-1)^2}} \ge e^{\Big (\frac{N(\lambda_2-1)^2}{N(\lambda_1-\lambda_2)^2}   \Big )^{1/2}} \ge  \Big (\frac{N(\lambda_2-1)^2}{N(\lambda_1-\lambda_2)^2} \Big )^{1/2} = \frac{\lambda_2-1}{\lambda_1-\lambda_2}.$$
Hence $\P (T_4 \le t_0) \le 4 e^{-\sqrt{N}(\lambda_2-1)}$, and the result follows, as $q_0 \le 2 \alpha/\beta$ for large $N$.
\end{proof}

\smallskip

For the next phase, after time $(\lambda_1-\lambda_2)^{-1} \psi^{-1}$, we approximate vector $\tx_N(t)$ by the solution $\tx (t)$ to~\eqref{eq-diff-eq-eigen}.
When $\mu_1 = \mu_2 = 1$, then equation~\eqref{eq-diff-eq-eigen} takes the form:
\begin{eqnarray*}
\frac{d \tx_1 (t)}{dt} & = & -(\lambda_1 - 1) \tx_1 (t) - \lambda_1 \tx_1 (t)^2 - \frac{(\lambda_1-\lambda_2)^2}{\lambda_1 (\lambda_1-1)} \Big (\frac{\tx_2 (t)}{a} \Big )^2 + \frac{(\lambda_1 - \lambda_2)\lambda_1}{\lambda_1-1} \tx_1 (t) \frac{\tx_2 (t)}{a}\\
\frac{d \tx_2 (t)}{dt} & = & - \frac{\lambda_1-\lambda_2}{\lambda_1} \tx_2 (t) - \lambda_2 \tx_2 (t) \tx_1 (t) + \frac{\lambda_2}{\lambda_1 a} \frac{\lambda_1 - \lambda_2}{\lambda_1-1} \tx_2 (t)^2.
\end{eqnarray*}

\begin{lemma}
\label{lem-lt-approx-nc}
Assume that $X_{N,1}(0) + X_{N,2}(0) \le 2 N (\lambda_1-1)/\lambda_1$, $X_{N,1}(0)/X_{N,2}(0) = \alpha/\beta + \eps_N$, 
where $\eps_N \to 0$ 
as $N \to \infty$. Let $x_i(0) = N^{-1} X_{N,i}(0)$ for $i=1,2$, with $\tilde{x}_i (0)$ being derived from $x_i (0)$ according to the change of variables.
%
Let $\omega = \omega (N)> 0$, where $\omega (N) \to \infty$ with $N$ and
$N (\lambda_1-\lambda_2)^2/\omega(N) \to \infty$.
For $t \ge 0$, let
$$f_N(t) = \max \Big \{ \frac{\lambda_1-1}{\lambda_1-\lambda_2}|\tx_{N,1}(t) - \tx_1 (t)|, |\tx_{N,2}(t) - \tx_2(t)| \Big \}.$$
Then, for $N$ large enough,
$$\P \Big (\sup_{t \le (\lambda_1-1)^{-1}e^{\omega/8}} f_N(t) > 8 \frac{\lambda_1-1}{\lambda_1-\lambda_2}\sqrt{\frac{\omega}{N}}
e^{32 (\lambda_1 \frac{\beta}{\alpha} +1)} \Big )\le 12 e^{-\omega/8}.$$
\end{lemma}
\begin{proof}
%
Using the integral form of~\eqref{eq-diff-eq-eigen} and its stochastic analogue as in Section~\ref{sec:lt-approx}, writing $\eta_1 = \lambda_1-1$ and $\eta_2 = (\lambda_1-\lambda_2)/\lambda_1$ and noting that $\eta_1 > \eta_2$ for large $N$, we can write
\begin{eqnarray*}
|\tx_{N,1}(t)-\tx_1(t)| & \le & |\tx_{N,1}(0)-\tx_1 (0)|e^{-t\eta_1} + \Big |\int_0^t e^{-\eta_1(t-s)} d M_{N,1}(s) \Big |\\
&  + & \lambda_1 \int_0^t e^{-\eta_1(t-s)} |\tx_{N,1}(s)-\tx_1(s)||\tx_{N,1}(s)+\tx_1(s)| ds\\
& + & \frac{(\lambda_1-\lambda_2)^2}{\lambda_1 (\lambda_1-1)a^2} \int_0^t e^{-\eta_1(t-s)} |\tx_{N,2}(s)-\tx_2(s)| (\tx_{N,2}(s) + \tx_2(s))ds \\
& + & \frac{(\lambda_1 - \lambda_2)\lambda_1}{a(\lambda_1-1)}\int_0^t e^{-\eta_1 (t-s)} |\tx_{N,1}(s)-\tx_1(s)| \tx_{N,2}(s)ds\\
& + & \frac{(\lambda_1 - \lambda_2)\lambda_1}{a(\lambda_1-1)}\int_0^t e^{-\eta_1(t-s)} |\tx_{N,2}(s)-\tx_2(s)| |\tx_1(s)|ds,\\
\end{eqnarray*}
and
\begin{eqnarray*}
\lefteqn{|\tx_{N,2}(t)-\tx_2(t)|} \\
& \le & |\tx_{N,2}(0) - \tx_2(0)|e^{-t\eta_2} +  \lambda_2 \int_0^t e^{-(t-s) \eta_2}  |\tx_{N,2}(s) - \tx_2(s)||\tx_{N,1}(s)|ds\\
& + & \lambda_2 \int_0^t e^{-(t-s) \eta_2}  |\tx_{N,1}(s) - \tx_1(s)|\tx_2(s)ds + \Big |\int_0^t e^{-(t-s) \eta_2} d M_{N,2}(s)\Big |\\
& + & \frac{\lambda_2}{\lambda_1 a} \frac{\lambda_1 - \lambda_2}{\lambda_1-1} \int_0^t e^{-(t-s) \eta_2} |\tx_{N,2}(s)-\tx_2(s)| (\tx_{N,2}(s) + \tx_2(s)) ds.
\end{eqnarray*}

Let
$$f_N(t) = \max \Big \{|\tx_{N,1}(t)-\tx_1(t)| \frac{\lambda_1-1}{\lambda_1-\lambda_2}, |\tx_{N,2}(t)-\tx_2(t)| \Big \},$$
and
let $g_N(t) = e^{t\eta_2} f_N(t)$.
Then
\begin{eqnarray*}
\lefteqn{\frac{\lambda_1-1}{\lambda_1-\lambda_2}|\tx_{N,1}(t)-\tx_1(t)|e^{t\eta_2}} \\
& \le & g_N(0)  + \lambda_1 \int_0^t  g_N(s)|\tx_{N,1}(s)+\tx_1(s)| ds \\
&& {} +  \frac{\lambda_1-\lambda_2}{\lambda_1a^2} \int_0^t g_N(s) (\tx_{N,2}(s) + \tx_2(s))ds +
\frac{(\lambda_1 - \lambda_2)\lambda_1}{a(\lambda_1-1)}\int_0^t g_N(s) \tx_{N,2}(s)ds \\
&& {} +  \frac{\lambda_1}{a} \int_0^t g_N(s) |\tx_1 (s)|ds +
\frac{\lambda_1-1}{\lambda_1-\lambda_2}e^{t\eta_2} \Big |\int_0^t e^{-\eta_1(t-s)} d M_{N,1}(s) \Big |,
\end{eqnarray*}
and
\begin{eqnarray*}
\lefteqn{ |\tx_{N,2}(t)-\tx_2 (t)|e^{t\eta_2}} \\
& \le & g_N(0) +  \lambda_2 \int_0^t g_N(s) |\tx_{N,1}(s)|ds
 +  \lambda_2 \frac{\lambda_1-\lambda_2}{\lambda_1-1} \int_0^t g_N(s) \tx_2 (s)ds\\
& + & \frac{\lambda_2}{\lambda_1 a} \frac{\lambda_1 - \lambda_2}{\lambda_1-1} \int_0^t g_N(s) (\tx_{N,2}(s) + \tx_2 (s)) ds
  +  e^{t \eta_2} \Big |\int_0^t e^{-(t-s) \eta_2} d M_{N,2}(s) \Big |.
\end{eqnarray*}
It follows that
\begin{eqnarray*}
\lefteqn{\frac{\lambda_1-1}{\lambda_1-\lambda_2}|\tx_{N,1}(t)-\tx_1(t)|e^{t\eta_2}}\\
& \le & g_N(0)  + 2\lambda_1 \int_0^t  g_N(s)|\tx_1 (s)| ds
 +  \frac{\lambda_1(\lambda_1-\lambda_2)}{\lambda_1-1} \int_0^t g_N(s)^2 e^{-s \eta_2}ds \\
& + & \frac{2(\lambda_1-\lambda_2)}{\lambda_1a^2} \int_0^t g_N(s) \tx_2 (s)ds
 +  \frac{\lambda_1-\lambda_2}{\lambda_1a^2} \int_0^t g_N(s)^2 e^{-s\eta_2}ds\\
& + & \frac{(\lambda_1 - \lambda_2)\lambda_1}{a(\lambda_1-1)}\int_0^t g_N(s) \tx_2 (s)ds
 +  \frac{(\lambda_1 - \lambda_2)\lambda_1}{a(\lambda_1-1)}\int_0^t g_N(s)^2 e^{-s\eta_2} ds\\
& + & \frac{\lambda_1}{a}\int_0^t g_N(s) |\tx_1 (s)|ds
 +  \frac{\lambda_1-1}{\lambda_1-\lambda_2}e^{t\eta_2} \Big |\int_0^t e^{-\eta_1(t-s)} d M_{N,1}(s) \Big |,
\end{eqnarray*}
and
\begin{eqnarray*}
\lefteqn{|\tx_{N,2}(t)-\tx_2 (t)|e^{t\eta_2}} \\
& \le & g_N(0) +  \lambda_2 \int_0^t g_N(s) |\tx_1(s)|ds + \lambda_2 \frac{\lambda_1-\lambda_2}{\lambda_1-1}\int_0^t g_N(s)^2e^{-s\eta_2}ds\\
& + &   \lambda_2 \frac{\lambda_1-\lambda_2}{\lambda_1-1} \int_0^t g_N(s) \tx_2 (s)ds
 +  \frac{2\lambda_2(\lambda_1-\lambda_2)}{a \lambda_1 (\lambda_1-1)} \int_0^t g_N(s) \tx_2(s)ds\\
& + & \frac{\lambda_2}{\lambda_1 a} \frac{\lambda_1 - \lambda_2}{\lambda_1-1} \int_0^t g_N(s)^2 e^{-s\eta_2}ds
  +  e^{t\eta_2} \Big |\int_0^t e^{-(t-s) \eta_2} d M_{N,2}(s) \Big |.
\end{eqnarray*}
So, since
$a \to 1$ as $N \to \infty$, for $N$ large enough,
\begin{eqnarray*}
\lefteqn{\frac{\lambda_1-1}{\lambda_1-\lambda_2}|\tx_{N,1}(t)-\tx_1(t)|e^{t \eta_2}}\\
& \le & g_N(0)  + 4 \lambda_1 \int_0^t  g_N(s)|\tx_1 (s)| ds +
\frac{2\lambda_1 (\lambda_1 - \lambda_2)}{\lambda_1-1}\int_0^t g_N(s) \tx_2 (s)ds\\
& + &
2\frac{\lambda_1 (\lambda_1-\lambda_2)}{\lambda_1-1} \int_0^t g_N(s)^2 e^{-s\eta_2}ds
 +  \frac{\lambda_1-1}{\lambda_1-\lambda_2}e^{t \eta_2} \Big |\int_0^t e^{-\eta_1(t-s)} d M_{N,1}(s) \Big |,
\end{eqnarray*}
and
\begin{eqnarray*}
\lefteqn{|\tx_{N,2}(t)-\tx_2 (t)|e^{t\eta_2}}\\
& \le & g_N(0) +  \lambda_2 \int_0^t g_N(s) |\tx_1(s)|ds +
4\frac{\lambda_2(\lambda_1-\lambda_2)}{\lambda_1-1} \int_0^t g_N(s) \tx_2 (s)ds\\
& + & 4\frac{\lambda_2(\lambda_1-\lambda_2)}{\lambda_1-1}\int_0^t g_N(s)^2e^{-s \eta_2}ds
  +  e^{t \eta_2} \Big |\int_0^t e^{-(t-s) \eta_2} d M_{N,2}(s) \Big |.
\end{eqnarray*}
Hence, if $N$ is large enough,
\begin{eqnarray*}
g_N(t) & \le & g_N(0)  + 4 \lambda_1 \int_0^t  g_N(s)|\tx_1(s)| ds +
4 \frac{\lambda_1(\lambda_1 - \lambda_2)}{\lambda_1-1}\int_0^t g_N(s) \tx_2 (s)ds\\
& + &
 4\frac{\lambda_1(\lambda_1-\lambda_2)}{\lambda_1-1} \int_0^t g_N(s)^2 e^{-s \eta_2}ds + e^{t \eta_2} M,
\end{eqnarray*}
where $M$ is the maximum of
$\frac{\lambda_1-1}{\lambda_1-\lambda_2}\Big |\int_0^t e^{-\eta_1(t-s)} d M_{N,1}(s)\Big |$ and
$\Big |\int_0^t e^{-(t-s) \eta_2} d M_{N,2}(s) \Big |$.

%
Let $T_1$ be the infimum of times $t$ such that $x_{N,1}(t) + x_{N,2}(t) > 4(\lambda_1-1)/\lambda_1$. 
Similarly to the calculations in the proof of Lemma~\ref{lem-lt-approx},
on the event $t < T_1$, for $N$ large enough,
\begin{eqnarray*}
\int_0^t \sum_y q(\tilde{x}_N(s), \tilde{x}_N(s) + y) (e^{\tilde{A}(t-s)}y)_1^2 ds & \le &
\frac{4 (\lambda_1 + 1)}{N} \int_0^t \frac{\lambda_1-1}{\lambda_1} e^{-2(t-s) \eta_1} ds \\
&\le & \frac{2 (\lambda_1 +1)} {\lambda_1 N} \le \frac{4}{N}, 
\end{eqnarray*}
and
\begin{eqnarray*}
\lefteqn{\int_0^t \sum_y q(\tilde{x}_N(s), \tilde{x}_N(s) + y) (e^{\tilde{A}(t-s)}y)_2^2 ds}\\
& \le & 
\frac{4 (\lambda_1 +1) (\lambda_1-1)}{N (\lambda_1-\lambda_2)} \int_0^t \frac{\lambda_1-\lambda_2}{\lambda_1}e^{-2(t-s) \eta_2} ds
\le  \frac{4 \lambda_1 (\lambda_1-1)}{N (\lambda_1-\lambda_2)}.
\end{eqnarray*}
Given $\omega = \omega (N)$ as in the statement of the lemma, let $T_2$ be the infimum of times $t$ such that
$$\Big |\int_0^t e^{-(\lambda_1-1)(t-s)} d M_{N,1}(s) \Big |> 6 \sqrt{\frac{\omega}{N}},$$
or
$$\Big |\int_0^t e^{-(t-s) \frac{\lambda_1-\lambda_2}{\lambda_1}} d M_{N,2}(s) \Big | >
6 \sqrt{\frac{\omega \lambda_1 (\lambda_1-1)}{N(\lambda_1-\lambda_2)}}.$$
Then on the event $t < T_1 \land T_2$, for $N$ large enough,
\begin{eqnarray*}
g_N(t) & \le & g_N(0)  + 4 \lambda_1 \int_0^t  g_N(s)|\tx_1 (s)| ds +
\frac{4\lambda_1 (\lambda_1 - \lambda_2)}{\lambda_1-1}\int_0^t g_N(s) \tx_2 (s)ds\\
& + &
 \frac{4\lambda_1 (\lambda_1-\lambda_2)}{\lambda_1-1} \int_0^t g_N(s)^2 e^{-s \eta_2}ds + 6 e^{t\eta_2} \frac{\lambda_1-1}{\lambda_1-\lambda_2}\sqrt{\frac{\omega}{N}}.
\end{eqnarray*}
Let $T_3$ be the infimum of times $t$ such that
$$g_N(t) > 10 e^{32 (\lambda_1 \beta/\alpha+1)} e^{t \eta_2} \frac{\lambda_1-1}{\lambda_1-\lambda_2}\sqrt{\frac{\omega}{N}}.$$
On the event $t < T_1 \land T_2 \land T_3$, for $N$ large enough,
\begin{eqnarray*}
\frac{4\lambda_1 (\lambda_1-\lambda_2)}{\lambda_1-1} \int_0^t g_N(s)^2 e^{-s \eta_2}ds & \le & 400 \lambda_1^2 e^{64 (\lambda_1 \beta/\alpha +1)} \frac{(\lambda_1-1)}{(\lambda_1-\lambda_2)^2} \frac{\omega}{N}e^{t \eta_2}\\
& \le & e^{t \eta_2}\frac{\lambda_1-1}{\lambda_1-\lambda_2}\sqrt{\frac{\omega}{N}},
\end{eqnarray*}
since we have assumed that 
$N (\lambda_1-\lambda_2)^2/\omega \to \infty$. It follows that, for $N$ large enough, on the event $t < T_1 \land T_2 \land T_3$,
\begin{eqnarray*}
g_N(t) & \le & g_N(0)  + 4 \lambda_1 \int_0^t  g_N(s)|\tx_1(s)| ds +
\frac{4\lambda_1 (\lambda_1 - \lambda_2)}{\lambda_1-1}\int_0^t g_N(s) \tx_2(s)ds \\
&& {} +
7 e^{t \eta_2} \frac{\lambda_1-1}{\lambda_1-\lambda_2}\sqrt{\frac{\omega}{N}}\\
& \le & 8 e^{t \eta_2} \frac{\lambda_1-1}{\lambda_1-\lambda_2}\sqrt{\frac{\omega}{N}} + 4 \lambda_1 \int_0^t  g_N(s)|\tx_1(s)| ds \\
&& {} + \frac{4\lambda_1 (\lambda_1 - \lambda_2)}{\lambda_1-1}\int_0^t g_N(s) \tx_2(s)ds,
\end{eqnarray*}
provided
$g_N(0) = f_N(0) \le  \frac{\lambda_1-1}{\lambda_1-\lambda_2}\sqrt{\frac{\omega}{N}}$. By Gr\"onwall's lemma, on the event $t < T_1 \land T_2 \land T_3$, for $N$ large enough,
$$g_N(t) \le 8 e^{t\eta_2} \frac{\lambda_1-1}{\lambda_1-\lambda_2}\sqrt{\frac{\omega}{N}}
e^{4 \lambda_1\int_0^t |\tx_1(s)| ds + \frac{4\lambda_1 (\lambda_1-\lambda_2)}{\lambda_1-1} \int_0^t \tx_2(s) ds}.$$
Now, by~\eqref{relation}, 
\begin{eqnarray}
\tx_2 (t) & = & x_2 (t) = (x_1(t)/x_1(0))^{\lambda_2/\lambda_1} x_2(0) e^{-t\frac{\lambda_1 - \lambda_2}{\lambda_1}} \nonumber \\
& \le &
(x_2(0)/x_1(0))^{\lambda_2/\lambda_1} x_1 (t)^{\lambda_2/\lambda_1}e^{-t\frac{\lambda_1 - \lambda_2}{\lambda_1}} \label{x2-upper} \\
& = & \Big ( \frac{\beta}{\alpha} (1+o(1)) \Big )^{\lambda_2/\lambda_1} x_1 (t)^{\lambda_2/\lambda_1} e^{-t\frac{\lambda_1 - \lambda_2}{\lambda_1}} \le 4 (\beta/\alpha)\frac{\lambda_1-1}{\lambda_1} e^{-t\frac{\lambda_1 - \lambda_2}{\lambda_1}}, \nonumber
\end{eqnarray}
where we have also used the facts that $x_1 (0) + x_2(0) \le 2 (\lambda_1-1)/\lambda_1$ implies $x_1(t) + x_2(t) \le 2 (\lambda_1-1)/\lambda_1$ for all $t$, and that $(\lambda_1-\lambda_2) \log (\lambda_1-1) \to 0$ (and so $x_1(t)^{\lambda_2/\lambda_1} \le 3 (\lambda_1-1)/\lambda_1$ for $N$ large enough). Thus
$$\frac{4\lambda_1(\lambda_1-\lambda_2)}{\lambda_1-1} \int_0^t \tx_2(s) ds \le 16 \lambda_1 (\beta/\alpha).$$
Also, while $\tx_1(t) \ge 0$, using~\eqref{eq-diff-eq-eigen}, we have for $N$ large enough,
\begin{eqnarray*}
\frac{d \tx_1 (t)}{dt} \le - \tx_1 (t) \Big ( (\lambda_1-1) - \frac{(\lambda_1-\lambda_2)\lambda_1}{(\lambda_1-1)a} \tx_2(t) \Big ) \le  - \frac{\lambda_1-1}{2} \tx_1 (t),
\end{eqnarray*}
since $\tx_2(t) =x_2(t) \le 2 (\lambda_1-1)/\lambda_1$ for all $t$. On the other hand, when $\tx_1 (t)$ becomes negative, and while $|\tx_1 (t)| \le (\lambda_1-1)/4\lambda_1$, then for $N$ large enough,
\begin{eqnarray*}
\frac{d \tx_1 (t)}{dt} \ge
- \frac{\lambda_1-1}{2} \tx_1 (t) - \frac{(\lambda_1-\lambda_2)^2}{\lambda_1(\lambda_1-1)} \Big (\frac{\tx_2(t)}{a} \Big )^2,
\end{eqnarray*}
and so, using $x_2(t) \le 4(\beta/\alpha)\frac{\lambda_1-1}{\lambda_1} e^{-t(\lambda_1-\lambda_2)/\lambda_1}$,
$$\tx_1(t) \ge -\frac{32 \beta^2}{\alpha^2 a^2 \lambda_1^2} (\lambda_1-\lambda_2)^2 e^{-2t (\lambda_1-\lambda_2)/\lambda_1}.$$
Since $\tx_1(0) \le 4 (\lambda_1-1)/\lambda_1$ for $N$ large enough, we see that, if $N$ is large enough, then
$4 \lambda_1 \int_0^t |\tx_1(s)| ds \le 32$.

It follows that if $N$ is large enough, then on the event $t < T_1 \land T_2 \land T_3$,
$$g_N(t) \le 8 e^{t(\lambda_1-\lambda_2)/\lambda_1} \frac{\lambda_1-1}{\lambda_2-\lambda_1}\sqrt{\frac{\omega}{N}} e^{32 (\lambda_1\beta/\alpha + 1)}.$$
and so
$$f_N(t) \le 8 \frac{\lambda_1-1}{\lambda_1-\lambda_2}\sqrt{\frac{\omega}{N}} e^{32 (\lambda_1\beta/\alpha + 1)}.$$

Let $T_4 = \inf \{t: f_N(t) > 8 \frac{\lambda_1-1}{\lambda_2-\lambda_1}\sqrt{\frac{\omega}{N}} e^{32 (\lambda_1\beta/\alpha + 1)} \}$.
Let $t_0 = t_0 (N) = 
(\lambda_1-1)^{-1} e^{\omega/8}$.
By the above,
\begin{eqnarray*}
\P (T_4 \le t_0) \le \P (T_1 \land T_2 \land T_3 \le T_4 \land t_0) \le \P(T_1 \le t_0) + \P (T_2 \le T_1 \land t_0) + \P(T_3 \le T_4).
\end{eqnarray*}
By Lemma~\ref{lem.stoch-dom}, and by Lemma~\ref{lem.supercrit-log} with $\lambda = \lambda_1$ and $\mu = 1$, 
if $N$ is large enough, then $\P (T_1 \le t_0) \le 4 e^{-\omega/8}$. 
By Lemma~\ref{lem.mart-dev} applied to $(\tilde{x}_N(t))$, this time taking $\sigma = (\lambda_1-1)^{-1}$, $K_1 = 4/N$, $K_2 = 4 \lambda_1(\lambda_1-1)/N(\lambda_1-\lambda_2)$,
$\P (T_2 \le T_1 \land t_0) \le 8 e^{-\omega/8}$. 
Also, clearly, for $N$ large enough, $\P (T_3 \le T_4) = 0$. It follows that
$\P (T_4 \le t_0) \le 12 e^{-\omega/8}$, 
as required.
\end{proof}

\begin{lemma}
\label{lem-lt-approx-nc-acc}
Let $\omega_1 = \omega_1 (N) \to \infty$ as $N \to \infty$.
Let the assumptions of Lemma~\ref{lem-lt-approx-nc} on $X_N(0)$  and $x(0)$ be satisfied.
Assume further that
$$N (\lambda_1-\lambda_2)^2/\log \Big ( \frac{\lambda_1-1}{\lambda_1-\lambda_2}  \Big ) e^{\omega_1/2} \to \infty.$$
%
%
%
%
For $t \ge 0$, let
$f_N(t) = |\tx_{N,2}(t) - \tx_2(t)|$.
Set
$$
\delta_N(t) = \Big ( 2 (\lambda_1-\lambda_2)^{-1} \log^{1/2} \Big (\frac{\lambda_1-1}{\lambda_1-\lambda_2} \Big ) e^{\omega_1/8} x_2 (t) + 8 \sqrt{\frac{\lambda_1(\lambda_1-1)}{\lambda_1-\lambda_2}} \Big ) \sqrt{\frac{\omega_1}{N}} e^{32 (1+ \beta /\alpha)}.
$$
Then, for $N$ large enough,
$$\P \Big (\sup_{t \le \frac{\lambda_1}{\lambda_1-\lambda_2}e^{\omega_1/8}} f_N(t) > \delta_N(t) \Big )
\le 16 e^{-\omega_1/8}.$$
\end{lemma}
\begin{proof}
Let $\omega_2 = \log \Big (\frac{\lambda_1(\lambda_1-1)}{\lambda_1-\lambda_2} \Big ) + \omega_1$, and note that $N(\lambda_1-\lambda_2)^2/\omega_2 \to \infty$. 

As in the proof of Lemma~\ref{lem-lt-approx-nc}, with $\eta_2 = (\lambda_1-\lambda_2)/\lambda_1$,
\begin{eqnarray*}
\lefteqn{f_N(t) = |\tx_{N,2}(t)-\tx_2(t)|} \\
 & \le & |\tx_{N,2}(0) - \tx_2(0)|e^{-t\eta_2} +  \lambda_2 \int_0^t e^{-(t-s) \eta_2}  |\tx_{N,2}(s) - \tx_2(s)||\tx_{N,1}(s)|ds\\
& + & \lambda_2 \int_0^t e^{-(t-s) \eta_2}  |\tx_{N,1}(s) - \tx_1(s)|\tx_2(s)ds + \Big |\int_0^t e^{-(t-s) \eta_2} d M_{N,2}(s)\Big |\\
& + & \frac{\lambda_2}{\lambda_1 a} \frac{\lambda_1 - \lambda_2}{\lambda_1-1} \int_0^t e^{-(t-s) \eta_2} |\tx_{N,2}(s)-\tx_2(s)| (\tx_{N,2}(s) + \tx_2(s)) ds.
\end{eqnarray*}
Let $g_N(t) = f_N(t) e^{t \eta_2}$.
Let 
$T_1$ be the infimum of times $t$ such that
$$|\tilde{x}_{N,1} (t) - \tilde{x}_1 (t)| > 8 \sqrt{\frac{\omega_2}{N}} e^{32 (\lambda_1\beta /\alpha + 1 )}.$$
Let $T_2$ be the infimum of times $t$ such that
$$\Big |\int_0^t e^{-(t-s) \eta_2} d M_{N,2}(s) \Big | >
6 \sqrt{\frac{\omega_1 \lambda_1 (\lambda_1-1)}{N(\lambda_1-\lambda_2)}}.$$
Then for $t < T_1 \land T_2$, if $N$ is large enough,
\begin{eqnarray*}
g_N(t) & \le & g_N(0) +  \lambda_2 \int_0^t g_N(s) |\tx_{N,1}(s)|ds\\
& + & 8 \lambda_2 \sqrt{\frac{\omega_2}{N}} e^{32 (\lambda_1 \beta /\alpha + 1 )} \int_0^t e^{s \eta_2}\tx_2(s)ds + 6 e^{t \eta_2} \sqrt{\frac{\omega_1 \lambda_1 (\lambda_1-1)}{N(\lambda_1-\lambda_2)}}\\
& + &  \frac{4 (\lambda_1 - \lambda_2)}{\lambda_1-1} \int_0^t g_N(s)\tx_2(s) ds + \frac{2 (\lambda_1 - \lambda_2)}{\lambda_1-1} \int_0^t g_N(s)^2 e^{-s \eta_2} ds.
\end{eqnarray*}
Let $T_3$ be the infimum of times $t$ such that
\begin{eqnarray*}
g_N(t) &>& 10 \sqrt{\frac{\lambda_1(\lambda_1-1)}{\lambda_1-\lambda_2}}\sqrt{\frac{\omega_1}{N}} e^{32 (1+ \beta /\alpha)} e^{t \eta_2} \\
&& {} + 4 \sqrt{\frac{\omega_1}{N(\lambda_1-\lambda_2)^2}} \log^{1/2} \Big ( \frac{\lambda_1-1}{\lambda_1-\lambda_2}  \Big )
e^{32 (1+ \beta /\alpha)} e^{\omega_1/8} x_2 (t) e^{t \eta_2}.
\end{eqnarray*}
Now, by~\eqref{x2-upper},
\begin{eqnarray*}
x_2(t)  \le 4 (\beta/\alpha) \frac{\lambda_1 - 1}{\lambda_1} e^{-t \eta_2},
\end{eqnarray*}
and, by~\eqref{relation} and~\eqref{x2-upper}, the fact that, for all $t$, $x_1 (t) + x_2 (t) \ge \frac12 \frac{\lambda_1-1}{\lambda_1}$ if $N$ is large enough, and the fact that $x_1(t)/x_2(t)$ is increasing, also
\begin{eqnarray}
\label{x2-lower}
x_2 (t) \ge \frac14 \frac{\beta}{\alpha + \beta} \frac{\lambda_1-1}{\lambda_1} e^{-t \eta_2}.
\end{eqnarray}
Then for $t < T_3$, if $N$ is large enough,
\begin{eqnarray*}
\lefteqn{ \frac{2 (\lambda_1 - \lambda_2)}{\lambda_1-1}\int_0^t g_N(s)^2 e^{-s \eta_2} ds} \\
  & \le & \frac{200\omega_1 \lambda^2_1}{(\lambda_1-\lambda_2)N}
 e^{64 (1 + \beta /\alpha)} e^{t \eta_2} + \frac{1024(\beta/\alpha)^2 (\lambda_1-1) \omega_1}{N(\lambda_1-\lambda_2)^2}
 e^{64 (1 + \beta /\alpha)} \log \Big ( \frac{\lambda_1-1}{\lambda_1-\lambda_2}  \Big ) e^{\omega_1/4}\\
 & \le & e^{t \eta_2} \sqrt{\frac{\omega_1 \lambda_1(\lambda_1-1)}{N(\lambda_1-\lambda_2)}} + \sqrt{\frac{\omega_1}{N(\lambda_1-\lambda_2)^2}} \log^{1/2} \Big ( \frac{\lambda_1-1}{\lambda_1-\lambda_2}  \Big ) e^{\omega_1/8} x_2 (t) e^{t \eta_2},
\end{eqnarray*}
since our assumptions imply that $N (\lambda_1 -1) (\lambda_1-\lambda_2)/\omega_1 \to \infty$ and $N (\lambda_1-\lambda_2)^2/\log \Big ( \frac{\lambda_1-1}{\lambda_1-\lambda_2} \Big ) e^{\omega_1/4}\to \infty$. So, for $t < T_1 \land T_2 \land T_3$, if $N$ is large enough and $g_N(0) \le \sqrt{\frac{\omega_1 \lambda_1(\lambda_1-1)}{N(\lambda_1-\lambda_2)}}$,
\begin{eqnarray*}
g_N(t) & \le &   \lambda_2 \int_0^t g_N(s) |\tx_1(s)|ds +  \frac{4 (\lambda_1 - \lambda_2)}{\lambda_1-1} \int_0^t g_N(s)\tx_2(s) ds\\
& + & 32 \sqrt{\frac{\omega_2}{N}} e^{32 (\lambda_1 \beta / \alpha + 1 )} \frac{\beta}{\alpha} (\lambda_1-1)t  +
8 e^{t \eta_2} \sqrt{\frac{\omega_1 \lambda_1 (\lambda_1-1)}{N(\lambda_1-\lambda_2)}}\\
& + &  8\lambda_2 \sqrt{\frac{\omega_2}{N}} e^{32 (\lambda_1 \beta /\alpha + 1 )} \int_0^t g_N(s) ds
 + \sqrt{\frac{\omega_1}{N(\lambda_1-\lambda_2)^2}} \log^{1/2} \Big ( \frac{\lambda_1-1}{\lambda_1-\lambda_2}  \Big ) e^{\omega_1/8} x_2 (t) e^{t \eta_2}.
\end{eqnarray*}
Let $t_0 = t_0 (N) = \lambda_1 (\lambda_1-\lambda_2)^{-1}e^{\omega_1/8}$.
By Gr\"onwall's lemma, for all $t \le t_0(N)$, on the event $t < T_1 \land T_2 \land T_3$,
\begin{eqnarray*}
g_N(t) & \le & \Big (32 \sqrt{\frac{\omega_2}{N}} e^{32 (\lambda_1 \beta /\alpha + 1 )} \frac{\beta}{\alpha} (\lambda_1-1)t  + 8 e^{t \eta_2} \sqrt{\frac{\omega_1 \lambda_1 (\lambda_1-1)}{N(\lambda_1-\lambda_2)}} \\
 &+ & \sqrt{\frac{\omega_1}{N(\lambda_1-\lambda_2)^2}} \log^{1/2} \Big ( \frac{\lambda_1-1}{\lambda_1-\lambda_2}  \Big ) e^{\omega_1/8} x_2 (t) e^{t \eta_2}\Big ) e^{H_N(t)},
\end{eqnarray*}
where
\begin{eqnarray*}
H_N(t) & = & \lambda_2 \int_0^t |\tx_1(s)|ds + \frac{4 (\lambda_1 - \lambda_2)}{\lambda_1-1} \int_0^t \tx_2 (s) ds +
8 \lambda_2 t\sqrt{\frac{\omega_2}{N}} e^{32 (\lambda_1 \beta /\alpha + 1 )} \\
& \le & 16 \Big (1 + \frac{\beta}{\alpha} \Big ) + \frac{8 \lambda_1 \lambda_2}{\sqrt{N} (\lambda_1 - \lambda_2)} \sqrt{\omega_2} e^{\omega_1/8}
e^{32 (\lambda_1 \beta /\alpha + 1 )}\\
& \le  & 16 \Big (1 + \frac{\beta}{\alpha} \Big ) + \frac{1}{\sqrt{N} (\lambda_1 - \lambda_2)} \sqrt{\log \Big ( \frac{\lambda_1-1}{\lambda_1-\lambda_2}  \Big )} e^{\omega_1/4}\\
& \le & 32 \Big (1 + \frac{\beta}{\alpha} \Big ),
\end{eqnarray*}
for $N$ large enough, since
$$N (\lambda_1-\lambda_2)^2/\log \Big ( \frac{\lambda_1-1}{\lambda_1-\lambda_2}  \Big ) e^{\omega_1/2} \to \infty.$$
It follows that, for $t \le t_0(N)$, on the event $t < T_1 \land T_2 \land T_3$, if $N$ is large enough, then
\begin{eqnarray*}
f_N(t) & \le & \Big (32 \sqrt{\frac{\omega_2}{N}} e^{32 (\lambda_1 \beta / \alpha + 1 )} \frac{\beta}{\alpha} (\lambda_1-1)t e^{-t\eta_2} + 8 \sqrt{\frac{\omega_1 \lambda_1 (\lambda_1-1)}{N(\lambda_1-\lambda_2)}} \\
& + &  \sqrt{\frac{\omega_1}{N(\lambda_1-\lambda_2)^2}} \log^{1/2} \Big ( \frac{\lambda_1-1}{\lambda_1-\lambda_2}  \Big ) e^{\omega_1/8} x_2 (t) \Big ) e^{32 \Big (1 + \frac{\beta}{\alpha} \Big )}.
\end{eqnarray*}
Now, using~\eqref{x2-upper} and~\eqref{x2-lower},
it thus follows that, for $t \le t_0(N)$, on the event $t < T_1 \land T_2 \land T_3$, for $N$ large enough,
\begin{eqnarray*}
f_N(t) &\le&  \Big (128 \frac{\alpha + \beta}{\alpha}\sqrt{\frac{\log (\lambda_1 (\lambda_1-1)/(\lambda_1 - \lambda_2) ) + \omega_1}{N}} e^{32 (\lambda_1 \beta /\alpha + 1 )}
\lambda_1 x_2 (t)t \\
&& \qquad + 8 \sqrt{\frac{\omega_1 \lambda_1 (\lambda_1-1)}{N(\lambda_1-\lambda_2)}}
 + \sqrt{\frac{\omega_1}{N(\lambda_1-\lambda_2)^2}} \log^{1/2} \Big ( \frac{\lambda_1-1}{\lambda_1-\lambda_2}  \Big ) e^{\omega_1/8} x_2 (t) \Big ) e^{32 \Big (1 + \frac{\beta}{\alpha} \Big )},
\end{eqnarray*}
Now, for all $t \le t_0(N)$ such that $t < T_1 \land T_2 \land T_3$, for $N$ large enough,
\begin{eqnarray*}
f_N(t) \le  \Big ( 2 (\lambda_1-\lambda_2)^{-1} \log^{1/2} \Big (\frac{\lambda_1-1}{\lambda_1-\lambda_2} \Big ) e^{\omega_1/8} x_2 (t) + 8 \sqrt{\frac{\lambda_1(\lambda_1-1)}{\lambda_1-\lambda_2}} \Big ) \sqrt{\frac{\omega_1}{N}} e^{32 \Big (1 + \frac{\beta}{\alpha} \Big )},
\end{eqnarray*}
Let $T_4$ be the infimum of times $t$ such that
$$f_N(t) >  \Big ( 2 (\lambda_1-\lambda_2)^{-1} \log^{1/2} \Big (\frac{\lambda_1-1}{\lambda_1-\lambda_2} \Big ) e^{\omega_1/8} x_2 (t) + 8 \sqrt{\frac{\lambda_1(\lambda_1-1)}{\lambda_1-\lambda_2}} \Big ) \sqrt{\frac{\omega_1}{N}} e^{32 \Big (1 + \frac{\beta}{\alpha} \Big )}.$$
By the above, and as in the proof of Lemma~\ref{lem-lt-approx-nc},
\begin{eqnarray*}
\P (T_4 \le t_0) \le \P (T_1 \land T_2 \land T_3 \le T_4 \land t_0) \le \P(T_1 \le t_0) + \P (T_2 \le T_1 \land t_0) + \P(T_3 \le T_4).
\end{eqnarray*}
By Lemma~\ref{lem-lt-approx-nc}, 
$\P (T_1 \le t_0) \le 12 e^{-\omega_2/8}$. 
By Lemma~\ref{lem.mart-dev} applied to $(\tilde{x}_N(t))$, this time taking $\sigma = (\lambda_1-\lambda_2)^{-1}$, $K_2 = 4 \lambda_1 (\lambda_1-1)/N(\lambda_1-\lambda_2)$, and $\eta = \eta_2$,
$\P (T_2 \le T_1 \land t_0) \le 4 e^{-\omega_1/8}$. 
Also, clearly, for $N$ large enough, $\P (T_3 \le T_4) = 0$. It follows that
$\P (T_4 \le t_0) \le 16 e^{-\omega_1/8}$, 
as required.
\end{proof}

{\bf Proof of Theorem~\ref{thm.extinction-nc}.}\,
Let
$$t_1 = t_1 (N) = (\lambda_1-1)^{-1/2} (\lambda_1-\lambda_2)^{-1/2} = \frac{1}{\lambda_1-\lambda_2} \sqrt{\frac{\lambda_1-\lambda_2}{\lambda_1-1}} = \frac{1}{\lambda_1-1} \sqrt{\frac{\lambda_1-1}{\lambda_1-\lambda_2}}.$$
By Lemma~\ref{lem.ratio} with $\psi = (\lambda_1-1)^{1/2}(\lambda_1-\lambda_2)^{-1/2}$, for $N$ large enough,
$$\P \Big (  \Big | \frac{X_{N,1}(t_1)}{X_{N,2}(t_1)} - \frac{X_{N,1}(0)}{X_{N,2}(0)} \Big | > 2 \Big ( \frac{\lambda_1-\lambda_2}{\lambda_1-1} \Big )^{1/8} \Big ) \le 4 e^{-\sqrt{N}(\lambda_2-1)} + e^{-\Big ( \frac{\lambda_1-1}{\lambda_1-\lambda_2} \Big )^{1/8}}.$$

Assume that $x_{N,1} (0) + x_{N,2}(0) > (\lambda_1-1)/\lambda_1$.
Let $x = \min \{x_{N,1}(0) + x_{N,2}(0), N^{-1}\lfloor N (\lambda_2-1) (N (\lambda_2-1)^2)^{1/8} \rfloor \}$, and let $T$ be the infimum of times $t$ such that $x_{N,1} (t) + x_{N,2}(t) =x$.
 Lemma 4.1 in Brightwell, House and Luczak (2017) is still valid for a supercritical stochastic SIS logistic epidemic, so by that lemma, combined with Lemma~\ref{lem.stoch-dom} in the present paper and Markov's inequality, for $N$ large enough,
$$\P (x_{N,1}(t_1/4) + x_{N,2} (t_1/4) > x ) \le 2(N (\lambda_2-1)^2)^{-1/8},$$
and so $T \le t_1/4$ with probability at least $1-2(N (\lambda_2-1)^2)^{-1/8}$.

Let $y(t)$ solve equation~\eqref{ode.logistic} with $\lambda = \lambda_1$ and $\mu = 1$, $y(T) = x$, and let $z(t)$ solve equation~\eqref{ode.logistic} with $\lambda = \lambda_2$ and $\mu = 1$, and $z(T) = x$. Let $Y_N(t)$ and $Z_N(t)$ be the corresponding stochastic SIS logistic epidemics satisfying $Y_N(T) = N x$ and $Z_N(T) = N x$ respectively.
It is easily seen from~\eqref{eq.log-sol2} that, if $N$ is sufficiently large and $T \le t_1/4$, then $y(t_1/2) \le 2 (\lambda_1-1)/\lambda_1$ and $z(t_1/2) \le 2 (\lambda_2-1)/\lambda_2$.  Furthermore, using (\ref{eq.log-sol2}) over the time-interval $[t_1/2,t_1]$, we see that in that case
$$|y(t_1) - (\lambda_1-1)/\lambda_1 | \le  \frac{\lambda_1-1}{\lambda_1} e^{-\frac12 \sqrt{\frac{\lambda_1-1}{\lambda_1-\lambda_2}}},
$$ and
$$|z(t_1) - (\lambda_2-1)/\lambda_2| \le  \frac{\lambda_2-1}{\lambda_2} e^{-(\lambda_2-1)t_1/2} \le  \frac{\lambda_2 -1}{\lambda_2} e^{-\frac14 \sqrt{\frac{\lambda_1-1}{\lambda_1-\lambda_2}}}.$$
We will now apply Lemma~\ref{lem.supercrit-log-1} twice, both starting at time $T$, once with $\lambda = \lambda_1$, $\mu = 1$, ending at time $\tau_1 = \inf \{t \ge T: y (t) \le 2(\lambda_1-1)/\lambda_1 \}$, and the second time with $\lambda = \lambda_2$, $\mu = 1$, ending at time $\tau_2 = \inf \{t \ge T: z (t) \le 2(\lambda_2-1)/\lambda_2 \}$. We further apply Lemma~\ref{lem.supercrit-log} twice, once to $Y_N(t)$, starting at time $\tau_1$, with $\lambda = \lambda_1$, $\mu = 1$ and $\omega = \sqrt{N (\lambda_1-1)^2}$, and once to $Z_N(t)$, starting at time $\tau_2$, with $\lambda = \lambda_2$, $\mu =1$ and $\omega = \sqrt{N (\lambda_2-1)^2}$. Additionally applying  Lemma~\ref{lem.stoch-dom}, we see that, for $N$ sufficiently large,
\begin{eqnarray*}
\P \Big ( \Big |x_{N,1}(t_1) + x_{N,2}(t_1) - \frac{\lambda_1-1}{\lambda_1} \Big | > 6 e^4 \frac{ (N (\lambda_1-1)^2)^{1/4}\sqrt{\lambda_1 +1}}{\sqrt{N\lambda_2}} + \lambda_1 - \lambda_2 \Big ) \\
\le 12 e^{-(N (\lambda_2 -1)^2)^{1/8}/8} + 2(N (\lambda_2-1)^2)^{-1/8}
\le 3 (N (\lambda_2-1)^2)^{-1/8}.
\end{eqnarray*}
In particular, we see that with probability $1-\delta_N$ event ${\mathcal E}_N$ holds that $x_{N,1}(t_1) + x_{N,2}(t_1) \le 2 (\lambda_1-1)/\lambda_1$, $x_{N,1}(t_1)/x_{N,2}(t_2) = \alpha/\beta + \eps_N$, where $\delta_N, \eps_N \to 0$ as $N \to \infty$.

In the case when $x_{N,1} (0) + x_{N,2}(0) < (\lambda_1-1)/\lambda_1$ (this is only relevant when $\lambda_1$ is bounded away from 1), we can skip the first two phases and only use Lemma~\ref{lem.supercrit-log}. We omit the details.

%
%

Let $\omega_1 = \omega_1 (N) \to \infty$ be such that $\omega_1 \le \log \Big ( \frac{\lambda_1-1}{\lambda_1-\lambda_2} \Big )$ and let $\omega_2 = 16 \log \Big ( \frac{\lambda_1-1}{\lambda_1-\lambda_2} \Big )$. Let $t_2 = t_2 (N) = t_1 (N) + \frac{\lambda_1}{\lambda_1-\lambda_2}e^{\omega_1/8}$, and note that $t_2 - t_1 \le (\lambda_1 -1)^{-1} e^{\omega_2/8}$.


Consider solution $x(t) = (x_1(t), x_2(t))^T$ to~\eqref{eq.det-comp} subject to condition $x_1 (t_1) = N^{-1}X_{N,1}(t_1)$, $x_2(t_1)= N^{-1}X_{N,2}(t_1))^T$. Let also $\tx (t) = (\tx_1(t), \tx_2(t))^T$ be the corresponding solution to~\eqref{eq-diff-eq-eigen}.
By~\eqref{x2-upper},
\begin{eqnarray*}
x_2(t_2)  \le 4 (\beta/\alpha) \frac{\lambda_1 - 1}{\lambda_1} e^{-(t_2-t_1)(\lambda_1-\lambda_2)/\lambda_1},
\end{eqnarray*}
and by~\eqref{x2-lower}
\begin{eqnarray*}
x_2 (t_2) \ge \frac14 \frac{\beta}{\alpha + \beta} (\lambda_1-1) e^{-(t_2-t_1) \frac{\lambda_1-\lambda_2}{\lambda_1}}.
\end{eqnarray*}
for $N$ large enough. Note that $\omega_1$ can be chosen in such a way that $(\lambda_1-\lambda_2)^{-1} x_2 (t_2) \to 0$: for instance, we choose $\omega_1$ satisfying $e^{\omega_1/8} = \log \Big (  \frac{\lambda_1-1}{\lambda_1-\lambda_2} \Big ) + \phi$ for a suitable $\phi = \phi (N) \to \infty$ such that $\phi \le \log \Big (  \frac{\lambda_1-1}{\lambda_1-\lambda_2} \Big )$. Since then $x_2 (t_2) \ge \frac{1}{4} \frac{\beta}{\alpha + \beta} (\lambda_1-\lambda_2)^2 (\lambda_1-1)^{-1}$, we further have $Nx_2(t_2) (\lambda_1-\lambda_2) \to \infty$.


%
%
Note that, since $N (\lambda_1-\lambda_2)^3 (\lambda_1-1)^{-1} \to \infty$, conditions of Lemmas~\ref{lem-lt-approx-nc} and~\ref{lem-lt-approx-nc-acc} are satisfied.
By Lemma~\ref{lem-lt-approx-nc} with $\omega = \omega_2$ and by Lemma~\ref{lem-lt-approx-nc-acc} with the value of $\omega_1$ above, with
$X_{N,1}(t_1), X_{N,2}(t_1)$ as initial values, with probability at least $1-16e^{-\omega_1/8}-12e^{-\omega_2/8}$, 
the event ${\mathcal E}(t_2)$ holds that
$$|\tilde{x}_{N,1}(t_2) - \tilde{x}_1(t_2)| \le 8 \sqrt{\frac{\omega_2}{N}}e^{32 (\lambda_1\beta/\alpha +1)},$$
and
\begin{eqnarray*}
\lefteqn{|\tilde{x}_{N,2}(t_2) - \tilde{x}_2(t_2)|} \\
&\le& \Big [2(\lambda_1-\lambda_2)^{-1} \log^{1/2} \Big ( \frac{\lambda_1-1}{\lambda_1-\lambda_2} \Big ) e^{\omega_1/8} x_2 (t_2) + 8 \sqrt{\frac{\lambda_1(\lambda_1-1)}{\lambda_1-\lambda_2}} \Big ] \sqrt{\frac{\omega_1}{N}}e^{32(1+\beta/\alpha)}.
\end{eqnarray*}
%
Hence also, for $N$ large enough, on ${\mathcal E}(t_2)$,
\begin{eqnarray*}
\lefteqn{|x_{N,1}(t_2) - (\lambda_1-1)/\lambda_1|} \\
&\le  &\Big [2(\lambda_1-\lambda_2)^{-1}  \log^{1/2} \Big ( \frac{\lambda_1-1}{\lambda_1-\lambda_2} \Big ) e^{\omega_1/8} x_2 (t_2) + 10 \sqrt{\frac{\lambda_1(\lambda_1-1)}{\lambda_1-\lambda_2}} \Big ] \sqrt{\frac{\omega_1}{N}}e^{32(1+\beta/\alpha)}\\
 &+ &8 \frac{\beta}{\alpha} \frac{\lambda_1-1}{\lambda_1} e^{-(t_2-t_1)(\lambda_1-\lambda_2)/\lambda_1}.
 \end{eqnarray*}
Note that
$$N (\lambda_1-\lambda_2)^3 (\lambda_1-1)^{-1}/\log \log \big(N(\lambda_1-\lambda_2)^2\big) \to \infty$$
implies that
$$N (\lambda_1-\lambda_2)^3 (\lambda_1-1)^{-1}/\log \log \Big ( \frac{\lambda_1-1}{\lambda_1-\lambda_2} \Big ) \to \infty,$$
and so we can choose $\phi$ so that
$$N (\lambda_1-\lambda_2)^3 (\lambda_1-1)^{-1}/\omega_1 e^{2\phi} \to \infty.$$
With this choice of $\phi$, it follows that
$$(\lambda_1-1) e^{-(t_2-t_1) \frac{\lambda_1-\lambda_2}{\lambda_1}}  \gg \sqrt{\frac{\lambda_1-1}{\lambda_1-\lambda_2}}\sqrt{\frac{\omega_1}{N}},$$
and so, on the event ${\mathcal E}(t_2)$, both $x_{N,1}(t_2)$ and $x_{N,2}(t_2)$ are concentrated around $(\lambda_1-1)/\lambda_1$ and $x_2(t_2)$ respectively, $x_{N,1}(t_2)$ with error of size $o(\lambda_1-\lambda_2)$ and $x_{N,2}(t_2)$ with error of size $o(x_2(t_2)) = o(\lambda_1 - \lambda_2)$. Also, $\P ({\mathcal E}(t_2) ) \to 1$ as $N \to \infty$.

Let
$$t_3 = t_2 + \frac{10\lambda_1}{\lambda_1-\lambda_2} \log \big( N(\lambda_1-\lambda_2)^2\big) \le t_1 + (\lambda_1-1)^{-1} e^{\omega_3/8},$$
where $$\omega_3 = 32 \log \Big ( \frac{\lambda_1-1}{\lambda_1-\lambda_2} \Big ) + 32 \log \log \big( N(\lambda_1-\lambda_2)^2 \big).$$
Also,
$$t_3 - t_1 \le \frac{\lambda_1}{\lambda_1-\lambda_2} e^{\omega_4/8},$$
where
$$\omega_4 = 16 \log \log \Big ( \frac{\lambda_1-1}{\lambda_1-\lambda_2}  \Big ) + 2 \log \log \big( N(\lambda_1-\lambda_2)^2 \big).$$
Note that $\omega_3$ and $\omega_4$ satisfy conditions of Lemmas ~\ref{lem-lt-approx-nc} and~\ref{lem-lt-approx-nc-acc} respectively, so we can apply these Lemmas on the interval $[t_1,t_3]$.

For $t \ge t_2$, let $\tilde{\mathcal E}(t)$ be the event that, for all $s \in [t_2,t]$, 
%
\begin{eqnarray*}
\lefteqn{|x_{N,1}(s) - (\lambda_1-1)/\lambda_1|} \\
&\le  &\Big [4(\lambda_1-\lambda_2)^{-1}  \Big ( \frac{\lambda_1-1}{\lambda_1-\lambda_2} \Big )^{1/4} x_2 (t_2) + 20 \sqrt{\frac{\lambda_1(\lambda_1-1)}{\lambda_1-\lambda_2}} \Big ] \sqrt{\frac{\omega_3}{N}}e^{32(1+\beta/\alpha)}\\
 &+ & 16 \frac{\beta}{\alpha} \frac{\lambda_1-1}{\lambda_1} e^{-(t_2-t_1)(\lambda_1-\lambda_2)/\lambda_1}.
 \end{eqnarray*}
and $x_{N,2}(s) \le 2 x_{N,2}(t_2)$.
Note that on the event $\tilde{\mathcal E}_t$, for all $t_2 \le s \le t$, $x_{N,1} (s)$ is concentrated around $(\lambda_1-1)/\lambda_1$ with an $o(\lambda_1-\lambda_2)$ error.

By Lemmas ~\ref{lem-lt-approx-nc} and~\ref{lem-lt-approx-nc-acc}, $\P (\tilde{\mathcal E}(t_3)) \to 1$ as $N \to \infty$.

On the event ${\mathcal E}(t_2)$, using a standard argument similar to the proof of Lemma~\ref{lem-final-phase} and the proof of Lemma 2.1 in Brightwell, House, and Luczak (2018), we couple the subsequent evolution of $X_{N,2}(t)$ with two linear birth-and-death chains, each with birth rate $\frac{\lambda_2}{\lambda_1} + o(\lambda_1-\lambda_2)$, and death rate $1$, so as to sandwich it between two such chains.
The next event after time $t \ge t_2$ in each of the three chains can be coupled together, as long as event $\tilde{\mathcal E}(t)$ holds.
%
Extinction happens by time $t_3$ with high probability, since
%
%
the length of the final phase is, with high probability,
\begin{eqnarray*}
\frac{\lambda_1}{\lambda_1-\lambda_2} \Big (\log N + \log x_2 (t_2) + \log \frac{\lambda_1-\lambda_2}{\lambda_1} + o(1) + G_N \Big )\\
= \frac{\lambda_1}{\lambda_1-\lambda_2} \Big (\log N (\lambda_1-\lambda_2)^2 - \phi (N) + O(1) \Big ),
\end{eqnarray*}
where $G_N$ converges to a Gumbel random variable $G$ as $N \to \infty$.

The length of the `fluid-limit' phase can be expressed as
\begin{eqnarray*}
\frac{\lambda_2}{\lambda_1-\lambda_2} \log (x_1(t_2)/x_1 (t_1)) - \frac{\lambda_1}{\lambda_1-\lambda_2}  \log (x_2(t_2)/x_2 (t_1)),
\end{eqnarray*}
and the length of the first phase is $t_1 = (\lambda_1-\lambda_2)^{-1} \sqrt{(\lambda_1-\lambda_2)/(\lambda_1-1)} = o((\lambda_1-\lambda_2)^{-1})$.

Hence, using the fact that $(\lambda_1-\lambda_2) \log (\lambda_1-1) \to 0$, the total time to extinction is, with high probability,
\begin{eqnarray*}
\frac{\lambda_1}{\lambda_1-\lambda_2} \Big ( \log \Big (\frac{N(\lambda_1-1) (\lambda_1-\lambda_2) \beta}{\lambda_1^2 \alpha}  + o(1) + G_N\Big ) \Big),
\end{eqnarray*}
thus proving Theorem~\ref{thm.extinction-nc}.

\subsection{Relaxing the assumption on separation from criticality}

\label{sub.discussion}

As stated above, we believe Theorem~\ref{thm.extinction-nc} is in fact valid under the weaker condition $N (\lambda_1-\lambda_2)^2 \to \infty$ (still assuming $\mu_1 = \mu_2 = 1$ and $(\lambda_1-\lambda_2)(\lambda_1 - 1)^{-1} \to \infty$).
Here is a sketch of how one might go about proving such an extension. The differential equation approximation phase can be split into a number of subphases, each corresponding to a refined version of Lemma~\ref{lem-lt-approx-nc-acc} with a smaller value of $x_{N,2}(0)$ and thus a smaller bound on the quadratic variation of the martingale term. Roughly speaking the first subphase would have $x_{N,2}(0)$ of order $\lambda_1-1$ and the martingale quadratic variation $N^{-1/2} (\lambda_1-\lambda_2)^{-1/2} (\lambda_1-1)^{1/2}$. The first subphase would last until $x_{N,2} (t)$ is of size about $N^{-1/2}  (\lambda_1-\lambda_2)^{-1/2} (\lambda_1-1)^{1/2} \omega^{3/4}$, for a suitable $\omega (N) \to \infty$, and would thus take time just slightly less than $(\lambda_1-\lambda_2)^{-1} \log \sqrt{N (\lambda_1-1) (\lambda_1-\lambda_2)}$.
The second subphase would have $x_{N,2}(0)$ of order $N^{-1/2} (\lambda_1-\lambda_2)^{-1/2}(\lambda_1-1)^{1/2} \omega^{3/4}$ and the martingale quadratic variation $N^{-3/4} (\lambda_1-\lambda_2)^{-3/4} (\lambda_1-1)^{1/4}  \omega^{3/8}$. It would last until $x_{N,2} (t)$ is of size about $N^{-3/4} (\lambda_1-\lambda_2)^{-3/4} (\lambda_1-1)^{1/4}  \omega^{15/16}$, and would thus take time just slightly less than $(\lambda_1-\lambda_2)^{-1} \log (N (\lambda_1-1) (\lambda_1-\lambda_2))^{1/4}$. The third subphase would have $x_{N,2}(0)$ of order $N^{-3/4} (\lambda_1-\lambda_2)^{-3/4} (\lambda_1-1)^{1/4}  \omega^{15/16}$ and the martingale quadratic variation $N^{-7/8} (\lambda_1-\lambda_2)^{-7/8}(\lambda_1-1)^{1/8}  \omega^{15/32}$. It would last until $x_{N,2} (t)$ is of size about $N^{-7/8} (\lambda_1-\lambda_2)^{-7/8}(\lambda_1-1)^{1/8}  \omega^{63/64}$, for a suitable $\omega (N) \to \infty$, and would thus take time just slightly less than $(\lambda_1-\lambda_2)^{-1} \log (N (\lambda_1-1) (\lambda_1-\lambda_2))^{1/8}$. And, in principle, one should be able to carry on this process. The phases would be joined together using the end value of $x_{N,2}(t)$ from the previous phase as initial condition for the differential equation in the next phase, and the various deterministic solutions with different random initial conditions would become closer and closer together over time.

As many phases would be used as needed to `reach' $x_{N,2}(t)$ of order $o(\lambda_1-\lambda_2)$, while keeping the deviation smaller than the mean. Since the `limiting' quadratic variation in the above process is $N^{-1} (\lambda_1-\lambda_2)^{-1}$, this should in principle be possible as long as $N (\lambda_1-\lambda_2)^2 \to \infty$, and as  $N (\lambda_1-\lambda_2)^2$ tends to infinity more and more slowly, the time spent in the differential equation phase becomes closer and closer to $(\lambda_1-\lambda_2)^{-1} \log (N(\lambda_1-1)(\lambda_1-\lambda_2))$.

After the condition  $N (\lambda_1-\lambda_2)^2 \to \infty$ fails, one can still carry out the differential equation phase from the time when $x_{N,2}$ if of the order $\lambda_1-1$ through the various phases until it is of the order about $N^{-1} (\lambda_1-\lambda_2)^{-1}$, which takes time of the order about $(\lambda_1-\lambda_2)^{-1} \log (N(\lambda_1-1)(\lambda_1-\lambda_2))$. After that, the fluctuations dominate, and the remaining time is about $(\lambda_1-\lambda_2)^{-2}$ steps, translating to a time of order $N^{-1}(\lambda_1-\lambda_2)^{-2} = o((\lambda_1-\lambda_2)^{-1})$.

A differential equation approximation phase is possible as long as the initial quadratic variation $N^{-1/2} (\lambda_1-\lambda_2)^{-1/2} (\lambda_1- 1)^{1/2}$ is $o(\lambda_1-1)$, which is as long as $N(\lambda_1-1)(\lambda_1-\lambda_2) \to \infty$.

This would join up our result nicely with that of Kogan et al. (2014), showing that when the two basic reproductive ratios are equal, then the time to extinction is of order $N$.

We hope all the details above can be filled in to yield a complete proof, but we leave this till the next paper.

\section{Acknowledgement}

The authors are grateful to Graham Brightwell for helpful comments that inspired Lemma~\ref{lem-lt-approx-nc-acc}, and as a result helped us achieve an improved understanding of near-critical phenomena in this model.







\end{document}